\documentclass[a4paper,12pt,times]{article}

\usepackage{setspace}
\usepackage{amstext} 
\usepackage[utf8]{inputenc}
\usepackage{amsmath, amsthm, amsfonts, amssymb}
\usepackage{mathrsfs}
\usepackage{hyperref}
\usepackage{enumitem}
\usepackage{tikz-cd}
\usepackage[margin=1in]{geometry}
\usepackage{CJKutf8}
\usepackage{url}
\usepackage{ytableau}
\usepackage{caption}
\usepackage{tikz}
\usepackage{stmaryrd}
\usepackage{mathabx}
\usepackage{makecell}
\makeatletter
\def\makebb#1{\expandafter\def\csname bb#1\endcsname{{\mathbb{#1}}}\ignorespaces}
\def\makerm#1{\expandafter\def\csname rm#1\endcsname{{\rm #1}}\ignorespaces}
\def\makebf#1{\expandafter\def\csname bf#1\endcsname{{\bf #1}}\ignorespaces}
\def\makegr#1{\expandafter\def\csname gr#1\endcsname{{\mathfrak{#1}}}\ignorespaces}
\def\makescr#1{\expandafter\def\csname scr#1\endcsname{{\mathscr{#1}}}\ignorespaces}
\def\makecal#1{\expandafter\def\csname cal#1\endcsname{{\cal #1}}\ignorespaces}
\def\makeudl#1{\expandafter\def\csname udl#1\endcsname{{\underline{#1}}}\ignorespaces}
\def\doLetters#1{%
  #1A #1B #1C #1D #1E #1F #1G #1H #1I #1J #1K #1L #1M
  #1N #1O #1P #1Q #1R #1S #1T #1U #1V #1W #1X #1Y #1Z}
\def\doletters#1{%
  #1a #1b #1c #1d #1e #1f #1g #1h #1i #1j #1k #1l #1m
  #1n #1o #1p #1q #1r #1s #1t #1u #1v #1w #1x #1y #1z}
\doLetters\makebb \doLetters\makecal \doLetters\makerm \doLetters\makebf \doLetters\makescr
\doletters\makerm \doletters\makebf \doLetters\makegr \doletters\makegr \doLetters\makeudl \doletters\makeudl

\newcommand{\qbinom}[2]{\genfrac{[}{]}{0pt}{}{#1}{#2}}
\newcounter{InSection}
\numberwithin{InSection}{subsection}
\newtheorem{thm}[InSection]{Theorem}
\newtheorem{lemma}[InSection]{Lemma}
\newtheorem{cor}[InSection]{Corollary}

\newtheorem{prop}[InSection]{Proposition}
\theoremstyle{definition}

\theoremstyle{definition}
\newtheorem{defn}[InSection]{Definition}
\theoremstyle{definition}
\newtheorem{rem}[InSection]{Remark}
\usepackage{titlesec}
\def\angles#1{\langle #1\rangle}

\def\makeop#1{
  \expandafter\def
  \csname#1\endcsname{
    \mathop{\rm #1}
    \nolimits
  }
  \ignorespaces
}

\def\makeoplist#1 {
  \def\@@tmpa{#1}
  \def\@@tmpb{***}
  \ifx\@@tmpa\@@tmpb
  \else
    \makeop{#1}
    \expandafter\makeoplist
  \fi
}

\makeoplist 
ad Ad Alb Alt Aut Br can Card card Char Cl Coker coker Cond cond Corr
depth diag D Diff Disc disc Div div dom End Ext Fil Fr Frac Frob Gal GL
Gr gr GSp GSpin H Hom hom hom  Id id Im im Ind ind Int inv Irr irr Isom
Ker ker len length lk Lie Nr Nrd O Ob ob ord PGL Pic Pr pr Proj PSL rank
RD Res Sgn sgn SL SO Sp Sp sp Spec Spf Spin 
St st Stab stab Std std SU Supp supp Swan Sym Tor tor tors 
torsion Tr tr Trd U UT val Vol vol wt ***

\DeclareMathOperator{\mdim}{mdim}
\usepackage{enumitem}
\usepackage{comment}
\usepackage[backend=biber]{biblatex}
\DeclareLanguageMapping{english}{english-apa}

\begin{document}
\title{
Zeta Functions for Spherical Tits Buildings of Finite General Linear Groups
}
\author{SHEN, Jianhao}
\maketitle

\begin{abstract}
In this paper, we define edge zeta functions for spherical buildings associated with finite general linear groups. We derive elegant formulas for these zeta functions and reveal patterns of eigenvalues of these buildings, by introducing and applying insightful tools including digraphs $X_0$ and $X_2$, cyclic $n$-partite graphs, partite-transitive group actions, and Springer's theorem on Hecke algebras.
\end{abstract}
\tableofcontents

\section*{Introduction}
\addcontentsline{toc}{section}{Introduction}

Buildings are a natural generalization of graphs, playing a significant role in various areas of study including geometry, combinatorics, and information theory. In the realm of algebra, two types of buildings naturally emerge: The Bruhat-Tits buildings, which are affine and associated with linear groups over local fields, and the Tits buildings, which are spherical and correspond to linear groups over finite fields.

Over the past two decades, the zeta function of finite quotients of affine buildings have been studied extensively, revealing crucial information about the affine building and finding the relation with $L$ function of such groups \cite{kang2010} \cite{li2011}. However, the zeta function of spherical buildings in the finite field case remains largely unexplored except for the rank 2 case.

For finite linear groups with rank 2, the spherical buildings degenerate to a special type of graphs known as generalized polygons. Generalized polygons were thoroughly studied by Feit and Higman \cite{feit_higman1964}. The spectral data of the generalized polygons were further analyzed in Tanner's work \cite{tanner1984} and Lubotzky's book \cite{lubotzky1994}. Because the spherical buildings reduce to graphs, the Ihara zeta function can be associated with such buildings, and can be computed easily with the spectral data presented in their work.

Beyond the rank 2 case, the zeta functions for spherical buildings in the finite field case have received little study. This presents a significant gap in our understanding of these structures. This paper seeks to bridge this gap by defining the edge zeta function of spherical buildings of finite linear groups, following a similar approach to that of Kang and Li \cite{kangli2014}\cite{li2019} for affine buildings, and agreeing with the Ihara zeta function in the rank 2 case. We then compute the edge zeta function for buildings associated with general linear groups $\GL_n(\mathbb{F}_q)$ and products of general linear groups $\prod_{i=1}^r \GL_{n_i}(\mathbb{F}_q)$.

Another relation among the various cases is that the building $\mathcal{B}(\GL_n(\mathbb{F}_p))$ can be viewed as a link of a vertex in the quotient of an affine building $\Gamma \backslash \mathcal{B}(\PGL_n(\mathbb{Q}_p))$, where $\Gamma$ is a regular discrete cocompact torsion-free subgroup of $\PGL_n(\mathbb{Q}_p)$, as studied by Kang and Li \cite{kangli2014}. Furthermore, by taking links of simplices in $\mathcal{B}(\GL_n(\mathbb{F}_q))$, we obtain buildings associated with products of general linear groups $\prod_{i=1}^r \GL_{n_i}(\mathbb{F}_q)$, where $\sum_{i=1}^r n_i = n$.

\paragraph{Structure of paper}
This paper is structured as follows: 

\begin{itemize}
\item Section 1 introduces the edge zeta function $Z(X,u)$ of a Tits building $X$ and specializes this definition to the buildings of $\GL(V)$ and $\GL(V_1) \times \GL(V_2)$. We then define two graphs $X_0$ and $X_2$ that capture the connectivity and the eigenvalues of these buildings. Our main theorem on the pattern of the eigenvalues of the buildings is then stated in terms of the eigenvalues of $X_0$, $X_2$.

\item In Sections 2 and 3, we develop graph-theoretical tools like cyclic $n$-partite
graphs and relative destination elements to analyze the connectivity and eigenvalues of the component subgraphs of $X_0$ and $X_2$. Section 4 provides explicit computations of the
relative destination elements that encode cycle counts in $X_0$ and $X_2$.

\item Sections 5 and 6 give an overview of unipotent representations and Springer's theorem on the action of a certain element central in the relevant Hecke algebra. This allows us to relate eigenvalues to character values.

\item Sections 7 and 8 use the above tools to derive formulas for the zeta functions of components of $X_0$ and $X_2$, and prove the main theorem. Section 9 concludes by summarizing the results and suggesting directions for further research.
\end{itemize}

The technical heart of the paper lies in Sections 2-6 where we develop the graph-theoretic and representation-theoretic tools to analyze the zeta functions of the graphs $X_0$ and $X_2$ and the associated buildings.

\paragraph{Acknowledgements}
I would like to thank my thesis advisor, Professor Jiu-Kang Yu, for his invaluable and constant guidance, support, and encouragement throughout this research. I also thank Professor Wen-Ching Winnie Li, Professor Caihua Luo, and Professor Scott Andrews for much helpful discussion.

\section{Geodesics Cycles in Buildings of $\GL(V)$ and $\prod\GL(V_i)$}
We begin by considering $V$ as an $n$-dimensional vector space over a field, denoted by $k$. For our discussions, we limit ourselves to cases where $n\geq 2$, as the building is empty by definition for $n=1$. In most of our discussion, $k$ is a finite field $\mathbb{F}_q$.

\subsection{Zeta Function of a Finite Spherical Building}

\paragraph{Defining the Building of $\GL(V)$}
Following Serre's work \cite{serre2005}, we define the building of $\GL(V)$ as follows:

\begin{defn}
The building of $\GL(V)$, or the building of $V$, is a simplicial complex $\mathcal{B}=\mathcal{B}(V)$ of dimension $n-2$. It is constructed as follows:
\begin{itemize}
    \item The vertices of $\mathcal{B}$ represent the proper nontrivial subspaces of $V$. We denote a vertex corresponding to a subspace $W$ by $x_W$.
    \item A collection $s=\{x_{W_1},...,x_{W_r}\}$ of vertices forms a simplex if and only if the corresponding subspaces $\{W_1,...,W_r\}$ constitute a flag. Equivalently, for any $i,j$, either $W_i\subseteq W_j$ or $W_j\subseteq W_i$.
\end{itemize}
\end{defn}

\paragraph{Opposite Vertices, Link of a Simplex, Geodesic Paths and Zeta Functions}
The spherical nature of the building $\mathcal{B}(V)$ introduces several key concepts, including `opposite vertices', the `link of a simplex'. We refer the reader to Abramenko and Brown's \textit{Buildings} \cite{abramenko2008buildings} for these concepts. In the case of $\mathcal{B}(V)$, the opposite relation is characterized by \cite{serre2005} as:

\begin{defn}
    Two vertices $x_{W}$ and $x_{W'}$ in $\mathcal{B}(V)$ are \textbf{opposite} if $V=W\oplus W'$.
\end{defn}

Now, we revisit the definition of the link of a simplex in a general simplicial complex:
\begin{defn}
For a simplex $A$ in a simplicial complex $X$, the \textbf{link} of $A$, denoted by $\lk A$ or $\lk_\mathcal{B}(A)$, is the subcomplex of $X$ containing simplicies 
$$\{B\in X: A\cap B=\varnothing, A\cup B \text{ is a simplex in } X\}.$$
\end{defn}

Notably, the link of any simplex in a spherical building $\mathcal{B}$ forms another spherical building and retains the concept of opposite vertices. This property is essential in defining geodesic paths and zeta functions in spherical buildings. 

\begin{defn}\label{def:1.1.4 geodesic and cycles}
    In a spherical building $\mathcal{B}$, a sequence of vertices $(x_0,x_1,...,x_l)$ constitutes a \textbf{path} of length $l$ if for each $0\leq i\leq l-1$, $x_i\neq x_{i+1}$ and $\{x_i,x_{i+1}\}$ forms a simplex in $\mathcal{B}$. If for each $1\leq i\leq l-1$, $x_{i-1}$ and $x_{i+1}$ are opposite in $\lk x_i$, the path is called a \textbf{geodesic path}. 
    
    A \textbf{geodesic cycle} is a special geodesic path that is closed and remains a geodesic cycle when extended indefinitely. Formally, a sequence of vertices $(x_0,x_1,...,x_l)$ forms a geodesic cycle if $x_l=x_0$ and $(x_0,x_1,...,x_l,x_1)$ is a geodesic path.
\end{defn}

For finite structures like $\mathcal{B}(V)$, where $V$ is a finite-dimensional vector space over a finite field $\mathbb{F}_q$, the number of geodesic paths for each length is finite, facilitating the definition of the edge zeta function. 

\begin{defn}
    In a finite spherical building $\mathcal{B}$, let $N(l)$, or $N(\mathcal{B},l)$ denote the number of geodesic cycles of length $l$. The edge zeta function $Z(\mathcal{B},u)$ is then the formal power series defined as
 $$Z(\mathcal{B},u)=\exp(\sum_{l=1}^{\infty}\frac{N(l)}{l} u^l).$$
\end{defn}

\subsection{Exploring Geodesics in $\mathcal{B}(V)$ and $\mathcal{B}(V_1,...,V_r)$}
In this section, we focus on the computation of the link of a single vertex in the context of $\mathcal{B}(V)$. Let $W\subseteq V$ be an $m$-dimensional subspace of $V$, with $0<m<n$. The link $\lk x_W= \lk\{x_W\}$ of the vertex $x_W$ in $\mathcal{B}(V)$ has vertex set $\mathcal{V}_1\cup \mathcal{V}_2$, where $\mathcal{V}_1=\{x_{W'}: 0\subsetneq W'\subsetneq W\}$ and $\mathcal{V}_2=\{x_{W'}: W\subsetneq W'\subsetneq V\}$. The simplices in $\lk(x_W)$ are the flags in $\mathcal{B}(V)$ consisting of elements in $\mathcal{V}_1\cup \mathcal{V}_2$. This simplicial complex $\lk x_W$ is isomorphic to the join of $\mathcal{B}(W)$ and $\mathcal{B}(V/W)$, which is the building of $\GL(W)\times \GL(V/W)$. We introduce these terms now.

\paragraph{Join of Simplicial Complexes}

\begin{defn}
The \textbf{join} of a finite collection $\{X_i\}_{i=1}^r$ of simplicial complexes, each with distinct vertex sets $\mathcal{V}_i$, is the complex $X_1*...*X_r$. This complex has a vertex set $\dot\bigcup_{i=1}^r \mathcal{V}_i$ and a simplex set $\{\dot\bigcup_{i=1}^r S_i :S_i \text{ is a simplex in } X_i\}$.
\end{defn}

Returning to the discussion of $\lk_{\mathcal{B}(V)} x_W$, $\mathcal{V}_2$ can be identified with the vertices in $\mathcal{B}(V/W)$ through $W'\mapsto W'/W$. Since flags with vertices in $\mathcal{V}_1\cup \mathcal{V}_2$ are exactly the union of a flag in $\mathcal{V}_1$ and a flag in $\mathcal{V}_2$, one can easily verify that $\lk x_W\simeq \mathcal{B}(W)*\mathcal{B}(V/W)$ via $x_{W'}\mapsto \begin{cases}
x_{W'}, \quad \text{if } W'\subseteq W \\
x_{W'/W}, \quad \text{if } W'\supseteq W.
\end{cases}$

We now define the buildings associated with vector spaces $(V_1, V_2, ..., V_r)$:

\begin{defn}
Let $(V_1,V_2, ..., V_r)$ be a finite collection of finite-dimensional vector spaces over $k$, with dimensions $n_1,n_2, ..., n_r$ respectively. The building $\mathcal{B}(V_1,V_2, ..., V_r)$ of $\prod_{i=1}^r \GL(V_i)$ is the join of the buildings of each $\GL(V_i)$. We also call $\mathcal{B}(V_1,V_2, ..., V_r)$ the building of $(V_1,V_2, ..., V_r)$.
\end{defn}

Next, we discuss opposing points in a joined spherical building structure. The result is summarized in the following proposition, which also serves here as a definition. See paragraph 2.39 in Tits' book \cite{tits1974buildings} for a reference.

\begin{prop}
In the building $\mathcal{B}=\mathcal{B}_1*...*\mathcal{B}_r$, where each $\mathcal{B}_i$ is a spherical building with a defined opposition relation, two vertices $x_{W},x_{W'}$ are opposite in $\mathcal{B}$ if and only if they are opposite vertices in $\mathcal{B}_i$ for some $i \in {1,...,r}$.
\end{prop}

\paragraph{Geodesic Paths and Cycles in $\mathcal{B}(V)$}

Let's turn our attention to the geodesic paths and cycles in $\mathcal{B}(V)$:

\begin{prop}\label{prop:1.2.5 on geodesic in X(V)}
    Let $W_0, W_1,..., W_l$ be proper nontrivial subspaces of a vector space $V$. The sequence $(x_{W_0}, x_{W_1},...,x_{W_l})$ is a geodesic path in $\mathcal{B}(V)$ if and only if for any $1\leq i\leq l-1$, either $W_{i-1}\oplus W_{i+1}=W_i$ or ($W_{i-1}\cap W_{i+1}=W_i$ and $W_{i-1}+W_{i+1}=V$).

    The sequence $(x_{W_0}, x_{W_1},...,x_{W_l})$ is a geodesic cycle in $\mathcal{B}(V)$ if and only if this relation among $(W_{i-1}, W_i, W_{i+1})$ holds for all $1\leq i\leq l$, with $W_l=W_0$, and $W_{l+1}$ is set to be $W_1$.
\end{prop}

\begin{proof}
    For $(x_{W_0}, x_{W_1},...,x_{W_l})$ to be a geodesic path in $\mathcal{B}(V)$, the condition is that for each $1\leq i\leq l-1$, $x_{W_{i-1}}$ and $x_{W_{i+1}}$ are opposite vertices in $\lk_V(x_{W_i})\simeq \mathcal{B}(W_{i})*\mathcal{B}(V/W_i)$. Then either $x_{W_{i-1}}$ and $x_{W_{i+1}}$ are opposite in $\mathcal{B}(W_{i})$, or $x_{W_{i-1}/W_i}$ and $x_{W_{i+1}/W_i}$ are opposite in $\mathcal{B}(V/W_i)$. In the former case, $W_{i-1}\oplus W_{i+1}=W_i$, and in the latter case, $W_{i-1}\supsetneq W_i$ and $W_{i+1}\supsetneq W_i$, with $W_{i-1}/W_i\oplus W_{i+1}/W_i= V/W_i$. Equivalently, $W_{i-1}\cap W_{i+1}=W_i$ and $W_{i-1}+W_{i+1}=V$.
    
    The condition for geodesic cycles then follows from Definition \ref{def:1.1.4 geodesic and cycles}.
\end{proof}
\paragraph{Geodesic Paths  in $\mathcal{B}(V_1,\ldots,V_r)$}
We now consider geodesic paths in the building $\mathcal{B}(V_1,\ldots,V_r)=\mathcal{B}(V_1)*\ldots*\mathcal{B}(V_r)$, where $V_1,\ldots,V_r$ are disjoint vector spaces over a common field $k$. For a subspace $W\in V_i$, the link of the vertex $x_W$ is $\lk_{\mathcal{B}(V_1,\ldots,V_r)}(x_W)=\mathcal{B}(V_1)*\ldots*\mathcal{B}(V_{i-1})*\lk_{\mathcal{B}(V_i)}(x_W)*\mathcal{B}(V_{i+1})*\ldots*\mathcal{B}(V_r)$, and is isomorphic to the join $\Asterisk_{j\neq i}\mathcal{B}(V_j) *\mathcal{B}(W)*\mathcal{B}(V_i/W)$.

\begin{prop}\label{prop:1.2.6 on geodesic paths}
Let $V_i$ be finite-dimensional vector spaces over a field $k$. Also, consider a sequence of vertices $(x_{W_0}, x_{W_1},...,x_{W_l})$ in $\mathcal{B}(V_1,...,V_r)$, where each $W_s$ is a subset of a certain $V_i$. The sequence $(x_{W_0}, x_{W_1},...,x_{W_l})$ is a geodesic path in the building $\mathcal{B}(V_1,\ldots,V_r)$ if and only if one of the following holds:

(1) All $W_s$ are contained in a single $V_i$, and the sequence $(x_{W_0}, x_{W_1},...,x_{W_l})$ forms a geodesic path in $\mathcal{B}(V_i)$.

(2) There exists some $i,j$ such that the $W_s$ are alternately contained in $V_i$ and $V_j$, and for any $0\leq s \leq l-2$, if $W_s\subseteq V_i$, then $W_s\oplus W_{s+2}=V_i$;  if $W_s\subseteq V_j$, then $W_s\oplus W_{s+2}=V_j$.
\end{prop}

\begin{proof}
For the sequence $(x_{W_0}, x_{W_1},...,x_{W_l})$ to be a geodesic path in $\mathcal{B}(V_1,\ldots,V_r)$, it is necessary that for each $0\leq s \leq l-2$, the vertices $x_{W_s}$ and $x_{W_{s+2}}$ are opposite in the link of $x_{W_{s+1}}$.

First, consider the case where $W_0$ and $W_1$ are contained in the same ambient vector space, say $V_i$. Then $W_0$ and $W_2$ are opposite in $\lk(W_1)\simeq \Asterisk_{j\neq i}\mathcal{B}(V_j) *\mathcal{B}(W_1)*\mathcal{B}(V_i/W_1)$. Then $x_{W_0}$ falls in $\mathcal{B}(W_1)\ast \mathcal{B}(V_i/W_1)$ and so does $x_{W_2}$. This implies that $W_2\subseteq V_i$ and $(x_{W_0}, x_{W_1}, x_{W_2})$ is a geodesic path in $\mathcal{B}(V_i)$. By induction, each $W_{s}\subseteq V_i$, and $(x_{W_0}, x_{W_1},...,x_{W_l})$ is a geodesic path in $\mathcal{B}(V_i)$. This accounts for situation (1).

Now, consider the case where $W_s$ are contained in different ambient spaces. Then there exists $i\neq j$ such that $W_0 \subseteq V_i$ and $W_1 \subseteq V_j$. In this scenario, the link $\lk(x_{W_1})$ is isomorphic to the join $\Asterisk_{k\neq j}\mathcal{B}(V_k) *\mathcal{B}(W_1)*\mathcal{B}(V_j/W_1)$, and we require $x_{W_0}$ and $x_{W_2}$ are opposite in this link. Since the vertex $x_{W_0}$ resides in $\mathcal{B}(V_i)$, the vertex $x_{W_2}$ must also lie in $\mathcal{B}(V_i)$. The condition now means that $W_0$ and $W_2$ are both subsets of $V_i$ with $W_0\oplus W_2=V_i$.

Now, let's consider $W_3$. Given that $W_1\subseteq V_j$ and $W_2\subseteq V_i$, we reason similarly to deduce that $W_3$ is a subset of $V_j$. Moreover, the direct sum of $W_1$ and $W_3$ equals $V_j$: $W_1\oplus W_3=V_j$.

Continuing in this manner for the entire sequence, we establish that the subspaces $W_s$ are alternately contained in $V_i$ and $V_j$. Furthermore, for any $0\leq s \leq l-2$, if $W_s\subseteq V_i$, then $W_s\oplus W_{s+2}=V_i$; if $W_s\subseteq V_j$, then $W_s\oplus W_{s+2}=V_j$. 
\end{proof}

It's noteworthy that the zeta functions of finite buildings, as defined above, are inverses of an integer polynomial with constant coefficient 1. This property is a variation of Hashimoto’s theorem\cite{hashimoto1989}. Therefore, we can write $1/Z(\mathcal{B},u)=(1-\lambda_1 u)...(1-\lambda_ku)$, where $\lambda_i$ are the roots of the polynomial. Unlike with graphs, we do not consider 0 as an eigenvalue of buildings, since we are primarily interested in properties of the zeta function rather than the full spectrum. We define these $\lambda_i$ as the \textbf{eigenvalues} of the building. 
The main result of this paper pertains to these eigenvalues.

\begin{thm}[Main theorem on eigenvalues of buildings]\label{thm:main thm on buildings}
For $\mathcal{B}=\mathcal{B}(\mathbb{F}_q^n)$, or $\mathcal{B}(\mathbb{F}_q^{n_1},..., \mathbb{F}_q^{n_r})$, all eigenvalues are a root of unity times a fractional power of $q$.
\end{thm}

\subsection{The Digraphs $X_0(V)$ and $X_2(V)$}

To capture the connectivity of our structures, we introduce two graphs, $X_0(V)$ and $X_2(V)$. We will define the closed walk zeta functions $Z_c(X,u)$ for a finite graph $X$, and describe how the closed walk zeta functions of the graphs $X_0(V)$ and $X_2(V)$ determines the edge zeta function of buildings $\mathcal{B}(V)$ and $\mathcal{B}(V_1,...,V_r)$.

\subsubsection{Basic Notions of Graphs}
Before diving into the specific graphs $X_0(V)$ and $X_2(V)$, let's lay out some basic definitions and concepts about graphs:

A \textbf{simple directed graph}, or a \textbf{simple digraph} $X = (\mathcal{V}, \mathcal{E})$, is a mathematical structure defined by two sets: a set of \textbf{vertices} $\mathcal{V}$, and a set of \textbf{edges} $\mathcal{E}$. The edge set $\mathcal{E}$ a subset of the Cartesian product $\mathcal{V} \times \mathcal{V}$ such that $(v,v) \not\in \mathcal{E}$ for any $v \in \mathcal{V}$. Each edge $e\in\mathcal{E}$ can be written as $e=(u,v)$ for some $u,v\in \mathcal{V}$, where $u$ is called the origin of $e$, and $v$ is called the terminus of $e$. We call $e$ an edge from $u$ to $v$.

An \textbf{undirected graph} is a special case of a simple directed graph where for every edge $(u, v)$, the reverse $(v, u)$ is also an edge. In other words, if there is an edge from vertex $u$ to vertex $v$, then there is also an edge from $v$ to $u$.

Throughout this paper, we will only consider simple graphs, so we will omit the adjective `simple' in our discussions. Furthermore, in this paper, the term `graph'  means `undirected graph', and `digraph' refers to both `directed' and `undirected' graphs, as all undirected graphs will be considered as digraphs in our context.

A \textbf{finite digraph} is a simple digraph in which both the set of vertices $\mathcal{V}$ and the set of edges $\mathcal{E}$ are finite. 

A \textbf{walk} of length $l$ in a directed graph is a sequence of vertices $(v_0, v_1, ..., v_l)$ such that for each $i = 0, 1, ..., l - 1$, the pair $(v_i, v_{i+1})$ is an edge in the graph. If the sequence ends where it started, i.e., $v_l=v_0$, it is known as a \textbf{closed walk} of length $l$.  Note that in some sources,  what we have defined here as a `walk' is referred to as a `path'. 

\paragraph{Closed Walk Zeta Function of a Digraph}

Next, we define the closed walk zeta function $Z_c(X,u)$ for a finite digraph $X$. This function captures the total number of closed walks of each length in the digraph $X$. For a finite digraph $X$, let $N_c(l)$ or $N_c(X,l)$ denote the number of closed walks of length $l$. The closed walk zeta function $Z_c(X,u)$ is then the formal power series defined as $$Z_c(X,u)=\exp(\sum_{l=1}^{\infty}\frac{N_c(l)}{l} u^l).$$

\paragraph{Eigenvalues of a Digraph}
Similar to the finite building case, the closed path zeta function $Z_c(X,u)$ is also the reciprocal of an integer polynomial with constant coefficient 1 (see Theorem \ref{thm:2.1.1 Hashi}). Therefore, we can write $1/Z_c(X,u)=(1-\lambda_1 u)...(1-\lambda_ku)$, where $\lambda_i$ are the roots of the polynomial. We define these $\lambda_i$ as the \textbf{eigenvalues} of a digraph.
In addition, the digraph $X$ has an eigenvalue 0 with multiplicity given by $|\mathcal{V}(X)| - \deg(1/Z_c(X,u))$, where $\mathcal{V}(X)$ is the vertex set of $X$. This definition agrees with the definition of eigenvalues as eigenvalues of the incidence matrix, as described in section \ref{sec:2}.

\subsubsection{Definitions of $X_0(V)$ and $X_2(V)$}

\begin{defn}
The graph $X_0(V)$ is a simple, undirected graph. Its vertex set is the same as the vertex set of $\mathcal{B}(V)$, and corresponds to the set of nontrivial proper subspaces of $V$. Again, for a nontrivial proper subspace $W$ of $V$, we denote by $x_W$ the vertex corresponding to $W$.

An undirected edge exists between two vertices $x_{W_1}$ and $x_{W_2}$ if and only if their direct sum of the corresponding subspaces equals $V$, i.e., $W_1\oplus W_2=V$. 
\end{defn}

\paragraph{Directed Flags and Definition of the Digraph $X_2(V)$}

\begin{defn}
A \textbf{directed flag} $F$ in a vector space $V$ is an ordered pair $(W_1,W_2)$ of proper nontrivial subspaces of $V$, such that $W_1\subsetneq W_2$ or $W_2\subsetneq W_1$.
\end{defn}

\begin{defn}
Given a directed flag $F = (W_1, W_2)$ in vector space $V$, the \textbf{multi-dimension} of $F$, denoted by $\mdim(F)$, is defined as the ordered pair $(\dim(W_1), \dim(W_2))$.

\end{defn}
Note that in a directed flag $F = (W_1, W_2)$, the vector spaces $W_1$ and $W_2$ have different dimensions. Therefore, the possible multi-dimensions of $F$ within an $n$-dimensional vector space form the set $\{(a,b):1\leq a,b\leq n-1, a\neq b\}$.

\begin{defn}
The digraph $X_2(V)$ is defined as a directed graph whose vertices are directed flags in $V$. In other words, each directed flag $F = (W_1, W_2)$ in $V$ represents a vertex in $X_2(V)$, denoted by $x_F$.
Let $F_1=(W_1,W_2)$ and $F_2=(W_3,W_4)$ be two directed flags in $V$, there is a directed edge from $x_{F_1}$ to $x_{F_2}$ in $X_2(V)$ if and only if $W_2=W_3$ and $(W_1, W_2, W_4)$ forms a geodesic path in $\mathcal{B}(V)$. 
\end{defn}

\begin{rem}
The digraph $X_2(V)$ can be seen as the "edge graph" of $\mathcal{B}(V)$. Each vertex in $X_2(V)$ corresponds to a directed edge in $\mathcal{B}(V)$, and each directed edge in $X_2(V)$ corresponds to a geodesic path of length 2 in $\mathcal{B}(V)$.
\end{rem}

\begin{rem}\label{rem:1.3.6 on redefining edge}
Based on Proposition \ref{prop:1.2.5 on geodesic in X(V)}, the condition that $(W_1, W_2, W_4)$ forms a geodesic path in $\mathcal{B}(V)$ is equivalent to $W_{1}\oplus W_{4}=W_2$ or ($W_{1}\cap W_{4}=W_2$ and $W_{1}+W_{4}=V$). Therefore, if there is a directed edge from $x_{F_1}$ to $x_{F_2}$, then the multi-dimension $\mdim(F_2)$ of $F_2$ is determined by $\mdim(F_1)$. We will formalize and extend this observation in Proposition \ref{prop:2.2.1 on multidim trend}.
\end{rem}

\subsubsection{Relating Geodesic Cycles in Buildings and Digraphs}
\paragraph{Reformulating Geodesics in $\mathcal{B}(V)$ and Their Zeta Function}
With the digraph $X_2(V)$, we can reformulate the problem of counting geodesic cycles in $\mathcal{B}(V)$ as counting closed walks in $X_2(V)$. The relationship between geodesic cycles in $\mathcal{B}(V)$ and closed walks in $X_2(V)$ is summarized in the following proposition:

\begin{prop}
There is a one-to-one correspondence between geodesic cycles in the building $\mathcal{B}(V)$ and closed walks in the digraph $X_2(V)$. In particular, each geodesic cycle $(x_{W_0}, x_{W_1}, ..., x_{W_l})$ in $\mathcal{B}(V)$ corresponds bijectively to a closed walk $(x_{F_0}, x_{F_1}, ..., x_{F_{l-1}},x_{F_0})$ in $X_2(V)$.
\end{prop}

\begin{proof}
Given a geodesic cycle $(x_{W_0}, x_{W_1}, ..., x_{W_l})$ in $\mathcal{B}(V)$, we can associate it with a closed walk $(x_{F_0}, x_{F_1}, ..., x_{F_{l-1}},x_{F_0})$ in $\mathcal{B}(V)$, where each ${F_i}$ is the directed flag $({W_i}, {W_{i+1}})$. This map is well-defined: the definition of a geodesic cycle ensures that $x_{W_l}=x_{W_0}$ and $(x_{W_0}, x_{W_1}, ..., x_{W_l},x_{W_1})$ is a geodesic path in $\mathcal{B}(V)$. Therefore, $(x_{F_0}, x_{F_1}, ..., x_{F_{l-1}},x_{F_0})$ forms a walk in $X_2(V)$.

Conversely, given a closed walk $(x_{F_0}, x_{F_1}, ..., x_{F_{l-1}},x_{F_0})$ in $X_2(V)$, there exists a unique sequence of spaces $x_{W_0},...,x_{W_l}$ such that for each $i$, ${F_i}=({W_i},{W_{i+1}})$. Then the condition that the walk $(x_{F_0}, x_{F_1}, ..., x_{F_{l-1}},x_{F_0})$ in $X_2(V)$ is closed implies that $x_{W_l}=x_{W_0}$, and that $(x_{W_0},x_{W_1},...,x_{W_l},x_{W_1})$ is a geodesic path. Then $(x_{W_0},x_{W_1},...,x_{W_l})$ forms a geodesic cycle in $\mathcal{B}(V)$, which we associate to the closed walk.

By construction, these two correspondences are inverse to each other. Therefore, there is a bijection between geodesic cycles in $\mathcal{B}(V)$ and closed walks in $X_2(V)$. 
\end{proof}

\begin{cor}\label{cor:1.3.8 on zeta X2 and X}
For $V$ a finite-dimensional vector space over a finite field, the closed walk zeta function of the digraph $X_2(V)$ is equal to the edge zeta function of the building $\mathcal{B}(V)$. Formally, we have $Z_c(X_2(V),u) = Z(\mathcal{B}(V),u)$.
\end{cor}
\begin{proof}
By the previous proposition, there is a bijection between geodesic cycles in $\mathcal{B}(V)$ and closed walks in $X_2(V)$. Therefore, the number of closed walks $N_c(X_2(V),l)$ equals the number of geodesic paths $N(\mathcal{B}(V),l)$ in $\mathcal{B}(V)$. The result then follows from the definitions of geodesics.
\end{proof}

\paragraph{Reformulating Geodesics in $\mathcal{B}(V_1,...,V_r)$ and Their Zeta Function}

Next, consider the building $\mathcal{B}(V_1,...,V_r)$, where $V_1,...,V_r$ are vector spaces over a common field $k$. The structure of geodesic paths in $\mathcal{B}(V_1,...,V)r)$ is more complex, and can be characterized by the following proposition:

\begin{prop}
The sequence of vertices  $(x_{W_0}, x_{W_1},...,x_{W_l})$ in $\mathcal{B}(V_1,...,V_r)$ is a geodesic cycle in the building $\mathcal{B}(V_1,...,V_r)$ if and only if one of the following holds:

(1) All $W_s$ are contained in a single $V_i$, and the sequence $(x_{W_0}, x_{W_1},...,x_{W_l})$ forms a geodesic cycle in $\mathcal{B}(V_i)$.

(2) There exists some $i,j$ such that the $W_s$ are alternately contained in $V_i$ and $V_j$, and $l$ is even. Moreover, the sequence $(x_{W_0}, x_{W_2},...,x_{W_l})$ is a closed walk in $X_0(V_i)$, and $(x_{W_1}, x_{W_3},...,x_{W_{l-1}},x_{W_1})$ is a closed walk in $X_0(V_{3-i})$, where $i=1$ or 2.
\end{prop}

\begin{proof}
By definition of geodesic cycles, the sequence $(x_{W_0}, x_{W_1},...,x_{W_l},x_{W_1})$ must be a geodesic path and $x_{W_0}=x_{W_l}$. According to Proposition \ref{prop:1.2.6 on geodesic paths}, we have two cases:

If all  $W_s$ are contained in a single $V_i$, then $(x_{W_0}, x_{W_1},...,x_{W_l},x_{W_1})$ forms a geodesic path in $\mathcal{B}(V_i)$. Since $x_{W_0}=x_{W_l}$, $(x_{W_0}, x_{W_1},...,x_{W_l})$ is a geodesic cycle in $\mathcal{B}(V_i)$.

If the $W_s$ are alternately contained in $V_i$ and $V_j$ with $W_0\subseteq V_i$, then $l$ is even since $W_l=W_0$ are contained in the same ambient space, and for each $0\leq r\leq l-1$, $W_r\oplus W_{r+2}$ equals either $V_i$ or $V_j$, where $W_{l+1}$ is defined as $W_1$. Then the sequence $(x_{W_0}, x_{W_2},...,x_{W_l}=x_0)$ is a closed walk in $X_0(V_i)$, and $(x_{W_1}, x_{W_3},...,x_{W_{l-1}},x_{W_1})$ is a closed walk in $X_0(V_j)$. The converse also holds by definition.
\end{proof}

\paragraph{Relating $Z_c(\mathcal{B}(V),u)$ and $Z(\mathcal{B}(V_1,...,V_r),u)$}

The relation between the closed walk zeta function of the graph $X_0(V)$ and the edge zeta function of the building $\mathcal{B}(V_1,...,V_r)$ can now be made precise. The formula involves tensor product of graphs, whose properties will be studied in the next section.
\begin{cor}\label{cor:1.3.10 on zeta X0 and XV1V2}
The edge zeta function of the building $\mathcal{B}(V_1,...,V_r)$ satisfies:
$$Z(\mathcal{B}(V_1,...,V_r),u) = \prod_{i=1}^r Z(\mathcal{B}(V_i),u)\times \prod_{1\leq i<j\leq r} Z_c(X_0(V_i)\times X_0(V_j),u^2),$$
where $X_0(V_i)\times X_0(V_j)$ is the tensor product of the graphs $X_0(V_i)$ and $X_0(V_j)$.
\end{cor}

\begin{proof}
From the above proposition, the geodesic cycles $(x_{W_0},...,x_{W_l})$ of length $l$ in $\mathcal{B}(V_1,...,V_r)$ can be divided into two cases: all $W_s$ are in $V_i$ for some $i$, or $W_s$ are alternately in $V_i$ and $V_j$ with $W_0\subseteq V_i$ for some $i\neq j$. 

The first cases contribute $\sum_i N(\mathcal{B}(V_i),l)$  to $N(\mathcal{B}(V_1,...,V_r),l)$. These contributions correspond to the factors $\prod_i Z(\mathcal{B}(V_i),u)$ in the edge zeta function $Z(\mathcal{B}(V_1,...,V_r),u)$.

In the second scenario, we firstly fix $i,j$ with $i< j$. The two alternating cases involving $V_i$ and $V_j$ with $W_0\subseteq V_i$ or $W_0\subseteq V_j$ totally contribute $2N_c(X_0(V_i),l/2)N_c(X_0(V_j),l/2)$ to $N(\mathcal{B}(V_1,...,V_r),l)$ when $l$ is even. This contribution corresponds to the factor $$\exp(\sum_{l=1}^{\infty}\frac{2N_c(X_0(V_i),l)N_c(X_0(V_j),l)}{2l} u^{2l})= Z_c(X_0(V_i)\times X_0(V_j),u^2)$$ in the edge zeta function $Z(\mathcal{B}(V_1,...,V_r),u)$, because $$N_c(X_0(V_i)\times X_0(V_j),l)=N_c(X_0(V_i),l)N_c(X_0(V_j),l)$$ for the tensor product of graphs $X_0(V_i)\times X_0(V_j)$ (See Section \ref{sec:2}). Letting $i<j$ run over the range $1\leq i<j\leq n$ gives the desired equality.
\end{proof}

This corollary reflects the intricate interplay of the structures of $\mathcal{B}(V_i)$, and $X_0(V_i)\times X_0(V_j)$ in the building $\mathcal{B}(V_1,...,V_r)$.
We now formulate the main theorem on graphs.
\begin{thm}[Main theorem on eigenvalues of graphs]\label{thm:main thm on graphs 1.3.11}
The nonzero eigenvalues of the digraphs $X_0(\mathbb{F}_q^n)$ and $X_2(\mathbb{F}_q^n)$ are a root of unity times a fractional power of $q$.
\end{thm}

\begin{proof}[Proof of Theorem \ref{thm:main thm on graphs 1.3.11} implies Theorem \ref{thm:main thm on buildings}]
Let $V$, $V_1,...,V_r$ be finite-dimensional vector spaces over $\mathbb{F}_q$.

For the building $\mathcal{B}(V)$, Corollary \ref{cor:1.3.8 on zeta X2 and X} states that the edge zeta function $Z(\mathcal{B}(V),u)$ equals the closed walk zeta function $Z_c(X_2(V),u)$. Therefore, the nonzero eigenvalues of $\mathcal{B}(V)$ are the same as the eigenvalues of $X_2(V)$, and are of the desired form by Theorem \ref{thm:main thm on graphs 1.3.11}

For the building $\mathcal{B}(V_1,...,V_r)$, by the expression in Corollary \ref{cor:1.3.10 on zeta X0 and XV1V2},
The eigenvalues of $\mathcal{B}(V_1,...,V_r)$ are: the nonzero eigenvalues of $\mathcal{B}(V_i)$ for some $i$, and the square roots of the nonzero eigenvalues of $X_0(V_i)\times X_0(V_j)$ for some $i<j$.

By Theorem \ref{thm:main thm on graphs 1.3.11}, the nonzero eigenvalues of $X_0(V_i)$ and $X_0(V_j)$ have the form of roots of unity times fractional powers of $q$. Their products, which, by Proposition \ref{prop:2.1.6 Eigenvalue of products}, are the nonzero eigenvalues of the tensor product graph $X_0(V_i)\times X_0(V_j)$, also have this form. Taking square roots preserves the structure. Therefore, the eigenvalues of $\mathcal{B}(V_1,...,V_r)$ must be roots of unity times fractional powers of $q$, as claimed.
\end{proof}

In the next sections, we will examine the eigenvalues of $X_0(V)$ and $X_2(V)$ by studying their graph structures. We will abstract their common properties by considering cyclic multipartite digraphs and group actions. We will develop general theories and then apply them to our specific cases.

\section{Properties of the Digraphs $X_0(V)$ and $X_2(V)$}\label{sec:2}
\subsection{Generalities on Zeta Function and Eigenvalues of a Digraph}

In this section, we delve into some general properties of the closed walk zeta function and the eigenvalues of a digraph. These properties lay the groundwork for our later discussions on the eigenvalues of the digraphs $X_0(V)$ and $X_2(V)$.

\subsubsection{Polynomial Nature of the Inverse Zeta Function}
A key result regarding the zeta function of a digraph is that the inverse of the Ihara zeta function of a digraph is a polynomial. We present the proof here.

Let $X$ be a finite  digraph, with vertex set $\calV=\{v_1,...,v_{|\calV|}\}$. Suppose $A$ is the adjacency matrix of the digraph $X$, where each entry $A_{ij}$ equals the number of edges from vertex $x_i$ to vertex $x_j$ in $X$. Therefore, the entry $(A^l)_{ij}$ equals the number of walks of length $l$ from vertex $i$ to vertex $j$ in $X$. In particular, $(A^l)_{ii}$ equals the number of closed walks of length $l$ starting and ending at vertex $i$ in $X$. 
The trace of the matrix $A^l$, denoted by $\Tr(A^l)$, thus equals the total number of closed walks of length $l$ in $X$. Therefore, we have $\Tr(A^l) = N_c(l)$.

\begin{thm}\label{thm:2.1.1 Hashi}
    Let $X$ be a finite simple directed graph. Then the inverse of the zeta function counting closed walks is a polynomial with integer coefficients.
\end{thm}
\begin{proof}
    Note that $$Z_c(X,u)=\exp(\sum_{m=1}^\infty \frac{N_c(m)}{m} u^m)=\exp(\sum_{m=1}^\infty \frac{\Tr(A^m)}{m} u^m).$$  Let $\lambda_1,...,\lambda_r$ be all the eigenvalues of $A$, counting algebraic multiplicities. Then by the following computational Lemma \ref{lemma:2.3.3}, $$Z_c(X,u)=\prod_{i=1}^r \frac{1}{(1-\lambda_i u)}=\frac{1}{\det(1-Au)}.$$  Each entry of the connectivity matrix $A$ equals 0 or 1. Hence, $\det(1-Au)\in \bbZ[u]$, and our result follows.
\end{proof}

This theorem confirms that the zeta functions of our digraphs $X_0(V)$ and $X_2(V)$ are inverses of polynomials with integer coefficients. Moreover, it shows that the eigenvalues of a finite digraph can be identified with the eigenvalues of its adjacency matrix.

We now present a useful computational lemma.
\begin{lemma}\label{lemma:2.3.3}
    Let $N$ be a $M$-dimensional  vector space over $\bbC$, and let $f:N\to N$ be a $\bbC$-linear map.  Let $\lambda_1,...,\lambda_M$ be the eigenvalues of $f$. For a positive integer $t$, we have the following identity:
    $$\exp\left(\sum_{m=1}^{\infty} \frac{\Tr(f^m)}{tm}u^{tm}\right) = \prod_{i=1}^M\frac{1}{(1-\lambda_i u^t)^{\frac{1}{t}}}.$$ Here, $(1-\lambda_i u^t)^{\frac{1}{t}}$ denotes the $t$-th root of $(1-\lambda_i u^t)$ chosen such that the constant term is 1.
\end{lemma}
\begin{proof}
   The trace of $f^m$ can be expanded as 
   \begin{align*}
\exp\left({\sum_{m=1}^{\infty} \frac{\Tr(f^m)}{tm}u^{tm}}\right) &= \prod_{i=1}^M \exp\left(\sum_{m=1}^{\infty}\frac{\lambda_i^m}{tm}u^{tm}\right) = \prod_{i=1}^M \exp\left(\frac{1}{t}\sum_{m=1}^{\infty}\frac{(\lambda_iu^t)^m}{m}\right) \\
&= \prod_{i=1}^M \frac{1}{(1-\lambda_iu^t)^{\frac{1}{t}}}.
\end{align*}
   The choice of the $1/t$ root in $(1-\lambda_iu^t)^{\frac{1}{t}}$ is justified because the constant term in the exponential expression is 1.
   
\end{proof}

\subsubsection{Operations on Digraphs and Their Zeta Functions}
We now study two operations, the disjoint union and the tensor product, that allow us to construct complex graphs from simpler ones and yield important correlations between the zeta functions and eigenvalues of the resulting and original graphs.

In graph theory, the disjoint union combines two digraphs into a larger digraph containing all vertices and edges of the originals, with no additional edges between them. From a categorical viewpoint, the disjoint union of digraphs is their coproduct. Similarly, the tensor product of digraphs can be viewed as their product in the category of digraphs.

\begin{defn}
The \textbf{disjoint union} of two digraphs $X_1 = (\calV_1, \calE_1)$ and $X_2 = (\calV_2, \calE_2)$, denoted by $X_1\ \dot\cup\ X_2$, is the digraph with vertex set $\calV_1\ \dot\cup\ \calV_2$ and edge set $\calE_1\ \dot\cup\  \calE_2$. There are no additional edges in this digraph between vertices of $X_1$ and vertices of $X_2$.
\end{defn}

The zeta function of the disjoint union is simply the product of the zeta functions of the individual digraphs:

\begin{prop}
For any two finite digraphs $X_1$ and $X_2$, we have
$$Z_c(X_1 \dot\cup X_2,u) = Z_c(X_1,u)Z_c(X_2,u).$$
\end{prop}

\begin{proof}
The number of closed walks of length $l$ in the disjoint union $X_1 \dot\cup X_2$ is the sum of the number of closed walks of length $l$ in $X_1$ and in $X_2$. Hence, $N_c(X_1 \dot\cup X_2, l) = N_c(X_1, l) + N_c(X_2, l)$.

Thus, the zeta function of the disjoint union is 
\[
Z_c(X_1 \dot\cup X_2,u) = \exp\left(\sum_{l=1}^{\infty}\frac{N_c(X_1, l)}{l}u^l + \sum_{l=1}^{\infty}\frac{N_c(X_2, l)}{l}u^l\right)
= Z_c(X_1,u)Z_c(X_2,u).
\]
\end{proof}

The tensor product combines two digraphs into a new digraph whose vertices are ordered pairs of vertices from the original digraphs.

\begin{defn}
The \textbf{tensor product} of two digraphs $X_1$ and $X_2$, denoted by $X_1 \times X_2$, is the digraph whose vertex set is the Cartesian product of the vertex sets of $X_1$ and $X_2$. There is an edge from  the vertex $(x_1, x_2)$ to the vertex $(y_1, y_2)$ in $X_1 \times X_2$ if and only if there is an edge from $x_1$ to $y_1$ in $X_1$ and there is an edge from $x_2$ to $y_2$ in $X_2$.
\end{defn}

The eigenvalues of the tensor product digraph are products of the original eigenvalues:

\begin{prop}\label{prop:2.1.6 Eigenvalue of products}
For any two finite digraphs $X_1$ and $X_2$ with adjacency matrices $A_1$ and $A_2$, the eigenvalues of the tensor product $X_1 \times X_2$ are the products $\lambda_1 \lambda_2$ for all pairs of eigenvalues $\lambda_1$ of $A_1$ and $\lambda_2$ of $A_2$.
\end{prop}

\begin{proof}
By construction, the adjacency matrix of the tensor product $X_1 \times X_2$ is the Kronecker product of the adjacency matrices of $X_1$ and $X_2$. 

It is well known that the eigenvalues of the Kronecker product of two matrices are the products of the eigenvalues of the two matrices. Therefore, the eigenvalues of $X_1 \times X_2$ are the products of the eigenvalues of $X_1$ and $X_2$.
\end{proof}

\subsubsection{Cyclic $n$-Partite Digraphs and Their Properties}

In our study of digraphs $X_0(V)$ and $X_2(V)$, we encounter structures that can be understood as bipartite or cyclic $n$-partite digraphs. These digraphs are characterized by a partition of their vertex set into two or more disjoint subsets, with specific adjacency constraints.

\begin{defn}
A \textbf{bipartite digraph} is a digraph whose vertex set can be partitioned into two disjoint sets, denoted as $V_1$ and $V_2$, such that every edge connects a vertex in $V_1$ with a vertex in $V_2$.
\end{defn}

This definition can be generalized to cyclic $n$-partite digraphs.

\begin{defn}
A \textbf{cyclic $n$-partite digraph} is a digraph whose vertex set can be partitioned into $n$ disjoint sets, denoted as $V_1, ..., V_n$, such that every edge is from a vertex in $V_i$ to a vertex in $V_{i+1}$ for some $1 \leq i \leq n$, where $V_{n+1}$ is understood as $V_1$. Vertices in $V_i$ are said to be of type $i$.
\end{defn}

\begin{rem}
    Every digraph can be considered as a cyclic $1$-partite digraph, and a bipartite digraph is a cyclic $2$-partite digraph. Moreover, in the category of digraphs, cyclic $n$-partite digraphs can be viewed as objects with a map to a directed cycle of length $n$.
\end{rem}

We now analyze the closed walks and zeta function of cyclic $n$-partite digraphs. 

Firstly, the adjacency matrix $A$ of a cyclic $n$-partite digraph has the property that its $t$-th power $A^t$ has zero trace, unless $n$ divides $t$. This is due to the fact that in a cyclic $n$-partite digraph, a walk of length $t$ that starts and ends at the same vertex would mean that the vertex would have to change partition sets $t$ times, which is only possible if $n$ divides $t$.

Next, we consider closed walks starting from different types
For any cyclic $n$-partite digraph $X$, we denote the set of closed walks of length $m$ starting and ending at vertices of type $i$ as $CP(X,m,i)$, and its cardinality as $N^{(i)}_c(X,m)$. 

Remarkably, all the closed walks numbers $N^{(i)}_c(X,m)$ of a cyclic $n$-partite digraph are equal. This can be shown using shifting maps:

\begin{prop}\label{prop:2.1.13}
For a cyclic $n$-partite digraph $X$, we have $N^{(i)}_c(X,m) = N^{(j)}_c(X,m)$ and $N_c(X,m)=nN^{(i)}_c(X,m)$ for all $i,j \in \{1,\ldots,n\}$ and all $m \geq 1$. 
\end{prop}

\begin{proof}
The equality $N^{(i)}_c(X,m) = N^{(j)}_c(X,m)$ can be proven with shifting maps. For any integer $k$, define a shifting map $sh_k : CP(X,m,i) \rightarrow CP(X,m,(i+k) \mod n)$ that takes a closed walk $(v_0, v_1, \ldots, v_m = v_0)$ of type $i$ to the closed walk $(v_k, v_{k+1}, \ldots, v_{k+m} = v_k)$ of type $(i+k) \mod n$. 

The map $sh_k$ is a bijection between $CP(X,m,i)$ and $CP(X,m,(i+k) \mod n)$. Thus, $N^{(i)}_c(X,m) = N^{(i+k) \mod n}_c(X,m)$. Hence, $N^{(i)}_c(X,m) = N^{(j)}_c(X,m)$ by taking $k = j-i$. Then $N_c(X,m)=\sum_{k=1}^n N^{(k)}_c(X,m)=N^{(i)}_c(X,m)$.
\end{proof}

We now prove a useful tool for computing zeta functions, to be used in Section 3. 
\begin{prop}\label{prop:2.1.14 Main property}
Let $X$ be a cyclically $n$-partite digraph with adjacency matrix $A$. Let $B$ be the submatrix of $A^n$ formed by taking the rows and columns corresponding to vertices in $\calV_1$, and $\mu_1,...,\mu_M$ be the eigenvalues of $B$. Then $Z_c(X,u)=\prod_{i=1}^M \frac{1}{1-\mu_i u^n}$.
\end{prop}

\begin{proof}
First, we order the vertices $\mathcal{V}$ as $\mathcal{V}_1 \cup \dots \cup \mathcal{V}_n$, with all vertices in $\mathcal{V}_1$ appearing first, vertices in $\mathcal{V}_2$ appearing next, and so on. Then the matrix of the adjacency matrix $A$ is a block matrix. Then, the matrix $A^n$ is block diagonal with $B$ as the first block. This follows from the cyclically $n$-partite nature of $X$. 

We only need to consider closed walk numbers $N_c(nl)$ for $l\geq 1$ because any closed walk has length divisible by $n$. Now, applying Proposition \ref{prop:2.1.13} to our cyclic $n$-partite digraph $X$, we have $N_c(nl)= nN_c^{(1)}(nl)$ for each $l \geq 1$. This implies that $\Tr(A^{nl}) = n\Tr(B^l)$, because $B^l$ is the first block of $A^{nl}$ with trace exactly $N_c^{(1)}(nl)$.

The zeta function of $X$ is then given by 
\[
Z(X,u) = \exp\left(\sum_{l=1}^{\infty} \frac{\Tr(A^{nl})}{nl}u^{nl}\right)= \exp\left(\sum_{l=1}^{\infty} \frac{n\Tr(B^l)}{nl}u^{nl}\right) = \exp\left(\sum_{l=1}^{\infty} \frac{\Tr(B^l)}{l}u^{nl}\right).
\]
The formula
\[
Z_c(X,u) = \prod_{i=1}^M \frac{1}{1-\mu_i u^n}.
\]
then follows from Lemma \ref{lemma:2.3.3}.
\end{proof}

There is an interesting relation between the eigenvalues $\mu_i$ of $B$ and the eigenvalues $\lambda_i$ of $A$. 

\begin{prop}\label{prop:2.1.11}
For each nonzero eigenvalue $\mu_i$ of $B$, there correspond $n$ eigenvalues $\lambda_{i,1}, \ldots, \lambda_{i,n}$ of $A$, which are all the $n$-th roots of $\mu_i$. These account for all the nonzero eigenvalues of $A$.
\end{prop}

\begin{proof}
Comparing Proposition \ref{prop:2.1.14 Main property} with Theorem \ref{thm:2.1.1 Hashi}, we have 
$$Z(X,u) = \prod_{i=1}^M\frac{1}{(1-\mu_i u^n)}= \prod_{k=1}^{M'}\frac{1}{(1-\lambda_k u)},$$
where $\mu_i$ are the eigenvalues of $B$, and $\lambda_k$ are the eigenvalues of $A$.

For the product, we may only consider $\lambda_k$ and $\mu_i$ that is nonzero. The result then follows from taking reciprocals of the above identity, and a factorization of $(1-\mu_i u^n)$:
\[
 \prod_{k=1}^{M'} (1-\lambda_k u)= \prod_{i=1}^M{(1-\mu_i u^n)}=\prod_{i=1}^M \prod_{j=1}^n (1-\zeta_n^j (\mu_i)^{1/n} u). 
\]
\end{proof}

This property relates the nonzero eigenvalues of $A$ with those of $B$.

\subsection{Components of the Digraphs $X_0(V)$ and $X_2(V)$}

The digraphs $X_0(V)$ and $X_2(V)$, representing the collections of subspaces and directed flags in an $n$-dimensional vector space $V$, exhibit intriguing structures through their specific subgraphs. These subgraphs, which emerge by focusing on certain subsets of vertices, can provide valuable insights into the overall structure and connectivity of these digraphs. The standing assumption in this subsection is that $V=\bbF_q^n$ with $n\geq 2$.

\subsubsection{Dimension Patterns Along Walks}

In the digraphs $X_0(V)$ and $X_2(V)$, which represent the collections of subspaces and directed flags in an $n$-dimensional vector space $V$, we can observe specific patterns in the dimensions of these subspaces or the multi-dimensions of directed flags along their walks. 

\paragraph{Dimension Patterns in $X_0(V)$}
In $X_0(V)$, each walk alternates between two distinct dimensions of subspaces, $i$ and $n-i$. This pattern results from the adjacency condition in $X_0(V)$, which requires that the direct sum of the subspaces associated with two adjacent vertices equals $V$. For instance, let's consider a walk $(x_{W_0}, x_{W_1}, ..., x_{W_l})$ in $X_0(V)$. If $\dim(W_0) = i$, then $\dim(W_1) = n-i$, $\dim(W_2) = i$, and so on.

\paragraph{Multi-dimension Patterns in $X_2(V)$}
For $X_2(V)$, the pattern in the multi-dimensions of directed flags along a walk is described in the following proposition, as anticipated in Remark \ref{rem:1.3.6 on redefining edge}.

\begin{prop}\label{prop:2.2.1 on multidim trend}  
Let $V$ be an $n$-dimensional vector space, and let $F_0,F_1,...,F_l$ be a sequence of directed flags in $V$. Suppose the multi-dimension of $F_0$ is $(a, b)$.

    \begin{enumerate}
        \item  If there exists an edge from $x_{F_0}$ to $x_{F_1}$, then the multi-dimension of  $F_1$ is $f(a,b)$ where $f$ is defined on $\{(a,b):1\leq a,b\leq n-1, a\neq b\}$ as follows:
        \[
        f(a, b) = 
        \begin{cases} 
        (b, b-a) & \text{if } b > a \\
        (b, n+b-a) & \text{if } b < a
        \end{cases}
        \]
        \item Suppose $a < b$. Let $i = a$, $j = b - a$, $k = n - b$ such that $\mdim(F_0) = (i, i + j)$, then the multi-dimension of $F_i$ is determined by $i \pmod{6}$. The cyclic pattern of multi-dimensions is summarized in the following table:
        \begin{table}[h]
            \centering
            \begin{tabular}{|c|c|c|c|c|c|c|}
            \hline
            & $F_0$ & $F_1$ & $F_2$ & $F_3$ & $F_4$ & $F_5$ \\ 
            \hline
            Multi-dimensions & $(i,i+j)$ & $(i+j,j)$ & $(j,j+k)$ & $(j+k,k)$ & $(k,k+i)$ & $(k+i,i)$ \\ 
            \hline
            \end{tabular}
            \caption{Multi-dimensions of directed flags in a walk in $X_2(V)$}
            \label{table:1}
            \end{table}
    \end{enumerate}
\end{prop}

\begin{proof}
    \begin{enumerate}
    \item If an edge exists from $x_{F_0}$ to $x_{F_1}$, the conditions from Remark \ref{rem:1.3.6 on redefining edge} imply the existence of vector subspaces $W_0,W_1,W_2$ such that $F_0=(W_0,W_1)$ and $F_1=(W_1,W_2)$, with either $W_0\oplus W_2=W_1$ or ($W_0\cap W_2=W_1$ and $W_0+W_2=V$). Given $\dim(W_0)=a$ and $\dim(W_1)=b$, in the first case where $a<b$, $\dim(W_2)=b-a$; in the latter case where $a>b$, $\dim(W_2)=n-b-a$. Hence, the multi-dimension of $F_1$ aligns with $f(a,b)$ in both cases.

    \item By the first part, $\mdim(F_{r+1})=f(\mdim(F_r))$ for each $r$. The multi-dimensions of $F_0$ through to $F_5$ then follow the pattern outlined in the table. Applying $f$ to $\mdim(F_5) = (k+i, i)$ gives $(i,n+i-(i+k))=(i,i+j)$, which equals $\mdim(F_0)$. This confirms the cyclic pattern of the multi-dimensions.
    \end{enumerate}
\end{proof}

\begin{cor}\label{cor:2.2.2 to prop2.2.1}
   The function $f$ defined in the above proposition holds the property $f^6 = \id$.
\end{cor}
\begin{proof}
Let $(a,b)$ satisfy $1\leq a,b\leq n-1, a\neq b$. We can express $(a,b)$ as either $(i,i+j)$ or $(i+j,j)$ for some $i,j$. The previous proposition shows that $f^6(i,i+j)=(i,i+j)$, and $f^6(i+j,j)=(i+j,j)$. Therefore, $f^6=\id$.  
\end{proof}

\subsubsection{Subgraphs in $X_0(V)$ and $X_2(V)$}\label{sec:2.2.2}

Certain subgraphs within the digraphs $X_0(V)$ and $X_2(V)$ can be identified by focusing on specific dimensions of the subspaces or multi-dimensions of the directed flags. It turns out that they are connected components of these digraphs.

\paragraph{Subgraphs of $X_0(V)$}

For $X_0(V)$, we can define an equivalence class $[k] = \{k, n-k\}$ for each $k$ where $1 \leq k < n$, or $[k] = \{k\}$ if $k = \frac{n}{2}$. These equivalence classes categorize the set of possible dimensions $\{1,\ldots,n-1\}$.

Using these equivalence classes, 
we define subgraphs $X_0^{[k]}(V)$ which form the connected components of $X_0(V)$, as will be shown in Section \ref{sec:4.1.4}.
\begin{defn}
The subgraph $X_0^{[k]}(V)$ of $X_0(V)$ consists of vertices representing subspaces of $V$ of dimensions in the class $[k]$. The edges of $X_0^{[k]}(V)$ are the edges in $X_0(V)$ that connect these vertices.
\end{defn}

These subgraphs $X_0^{[k]}(V)$ are disjoint, as the alternating pattern of dimensions along walks in $X_0(V)$ ensures that there are no edges between vertices of dimensions not in the same class $[k]$. The union of these subgraphs forms the entire graph $X_0(V)$. Thus, the graph $X_0(V)$ can be expressed as a disjoint union of its subgraphs $X_0^{[k]}(V)$:

\[X_0(V) = \dot\bigcup_{[k]} X_0^{[k]}(V).\]

By the properties of the zeta function for disjoint union of digraphs, the zeta function of $X_0(V)$ is then the product of the zeta functions of its subgraphs $X_0^{[k]}(V)$:

\[Z_c(X_0(V),u) = \prod_{[k]} Z_c(X_0^{[k]}(V),u).\]

The subgraph $X_0^{[k]}(V)$, for $k \neq n/2$, is bipartite, split into two types of vertices: those representing subspaces of dimension $k$ and those of dimension $n-k$. Edges only exist between these differing types, illustrating the bipartite nature.

\paragraph{Subgraphs of $X_2(V)$}
In $X_2(V)$, we define an equivalence relation on the set of possible multi-dimensions $\{(a,b) : 1 \leq a,b \leq n-1, a \neq b\}$ using the function $f$ defined in Proposition \ref{prop:2.2.1 on multidim trend}.

Two multi-dimensions $(a,b)$ and $(a',b')$ are equivalent, denoted $(a,b) \sim (a',b')$, if and only if there exists an integer $k \geq 0$ such that $f^k(a,b) = (a',b')$. For each such $(a,b)$, we define the equivalence class $[(a,b)] := \{(a',b'): (a',b') \sim (a,b)\}$. 
We then define a subgraph $X_2^{[(a,b)]}(V)$ as:

\begin{defn}
The subgraph $X_2^{[(a,b)]}(V)$ of $X_2(V)$ includes vertices that correspond to directed flags with multi-dimensions in the class $[(a,b)]$. The edges of $X_2^{[(a,b)]}(V)$ are the edges in $X_2(V)$ that connect these vertices.
\end{defn}

Since the cyclic pattern of multi-dimensions along walks in $X_2(V)$ ensures that there are no edges between vertices of multi-dimensions not in the same class $[(a,b)]$, these subgraphs $X_2^{[(a,b)]}(V)$ are disjoint. Moreover, each $X_2^{[(a,b)]}$ is connected, as will be shown in Section \ref{sec:4.1.4}. This yields the following connected components decomposition of the digraph $X_2(V)$:

\[X_2(V) = \dot\bigcup_{[(a,b)]} X_2^{[(a,b)]}(V)\]

The zeta function of $X_2(V)$ is then the product of the zeta functions of its components:
\[Z_c(X_2(V),u) = \prod_{[(a,b)]} Z_c(X_2^{[(a,b)]}(V),u)\]

Now let's discuss the size of the equivalence class $[(a,b)]$. As demonstrated in the proof of Corollary \ref{cor:2.2.2 to prop2.2.1}, we can express $(a,b)$ as $(i,i+j)$ or $(i+j,j)$. By the computation in Proposition \ref{prop:2.2.1 on multidim trend}, where $k=n-i-j$, the class $[(a,b)]$ equals to the set $\{(i,i+j), (i+j,j), (j,j+k), (j+k,k), (k,k+i), (k+i,i)\}$. The cardinality of the equivalence class $[(a,b)]$ is 2 if $i=j=k$, and 6 otherwise.

The structure of the subgraph $X_2^{[(a,b)]}(V)$ depends on these cardinalities. When $i$, $j$, and $k$ are distinct, $X_2^{[(a,b)]}(V)$ is a cyclic 6-partite digraph, the vertices of which correspond to the multi-dimensions in $\{(i,i+j), (i+j,j), (j,j+k), (j+k,k), (k,k+i), (k+i,i)\}$, with edges connecting vertices in cyclic order.

However, when $i=j=k$, $X_2^{[(a,b)]}(V)$ becomes a bipartite digraph. Its vertices correspond to flags with multi-dimensions $(i,2i)$ and $(2i,i)$, with no edges connecting flags of the same type.

\section{Group Action on a Cyclic $n$-Partite Digraph}
We now turn our attention to the study of zeta functions of component subgraphs of the digraphs $X_0(V)$ and $X_2(V)$. It is important to note that the group $\GL(V)$ acts on these subgraphs and acts transitively on each type of vertices. In this section, we will explore the implications of this group action on a cyclic $n$-partite digraph.
\subsection{Group Actions, Averaging, and Module Structures}
In this subsection, we recall the $\bbC G$-module structure induced by a finite group $G$ acting on a set $X$. We further investigate the effects of averaging subsets of $G$ and how this relates to the module structure. This allows for an inclusion map from $\bbC [G.x]$ to $\bbC G$ and paves the way for the relative destination elements in the next subsection.

\begin{defn}
Consider a finite group $G$ acting on a set $X$. The action induces a structure of a $\bbC G$-module on the vector space $\bbC [X]$ over $\bbC$ with basis $X$ where 
\[
\bbC [X]= \left\{\sum_{x\in X} c_x x: c_x\in \bbC, \text{ only finitely many $c_x$ are nonzero} \right\},
\]
and $\bbC G$ acts on $\bbC [X]$ by linearly extending the $G$-action on $X$.
\end{defn}

For each $x \in X$, the stabilizer in $G$ is denoted by $G_x = \{g \in G : g.x = x\}$. By the theory of group action, $G.x\simeq G/G_x$ as $G$-sets, with $G$ acting on $G/G_x$ by left multiplication. This leads to the module isomorphism $\bbC[G.x]\simeq \bbC[G/G_x]$. To identify the module $\bbC[G.x]$ as a subset of the module $\bbC[G]$, we need to define the average of certain subsets of $G$.

\begin{defn}
Given a nonempty subset $S$ of $G$, the average of $S$, denoted by $e_S$, is defined as the element of $\bbC G$ given by: 
\[
e_{S} = \frac{1}{|S|}\sum_{s \in S} s.
\]
\end{defn}

This leads to a basic fact about averaging elements.

\begin{lemma}\label{lemma:mul_idem}
If $H$ and $K$ are subgroups of $G$, and $g\in G$, then $e_{HgK}=e_H\cdot g\cdot e_K$. In particular, $e_{HK}=e_H e_K$ and $e_H=e_H^2$.
\end{lemma}

\begin{proof}
Consider $e_H  g  e_K = \frac{1}{|H|}\sum_{h \in H}h \cdot g \cdot \frac{1}{|K|}\sum_{k \in K}k$. This can be rewritten as $\sum_{x\in G}c_x x$, where $c_x$ is the coefficient of $x$. Note that $\sum_{x\in G}c_x = 1$ and $c_x \neq 0$ only when $x \in HgK$. 

Now consider any $x,y \in HgK$. If $y=hxk$ for some $h \in H$ and $k \in K$, then $he_H  g  e_K  k = e_H  g  e_K$, implying $c_x=c_y$. Therefore, $c_x = \frac{1}{|HgK|}$ for all $x \in HgK$, and we have $e_H  g  e_K = \frac{1}{|HgK|}\sum_{x\in HgK}x = e_{HgK}$.
\end{proof}
\begin{thm}\label{thm:isomorphism_modules}
Consider a finite group $G$ acting on a set $X$. Let $H$ be a subgroup of $G$, and $x\in X$. Then the following hold:
\begin{enumerate}
    \item $\bbC G e_H$ is isomorphic to $\bbC[G/H]$ as $\bbC G$-modules via the correspondence $e_H\leftrightarrow H$.
    \item There is a $\bbC G$-module isomorphism $\alpha_x:\bbC[G.x]\to \bbC G e_{G_x}$ with $\alpha_x(x)= e_{G_x}$.
\end{enumerate}
\end{thm}

\begin{proof}
\begin{enumerate}
    \item The group $G$ acts on $\bbC G$ by left multiplication and the stabilizer of $e_H$ is $H$. Hence, we can identify the orbit $G.e_H$ with the set of left cosets $G/H$. This implies that $\bbC [G.e_H]\simeq \bbC[G/H]$ as $\bbC G$-modules via $e_H\leftrightarrow H$. 
    
    The set $G.e_H$ is $\bbC$-independent, and hence a $\bbC$-basis for the submodule $\bbC G e_H$ of $\bbC G$. Therefore, the $\bbC G$-module $\bbC[G.e_H]$ is naturally identified with $\bbC G e_H$.
    
    \item Similarly, for any $x\in X$, the stabilizer of $x$ is $G_x$, and we can identify the orbit $G.x$ with the set of left cosets $G/G_x$. Then $\bbC[G.x]$ is isomorphic to $\bbC[G/H]$ and $\bbC Ge_H$ by part 1. The $\bbC G$-module isomorphism from $\bbC[G.x]$ to $\bbC G e_{G_x}$ sends $x$ to the average $e_{G_x}$ of $G_x$.
\end{enumerate}
\end{proof}

\begin{prop}\label{prop:3.1.5 two alpha}
Assume the notations in the previous theorem, and let $y\in X$. Suppose there exists $g\in G$ satisfies $g.x=y$. 
Then $G.x=G.y$, and the isomorphism $\alpha_{y}\circ \alpha_x^{-1}:\bbC G e_{G_x}\to \bbC G e_{G_{y}}$ is equivalent to $r_{g^{-1}}$, the right multiplication by $g^{-1}$.
\end{prop}

\begin{proof}
Since $y=g.x$, we have $G_{y}=\Stab(y)=g G_x g^{-1}$. Then by Lemma \ref{lemma:mul_idem}, we obtain $e_{G_{y}}=e_{g G_x g^{-1}}=g e_{G_x} g^{-1}$. Hence, $\bbC G e_{G_{y}}=\bbC G g e_{G_x} g^{-1}=\bbC G e_{G_x} g^{-1}$. Therefore, the map $r_{g^{-1}}:\bbC G e_{G_x}\to \bbC G e_{G_{y}}$ is well-defined.

For any $h\in  G$, we have $\alpha_x^{-1}(h e_{G_x})=h.x=h g^{-1}.y$. Therefore $\alpha_{y}\alpha_x^{-1}(h e_{G_x})=h g^{-1} e_{G_{y}}$ $=h e_{G_x} g^{-1}=r_{g^{-1}}(h e_{G_x})$. Then $\alpha_{y}\alpha_x^{-1}$ and $r_{g^{-1}}$ agree on $G.e_{G_x}$. They also agree on $\bbC G e_{G_x}$ by extending by linearity. 
\end{proof}

\subsection{Group Action on Digraphs and Destination Elements}
We begin by defining the walk numbers and destination sums, which are concepts that allow us to count the number of walks of a given length between vertices and express the distribution of such walks.

\begin{defn}[Walk Number and Destination Sum]
Let $u, v$ be vertices in $X=(\mathcal{V},\mathcal{E})$, and let $l$ be a positive integer. The \textbf{walk number} (or path number) $P(u, v, l)$ is defined as the number of walks from $u$ to $v$ of length $l$ in $X$.

Let $u$ be a vertex in $X$, and let $l$ be a positive integer. The \textbf{destination sum} $D(u, l)$ is defined as the following formal sum in $\bbC \mathcal{V}$:
\[D(u, l) = \sum_{v \in \mathcal{V}} P(u, v, l). v.\]
\end{defn}

With this, we can define and study the adjacency operator.

\begin{defn}
Let $X$ be a digraph with adjacency matrix $A$ and vertex set $\mathcal{V}$. We define the adjacency operator of $X$ as a linear map $T:\mathbb{C}\mathcal{V} \rightarrow \mathbb{C}\mathcal{V}$ such that for each $u\in\mathcal{V}$, $T(u)=D(u,1)$, which is the sum of vertices $v$ directly reachable from $u$.
\end{defn}

Note that the coefficient of $v$ in $T(u)$ is then $P(u,v,1)$, which is equal to $A_{u,v}$ by definition of the adjacency matrix. With respect to the basis $\mathcal{V}$ of $\mathbb{C}\mathcal{V}$, the matrix of $T$ is then $A^T$, the transpose of $A$.

Then for a positive integer $m$, the matrix of $T$ with respect to $\mathcal{V}$ is then $(A^m)^T$. For any $u,v\in \mathcal{V}$, the coefficient of $v$ in $T^m(u)$ is then $(A^m)_{u,v}=P(u,v,m)$. Therefore, $T^m$ is the linear map such that $T^m(u)=\sum_{v\in \mathcal{V}} P(u,v,l)=D(u,l)$.

\paragraph{Partite Transitive Group Action}
Recall the definition of group action on digraphs.
\begin{defn}
A \textbf{group action} of a group $G$ on a simple digraph $X=(\mathcal{V},\mathcal{E})$ is a group action of $G$ on the vertex set $\mathcal{V}$, preserving the edge relations. Specifically, if $(u,v)\in \mathcal{E}$, then $(gu,gv)\in \mathcal{E}$.
\end{defn}

We will consider a special type of group action on cyclic $n$-partite digraph, called partite transitive action.
\begin{defn}
Let $X=(\mathcal{V},\mathcal{E})$ be a cyclic $n$-partite digraph with vertex parts $\mathcal{V}_1, ..., \mathcal{V}_n$. A \textbf{partite transitive action} of a group $G$ on $X$ is a group action of $G$ on the vertex set $\mathcal{V}$ such that for each $i$, $G.\mathcal{V}_i = \mathcal{V}_i$ and $G$ acts transitively on $\mathcal{V}_i$. In other words, the group action on vertices preserves type, and acts transitively on each type of vertices.
\end{defn}

For the remainder of this section, we make the standing assumption that $G$ is a finite group acting partite transitively on a cyclic $n$-partite digraph $X=(\mathcal{V},\mathcal{E})$. Note that in this case, $T^m$ maps each $\bbC\mathcal{V}_i$ to $\bbC \mathcal{V}_{i+m\pmod n}$.

The following proposition reveals the relationship between the group action and the walk characterizing concepts in our cyclic $n$-partite digraph:
\begin{prop}\label{prop:3.2.5 G-Properties}
Under our standing assumption, let $l$ be a positive integer. Suppose $u, v$ are vertices in $X$ with $\text{type}(v) = \text{type}(u) + l \pmod{n}$. The following properties hold:

1. $P(gu,gv,l)=P(u,v,l)$.

2. $D(gu,l)=gD(u,l)$.

3. The map $T^l$ is $G$-linear.
\end{prop}

\begin{proof}
1. If there exists a walk of length $l$ from $u$ to $v$, acting with $g \in G$ on this walk gives a walk of the same length from $gu$ to $gv$. Therefore, the walk number $P(u,v,l)$ is invariant under the action of the group $G$.

2. By definition and the first property, $$D(gu,l)=\sum_{v\in \mathcal{V}} P(gu,v,l).v=\sum_{v\in \mathcal{V}} P(gu,gv,l).gv=\sum_{v\in \mathcal{V}} P(u,v,l).gv=g.D(u,l).$$

3. For any $u\in \mathcal{V}$, $T^l(gu)=D(gu,l)=gD(u,l)=gT^l(u)$.
\end{proof}

\begin{prop}\label{prop:3.3.2} Suppose $l$ is a positive integers, and $i,j$ are types with $i+l=j\pmod n$. Suppose $u\in \mathcal{V}_i$ and $v\in \mathcal{V}_j$. Then there exists some unique $D(u,v,l)\in e_{G_u}\bbC Ge_{G_v}$ such that the following diagram commutes. Here, $D(u,v,l)_r$ means right multiplication by $D(u,v,l)$.
\begin{center}
\begin{tikzcd}[column sep=large]
\mathbb{C}\mathcal{V}_i \arrow[r, "\alpha_u"] \arrow[d, "T^l"'] & \mathbb{C}Ge_{G_u} \arrow[d, "{D(u,v,l)}_r"] \\
\mathbb{C}\mathcal{V}_j \arrow[r, "\alpha_v"'] & \mathbb{C}Ge_{G_v}
\end{tikzcd}
\end{center}
\end{prop} 

\begin{proof}
    Note that $T^l$ is a $\bbC G$-homomorphism, and $\alpha_u$ and $\alpha_v$ are $\bbC G$-isomorphisms. Therefore, the map $\alpha_v\circ T^l\circ \alpha_u^{-1}:\bbC Ge_{G_u}\to \bbC Ge_{G_v}$ is a $\bbC G$-homomorphism. 
    
    Recall that if $e$ is an idempotent in a ring $R$, and $M$ is an $R$-module, then each $R$-linear map $\phi:Re\to M$ is afforded by $m_r:a\mapsto am$ for some unique $m\in eM$, which equals $\phi(e)$. Apply this module theory proposition to the map $\alpha_v\circ T^l\circ \alpha_u^{-1}:\bbC Ge_{G_u}\to \bbC Ge_{G_v}$ yields the result. Moreover, $D(u,v,l)=\alpha_v\circ T^l\circ \alpha_u^{-1}(e_u)=\alpha_v\circ T^l(u)=\alpha_v(D(u,l))$.
\end{proof}

We take the expression $\alpha_v(D(u,l))$ as the definition of $D(u,v,l)$, and call it a relative destination element because it reflects the destinations of walks from $u$, relative to the vertex $v$. It is defined to make the diagram above commute.

\begin{defn}[Relative Destination Element]
Under our standing assumption, let $u, v$ be vertices in $X$, and let $l$ be a positive integer. If $\text{type}(v) = \text{type}(u) + l \pmod{n}$, then the \textbf{relative destination element} $D(u, v, l)$ is defined as $\alpha_v(D(u, l))$.
\end{defn}

We now prove a key property of relative destination elements that facilitates computation: they are multiplicative under the concatenation of walks.

\begin{prop}[Multiplicativity of Relative Destination Elements]
Let $u, v, w$ be vertices in $X$, and let $l_1, l_2$ be positive integers such that the types of the vertices are suitable for walks of lengths $l_1, l_2, l_1+l_2$ in $X$, i.e., $\text{type}(v) = \text{type}(u) + l_1 \pmod{n}$ and $\text{type}(w) = \text{type}(v) + l_2 \pmod{n}$. Then
\[D(u, v, l_1)D(v, w, l_2) = D(u, w, l_1 + l_2).\]
\end{prop}

\begin{proof}
This translates to $T^{l_1} \circ T^{l_2}=T^{l_1+l_2}$ via the commutative diagram.
\end{proof}

The main result of this section is the following proposition that relates the zeta function of graphs to the eigenvalues of certain action on modules, allowing us to apply tools in representation theory to the graph theory question of computing zeta functions. 
\begin{prop}\label{prop:3.3.4 Main prop on group zeta}
    Let $X$ be a finite cyclic $n$-partite digraph acted on partite-transitively by a finite group $G$. Let $u_1\in \mathcal{V}_1$ be a vertice of the first type, and $G_{u_1}$ be its stabilizer in $G$. The zeta function of $X$ equals \[Z_c(X,u) = \prod_{i=1}^m \frac{1}{1-\lambda_i u^n},\]
    where $\lambda_i$ are the eigenvalues of the right action of $D(u_1,u_1,n)$ on $\bbC G e_{G_{u_1}}$.
\end{prop}
\begin{proof}
This follows from Proposition \ref{prop:2.1.14 Main property} and the observation that $\Tr(D(u_1,u_1,n)_r)=\Tr(T^n|_{\bbC \mathcal{V}_1})=\Tr(A^n|_{\bbC \mathcal{V}_1})$.
\end{proof}

\subsection{Application to the subgraphs $X_0^{[k]}(V)$ and $X_2^{[(a,b)]}(V)$}\label{sec:3.4}

This section applies prior discussions to the subgraphs $X_0^{[k]}(V)$ and $X_2^{[(a,b)]}(V)$ for $V=\bbF_q^n$. In both cases, the general linear group $G=\GL(V)$ acts on $X_0$ and $X_2$ by $g.x_W=x_{g.W}$ and $g.x_{F}=x_{g.F}$, respectively. Here, if $F=(W_1,W_2)$, then $gF=(gW_1,gW_2)$. The action of $G$ respects the dimensions or multi-dimensions and operates transitively on vertices of similar type, as characterized by aforementioned dimensions or multi-dimensions. Therefore, Proposition \ref{prop:3.3.4 Main prop on group zeta} is applicable and can be used to determine the zeta functions $Z(X_0^{[k]}(V),u)$ and $Z(X_2^{[(a,b)]}(V),u)$:

\begin{itemize}
    \item For $X_0^{[k]}(V)$ with $k\neq n/2$, the graph is bipartite. The vertices in the first part are $x_W$ with $W$ being $k$-dimensional, and the vertices in the second part are $x_W$ with $W$ being $(n-k)$-dimensional. The group $G$ acts on both partite sets transitively. Let $W_1$ be a $k$-dimensional vector subspace, and let $P_1=G_{x_{W_1}}=\Stab(W_1)$ be its stabilizer.  By proposition \ref{prop:3.3.4 Main prop on group zeta}, to find the zeta function of $Z_c(X_0^{[k]}(V),u)$, it suffices to find the eigenvalues of $D(x_{W_1},x_{W_1},2)_r$ on $CGe_{P_1}$. We will show in Section \ref{sec:7} that all these eigenvalues are a non-negative integer power of $q$.
    
    \item For $X_0^{[k]}(V)$ with $k=n/2$, the graph is 1-partite. The group $G$ acts transitively on all vertices. Let $W_1$ and $P_1$ be defined as in the previous case. 
    Again, to find the zeta function of $Z_c(X_0^{[k]}(V),u)$, it suffices to find the eigenvalues of $D(x_{W_1},x_{W_1},1)_r$ on $CGe_{P_1}$. We will show in Section \ref{sec:8} that all these eigenvalues are of the form $\pm q^m$ for some $m\in \bbZ_{\geq 0}$. 
    \item For $X_2^{[(a,b)]}(V)$ with $[(a,b)]$ consisting of 6 elements, the digraph is cyclically 6-partite. The group $G$ acts transitively on each part. Choose $F_1$ with multi-dimension $(i,i+j)$ and let $P_1$ be its stabilizer. We now need to find the eigenvalues of $D(x_{F_1},x_{F_1},6)_r$ on $CGe_{P_1}$. We will show in Section \ref{sec:7} that all these eigenvalues are a non-negative integer power of $q$.
    \item For $X_2^{[(a,b)]}(V)$ with $[(a,b)]=\{(n/3,2n/3), (2n/3,n/3)\}$, the digraph is bipartite according to multi-dimensions. The group $G$ acts transitively on vertices of each type. Choose a flag $F_1$ of multi-dimension $(n/3, 2n/3)$, and let $P_1$ be its stabilizer. It now remains to find the eigenvalues of $D(x_{F_1},x_{F_1},2)_r$ on $CGe_{P_1}$. We will show in Section \ref{sec:8} that all these eigenvalues are of the form $\zeta q^m$ for some $m\in \bbZ_{\geq 0}$, where $\zeta^3=1$. 
\end{itemize}

To ease notations, we write $W$ for $x_W$ to represent a vertex in $X_0(V)$, and $F$ for the vertex $x_F$ in $X_1(V)$. In the following section, we will compute the elements $D(W_1,W_1,2)$, $D(W_1,W_1,1)$, $D(F_1,F_1,6)$, and $D(F_1,F_1,2)$. The remainder of the paper will be devoted to clarifying their action on the corresponding module $CGe_{P_1}$.

\section{Computing Relative Destination Elements} \label{sec:4}
This section aims to explicitly compute the elements $D(W,W;m)$ and $D(F,F;m)$ that capture the counts of cycles in the digraphs $X_0(V)$ and $X_2(V)$, respectively.

We first introduce some preliminary results on Bruhat decomposition that will be useful. Let $G=\GL_n(\mathbb{F}_q)$ and let $B=\UT_n(\mathbb{F}_q)$ be the subgroup of upper triangular matrices. Let $\overline{B}$ denote the subgroup of lower triangular matrices, which is a conjugate of $B$ via the longest element $w_0\in S_n$ with $w_0(i)=n+1-i$ for each $i$.

Here, we identify $S_n$ as a subgroup of permutation matrices in $G$ in the following way: each $w\in S_n$ is identified with a matrix which permutes the basis elements $e_1,...,e_n$ by letting $w$ act on the indices (i.e. $w e_i=e_{w(i)}$). Then the longest element $w_0\in S_n$ is identified with the anti-diagonal matrix $(\delta_{i+j,n+1})$.

\begin{thm}[Bruhat Decomposition]\label{theorem:4.0.1, BruhatD}
With the above notations, the following decompositions hold:
\begin{enumerate}
\item $G=BS_nB$
\item $G=B'S_nB''$ for any subgroups $B',B''$ that are conjugates of $B$ via elements of $S_n$.
\item $G=B\overline{B}B$.
\end{enumerate}
\end{thm}
\begin{proof}
The first statement is a standard result in linear algebraic groups \cite[Theorem~65.4]{curtis_reiner1986}. For the second statement, suppose $B'=w_1Bw_1^{-1}$ and $B''=w_2Bw_2^{-1}$ for some $w_1,w_2\in S_n$. Then $B'S_nB''=w_1Bw_1^{-1}S_nw_2Bw_2^{-1}=w_1BS_nBw_2^{-1}=w_1Gw_2^{-1}=G$. The third statement will be proved in Corollary \ref{cor:5.3.4}.
\end{proof}
In particular, $G=BS_n \overline{B}$.

Here is a lemma that shows how walk numbers depend only on certain cosets in a digraph under group action, which will be useful in our  computation.
\begin{lemma}\label{lemma:4.0.2 PathNumbers}
Let $G$ act on a digraph $X$, and let $u,v$ be vertices in $X$. Let $l$ be a positive number. Then, for $g \in G$, the walk number $P(u,g.v,l)$ only depends on the coset $G_u g G_v$. That is, if $h=agb$ for some $a \in G_u$ and $b \in G_v$, then $P(u,g.v,l)=P(u,h.v,l)$.
\end{lemma}

\begin{proof}
By Proposition \ref{prop:3.2.5 G-Properties}, $P(u,h.v,l)=P(u,agb.v,l)=P(a^{-1}.u, gb.v,l)$. Since $a\in G_u$ and $b \in G_v$, $a^{-1}.u=u$ and $b.v=v$. Then $P(u,h.v,l)=P(u,g.v,l)$.
\end{proof}

\subsection{Relative Destination Element \texorpdfstring{$D(W,W;m)$}{D(W,W;m)} for \texorpdfstring{$X_0(V)$}{X0(V)}}\label{sec:4.1}
This section aims to explicitly compute the elements $D(W,W;m)$ in $\bbC G$.

\subsubsection{Calculation of \texorpdfstring{$D(W_0,1)$}{D(W0,1)} and \texorpdfstring{$D(W_0,W_1,1)$}{D(W0,W1,1)}}\label{sec:calculationDW}
To perform the calculation, we first recall that $V=\bbF_q^n$ with a basis $e_1, \ldots, e_n$. Consider the nonzero proper subspaces $W_0=\angles{e_1,...,e_i}$ and $W_1=\angles{e_{i+1},...,e_n}$ of $V$. We denote $P_0=\Stab(W_0)$ and $P_1=\Stab(W_1)$. Note that $P_0\supseteq B$ and $P_1\supseteq \overline{B}$. Therefore, by Theorem \ref{theorem:4.0.1, BruhatD}, $G=BS_n\overline{B}=P_0 S_n P_1$.

\begin{prop}\label{prop:4.1.1}
With the notation as above, the element $D(W_0,1)$ can be expressed as the following formal sum over all left cosets $gP_1$ in $P_0P_1$, $$D(W_0,1)=\sum_{gP_1\in P_0P_1}gW_1.$$
\end{prop}

\begin{proof}
By definition, $D(W_0,1)=\sum_{W'} P(W_0,W',1)W'$. For dimension reasons, all the $W'$ with nonzero coefficient lies in $G.W_1$. It then suffices to consider the terms $P(W_0,gW_1,1)$ for $g\in G$. By definition of walk in $X_0$, the walk number $P(W_0,gW_1,1)$ equals 1 when $W_0\oplus gW_1=V$, and 0 otherwise. 

By Lemma \ref{lemma:4.0.2 PathNumbers}, the number $P(W_0,gW_1,1)$ only depends on the double coset $P_0gP_1$. Since $G=P_0 S_n P_1$, we may assume $g\in P_0 wP_1$ for some $w\in S_n$. Then $P(W_0,gW_1,1)=P(W_0,wW_1,1)$. 

 Note that $W_0\oplus 1.W_1=V$. This implies $P(W_0,gW_1,1)=P(W_0,1.W_1,1)=1$ for $g\in P_0P_1$.
Conversely, suppose $P(W_0,gW_1,1)=1$. Then $P(W_0,wW_1,1)=1$ and so $W_0\oplus wW_1=V$.  In particular, $W_0\cap wW_1=0$, which implies that $\{1,...,i\}\cap w\{i+1,...,n\}=\varnothing$. Then $w \in S_{\{1,...,i\}}\times S_{\{i+1,...,n\}}\subseteq P_0$. Thus, $g\in P_0w P_1=P_0P_1$. It follows that $P(W_0,gW_1,1)=1$ if and only if $g\in P_0 P_1$. Therefore, $D(W_0,1)=\sum_{gP_1\in P_0P_1}gW_1$.
\end{proof}

\begin{cor}\label{cor:4.1.2}
The relative destination element $D(W_0, W_1; 1)$ is given by 
\[D(W_0, W_1; 1) = \frac{|P_0 P_1|}{|P_1|} e_{P_0 P_1}=\frac{1}{|P_1|} \sum_{g\in P_0P_1}g.\]
\end{cor}

\begin{proof}
By definition of the relative destination element, 
\[D(W_0,W_1,1) =\alpha_{W_1}(D(W_0,1))=\alpha_{W_1}(\sum_{gP_1\in P_0P_1}gW_1)= \frac{1}{|P_1|}\sum_{g \in P_0 P_1} g e_{P_1}=\frac{|P_0P_1|}{|P_1|}e_{P_0P_1}e_{P_1}.\]

Note that $e_{P_0P_1}=e_{P_0}e_{P_1}$, and is invariant under right multiplication by $e_{P_1}$. Then  $D(W_0,W_1,1) =\frac{|P_0P_1|}{|P_1|}e_{P_0P_1}e_{P_1}=\frac{|P_0P_1|}{|P_1|}e_{P_0P_1}$, which equals $\frac{1}{|P_1|} \sum_{g\in P_0P_1}g$.
\end{proof}

\begin{defn}\label{def:element_a}
Given two subgroups $H_1$ and $H_2$ of $G$, we define another average element $a_{H_1 H_2}$ of the double coset $H_1 H_2$ as:
\[a_{H_1 H_2} = \frac{1}{|H_2|} \sum_{x \in H_1 H_2} x.\] 
\end{defn}

Note that $a_{H_1 H_2}=\frac{|H_1 H_2|}{|H_2|}e_{H_1 H_2}$.

By the corollary, we can now express the relative destination element $D(W_0, W_1; 1)$ as $D(W_0, W_1; 1) = a_{P_0 P_1}$. We denote $a_{P_0 P_1}$ as $a_{0,1}$ for simplicity. A similar argument shows that $D(W_1,W_0;1)=\alpha_{W_0}(D(W_1,1))$ can be expressed as the average element $a_{P_1 P_0}$, which we denote as $a_{1,0}$. Then $$D(W_0,W_0,2)=D(W_0,W_1;1)D(W_1,W_0;1)=a_{0,1}a_{1,0}.$$

\subsubsection{Simplification of the expression \texorpdfstring{$a_{0,1}a_{1,0}$}{a0,1 a1,0}}\label{sec:4.1.2simplify}

We aim to find the eigenvalues associated with the action of $(a_{0,1}a_{1,0})_r$ on the $\mathbb{C}G e_{P_0}$ module. As a first step, we simplify the expression $a_{0,1}a_{1,0}$ using the longest element in $S_n$. Let $w_0$ denote the longest element in $S_n$. Then $\overline{B} = w_0 B w_0$.

The Levi subgroups of $P_0$ and $P_1$ are both $L = \GL(\langle e_1,...,e_i\rangle) \times \GL(\langle e_{i+1},...e_n\rangle)\subseteq G$. This results in $P_0 = L B = B L$ and $P_1 = \overline{B} L = L \overline{B}$. By Lemma \ref{lemma:mul_idem}, $e_{P_0}=e_Be_L=e_Le_B$ and 
$e_{P_1}=e_{\overline{B}} e_L=e_Le_{\overline{B}}$, 
and we can express $e_{P_0} e_{P_1} e_{P_0}$ as $e_L e_B e_{\overline{B}} e_B $, which further equals $ e_L e_B w_0 e_B w_0 e_B$.

Recall that the product $a_{0,1} a_{1,0}$ equals $\frac{|P_0 P_1|^2}{|P_0| |P_1|} e_{P_0} e_{P_1} e_{P_0}.$ Then $a_{0,1} a_{1,0}=\frac{|P_0 P_1|^2}{|P_0| |P_1|} e_L e_B e_{\overline{B}} e_B$. As multiplication by $e_L$ on the right on $\mathbb{C}G e_{P_0}$ is an identity operation, the action of $D(W_0, W_0, 2)_r$ on $\mathbb{C}G e_{P_0}$ is equivalent to the right multiplication by $\frac{|P_0 P_1|^2}{|P_0| |P_1|}e_B e_{\overline{B}} e_B$ on $\mathbb{C}G e_{P_0}$. 

Therefore, we are primarily interested in the right multiplication by $e_B e_{\overline{B}} e_B=(e_{B} w_0 e_{B})^2$ on $\mathbb{C}G e_{P_0}$. Developing tools for this question will be the focus of Sections 5-7.

\subsubsection{Special Case \texorpdfstring{$D(W,W;1)$}{D(W,W;1)}}
\label{sec:special_caseDW4.1.3}

In the special case where $i = \frac{n}{2}$, the expression $D(W_0, W_0; 1)$ is meaningful. In this case, $D(W_0, W_0; 1)=\alpha_{W_0}(D(W_0, 1))=\alpha_{W_0}\circ\alpha^{-1}_{W_1}(D(W_0,W_1, 1)).$ Because $W_0=w_0W_1$, by Proposition \ref{prop:3.1.5 two alpha}, we have$D(W_0, W_0; 1)=D(W_0,W_1, 1)w_0^{-1}$, which further simplifies as $D(W_0,W_1, 1)w_0^{-1}=a_{P_0 P_1} w_0 = \frac{|P_0 P_1|}{|P_1|} e_{P_0} e_{P_1} w_0 = \frac{|P_0 P_1|}{|P_1|} e_{P_0} w_0 e_{P_0}$.
We will study the right  multiplication of $e_{P_0} w_0 e_{P_0}$ on $\bbC Ge_{P_0}$ in Section 8.

\subsection{The digraph \texorpdfstring{$X_2(V)$}{X2(V)} and the elements \texorpdfstring{$D(F,F;m)$}{D(F,F;m)}}
\label{sec:4.2}

In this section, we will consider the digraph $X_2(V)$, and carry out explicit calculations of the elements $D(F,F;m)$ in $\bbC G$.

Consider a partition of $n$ into three positive integers $i$, $j$, and $k$, i.e., $n=i+j+k$. Let $I=\{1,...,i\}, J=\{i+1,...,i+j\}, K=\{i+j+1,...,i+j+k=n\}$. We denote by $V_I$ the vector space spanned by the first $i$ basis vectors, i.e., $V_I=\angles{e_1,...,e_i}$. Similarly, we define $V_J=\angles{e_{i+1},...,e_{i+j}}$, and $V_K=\angles{e_{i+j+1},...,e_n}$.

Let $F_0=(V_I,V_I+V_J)$ be a directed flag with multi-dimension $(i,i+j)$. We can also define a series of flags $F_1$ to $F_5$ with suitable multi-dimensions as follows:
\begin{align*}
F_0 &= (V_I,V_I+V_J),\  F_1 = (V_I+V_J,V_J),\  F_2  = (V_J,V_J+V_K), \\
F_3 & = (V_J+V_K,V_K), F_4  = (V_K,V_K+V_I), 
F_5  = (V_K+V_I,V_I).
\end{align*}

We also define $P_0$ to $P_5$ as the stabilizers of the flags $F_0$ to $F_5$, i.e., $P_r=\Stab(F_r)$ for $r=0,...,5$. Let $F_6=F_0$, and $P_6=P_0$. Again, each $P_i$ contains a conjugate of $B$, and $G=P_i S_n P_j$ for any $i,j$ by Bruhat Decomposition \ref{theorem:4.0.1, BruhatD}.

\subsubsection{Calculation of \texorpdfstring{$D(F_r,t)$}{D(Fr,t)} and \texorpdfstring{$D(F_r,F_s,t)$}{D(Fr,Fs,t)}}
\label{sec:4.2.1}

We now aim to calculate the destination sums $D(F_r,t)$ and $D(F_r,F_s;t)$. We have the following general proposition expressing expressing these destination elements. Recall the definition of $a_{H_1H_2}$ in Definition \ref{def:element_a}. 

\begin{prop}\label{prop:4.2.1 on D(F,F,t)}
Let $t\in \{1,2,3\}$ and $r,s\in \{0,1,...,5\}$ such that $r+t\equiv s \pmod 6$. Then the destination sum is given by $D(F_r,t)=\sum_{gP_s\in P_rP_s} gF_s$, and $D(F_r,F_s;t)=a_{P_rP_s}$. 
\end{prop}

\begin{proof} The proof is divided into two parts. Firstly, we discuss some conditions for the statement on $D(F_r,t)$ and $D(F_r,F_s;t)$ to hold. Secondly, we verify that these conditions hold for all $r,s,t$ here.

\textbf{Conditions (a)-(c)}\quad
To compute $D(F_r,t)$, we only need to find $D(F_r,gF_s,t)$ for $g\in G=\GL_n(\mathbb{F}_q)$. This can be done by verifying the following conditions: 
\begin{enumerate}[label=(\alph*)]
    \item $P(F_r,gF_s,t)=0$ or 1 for any $g\in G$.
    \item $P(F_r,F_s,t)=1$
    \item If $w\in S_n$ satisfies $P(F_r,wF_s,t)=1$, then $w\in P_r$.
\end{enumerate}

If these conditions are met, then by condition (b), $P(F_r,gF_s,t)=P(F_r,F_s,t)=1$ for $g\in P_0P_1$ because $P(F_r,gF_s,t)$ depends only on the coset $P_0 gP_1$. 

Conversely, suppose $P(F_r,gF_s,t)=1$. Note that there must exist some $w\in S_n$ such that $g\in P_r wP_s$. Then, $P(F_r,wF_s,t)=1$, and by condition (c), $w\in P_r$, and so $g\in P_rP_s$. This implies that $P(F_r,F_s,t)=1$ if and only if $g\in P_rP_s$.

Hence, $D(F_r,t)=\sum_{gP_s\in P_rP_s} gF_s$. The statement $D(F_r,F_s;t)=a_{P_r,P_s}$ follows from  with an appropriate adaptation of the computation in Corollary \ref{cor:4.1.2}. 

\vspace{12pt}

\textbf{Verification of The Condition}\quad
To verify conditions (a), (b), and (c) for various $r,s,t$, we find that $P(F_r,F_s,t)\geq 1$ is always satisfied because $(F_0,F_1,...,F_5,F_0)$ is a cyclic walk in $X_2(V)$. Therefore, condition (a) implies condition (b). 

We now show that condition (a) and (c) are satisfied. All indices are taken in $\mathbb{Z}/6\mathbb{Z}$. It suffices to prove (a) and (c) for $t=3$: For $t=1$ or 2, $P(F_r,gF_{r+3},3)\geq P(F_r,gF_{r+t},t)P(gF_{r+t},gF_{r+3},3-t)$ by concatenation of walks. Since $P(gF_{r+t},gF_{r+3},3-t)=P(F_{r+t},F_{r+3},3-t)\geq 1$, $P(F_r,gF_{r+t},t)\leq P(F_r,gF_{r+3},3)$, which is less than or equal to 1 if (a) holds for $t=3$.
Similarly, if $P(F_r,wF_{r+t},t)=1$, then $P(F_r,wF_{r+3},3)\geq 1$. Then by (a), (c) for $t=3$, we have $w\in P_r$.

Now, we prove (a),(c) when $t=3$. If we rotate the indices $(i,j,k)$, we replace $r$ by $r+2$ or $r+4$. So it suffices to prove $r=0$ and $r=1$ cases. 

Consider $r=0$.
Consider two flags $F'_0=(W_0,W_1)$ and $F'_3=(W_3,W_4)$ with dimensions $\dim(W_0)=i$, $\dim(W_1)=i+j$, $\dim(W_3)=j+k$ and $\dim(W_4)=k$. Assume there is a walk of length 3 in the digraph $X_2$ from $F'_0$ to $F'_3$. Then, there must exist a subspace $W_2$ of dimension $j$ such that $(W_0,W_1,W_2,W_3,W_4)$ forms a geodesic path in the building $X(V)$. This implies that $W_0\oplus W_2=W_1$, $W_2\oplus W_4=W_3$, $W_1\cap W_3=W_2$, and $W_1+W_3=V$ by Proposition \ref{prop:1.2.5 on geodesic in X(V)}. Hence, such a subspace $W_2$ is unique and equals $W_1\cap W_3$. In particular, the walk in $X_2$ is unique. This proves $P(F_0,gF_3,3)\leq 1$ as required.

Moreover, we note that $W_0\cap W_3=W_0\cap W_1\cap W_3 =W_0\cap W_2=0$, and $W_1\cap W_4=W_1\cap W_3\cap W_4=W_2\cap W_4=0$.
Suppose now $P(F_0,wF_3;1)=1$. Note that $F_0=(V_I,V_I+V_J), wF_3=(w(V_J+V_K),wV_K)$. Therefore, $V_I\cap w(V_J+V_K)=0$, and $wV_K\cap (V_I+V_J)=0$. In particular, $w(J\cup K)\cap I=0$, and $w(K)\cap (I\cup J)=0$. Hence, $w$ preserves index sets $I,J,K$. Then $w\in S_I\times S_J\times S_K\subseteq P_0$. Hence (c) is proved.

Now we consider $r=1$. The situation is analogous. Suppose $F_1'=(W_1,W_2)$, and $F_4'=(W_4,W_5)$ with multi-dimensions matching with $F_1$ and $F_4$ respectively. Then a length 3 walk from $F_1'$ to $F_4'$ necessitates the existence of a $W_3$ with $(W_1,W_2,W_3,W_4,W_5)$ satisfying the conditions in Proposition \ref{prop:1.2.5 on geodesic in X(V)}. Then $W_3=W_2\oplus W_4$ is unique, which proves (a), the uniqueness of walks. Again, $W_1\cap W_4=0$ and $W_2\cap W_5=0$ by the same argument as before. Then if $P(F_1,wF_4,3)=1$, then $(V_I+V_J)\cap wV_K=0$ and $V_J\cap w(V_I+V_K)=0$. In particular, $w$ stabilizes $K$ and $J$, and thus $w\in S_I\times S_J\times S_K\subseteq P_0$, proving (c) and concluding our proof.
\end{proof}

\begin{rem}
Since $D(F_0,F_2;2)=D(F_0,F_1;1)D(F_1,F_2;1)$, we have $a_{P_0, P_2}=a_{P_0,P_1}a_{P_1P_2}$. Thus, we find that \[\frac{|P_0P_1|}{|P_1|}\frac{|P_1P_2|}{|P_2|}e_{P_0P_1}e_{P_1P_2}=\frac{|P_0P_2|}{|P_2|}e_{P_0P_2}.\] Taking the sum of coefficients gives $\frac{|P_0P_1||P_1P_2|}{|P_1|}=|P_0P_2|$. Moreover, we find that $e_{P_0P_2}=e_{P_0P_1}e_{P_1P_2}$. Thus, $P_0P_1P_2=P_1P_2$, and specifically $P_1\subseteq P_0 P_2$. Similarly, $a_{P_0 P_3}=a_{P_0P_2}a_{P_2P_3}$ implies  $P_2\subseteq P_0P_3$ and a similar result on coefficients.
\end{rem}

Now, we observe that $D(P_0,P_0;6)=D(P_0,P_3;3)D(P_3,P_0;3)=a_{P_0,P_3}a_{P_3P_0}$. Let's denote $L=\GL(V_I)\times\GL(V_J)\times\GL(V_K)$. Here, $L$ is the Levi subgroup of both $P_0$ and $P_3$, with $P_0=BL=LB$ and $P_3=\overline{B}L=L\overline{B}$. The product $a_{P_0,P_3}a_{P_3P_0}$ can be represented as $C e_L e_Be_{\overline{B} }e_B$, where $C=\frac{|P_0P_3|^2}{|P_0||P_3|}$.

For the right multiplication of $a_{P_0P_3}$ on $\bbC Ge_{P_0}$, $(e_L)_r$ is again the identity map. We are interested in the trace of $(e_Be_{\overline{B} }e_B)_r$ on $\bbC Ge_{P_0}$. This leads to a problem that is essentially the same as the one discussed at the end of Section \ref{sec:4.1.2simplify}, only with a different definition for the parabolic subgroup $P_0$. We will discuss this in the following sections.

\subsubsection{Special Case \texorpdfstring{$D(F_0,F_0;2)$}{D(F0,F0;2)}}
\label{sec:4.2.2}
Lastly, we consider the special case where $i=j=k=\frac{n}{3}$. Only in this case does $D(F_0,F_0;2)$ make sense in terms of multi-dimensions.

Consider the element $w_1=(1,i+1,2i+1)(2,i+2,2i+2)...(i,2i,3i)\in S_n$. In this case, $w_1F_0=F_2$ and $P_2=w_1P_0w_1^{-1}$. The element $D(F_0,F_0;2)$ equals $\alpha_{F_0}(D(F_0,2))$ which can be computed using Proposition \ref{prop:3.1.5 two alpha}: $\alpha_{F_0}(\alpha_{F_2}^{-1}(a_{P_0P_2}))=a_{P_0P_2}w_1$, because $F_0=w_1^{-1}F_2$.

On the other hand, $a_{P_0P_2}=|P_0P_2|/|P_2|e_{P_0}e_{P_2}=(|P_0P_2|/|P_2|)e_{P_0}w_1e_{P_0}w_1^{-1}$. Then $D(F_0,F_0;2)= (|P_0P_2|/|P_2|)e_{P_0}w_1e_{P_0}$.
Note that this element also equals $a_{P_0 w_1 P_0}=\frac{1}{|P_0|}\sum_{x\in P_0 w_1 P_0}x$, because the sum of coefficients in $a_{P_0 w_1 P_0}$ is $\frac{|P_0 w_1 P_0|}{|P_0|}=\frac{|P_0P_2|}{|P_0|}=\frac{|P_0P_2|}{|P_2|}$. 
Therefore, \[D(F_0,F_0;2)=a_{P_0 w_1 P_0}=\frac{|P_0P_2|}{|P_0|}e_{P_0}w_1e_{P_0}.\]
We are interested in the eigenvalues of $(e_{P_0}w_1e_{P_0})_r: \bbC Ge_{P_0}\to\bbC Ge_{P_0}$. This topic will be addressed in Section 8.

\subsubsection{Connectivity of the subgraphs}\label{sec:4.1.4}
Now we discuss the connectivity of the subgraphs $X_0^{[i]}(V)$ and $X_2^{[(a,b)]}(V)$, as promised in Section \ref{sec:2.2.2}
\begin{prop}
Definition as in Section \ref{sec:2.2.2}. The graphs $X_0^{[i]}(V)$ and  $X_2^{[(a,b)]}(V)$ are connected.
\end{prop}
\begin{proof}
For the $X_0^{[i]}(V)$ case, we retain the notation is Section \ref{sec:4.1}.
Note that $G=B\overline{B}B=P_0P_1P_0$ by Theorem \ref{theorem:4.0.1, BruhatD}. Then any $g\in G$ can be expressed as $g=g_1 p_0$ for some $g_1\in P_0P_1$, and $p_0\in P_0$. By Proposition \ref{prop:4.1.1}, $W_0$ is connected with $g_1W_1$. Note that $W_1$ is connected with $W_0=p_0W_0$. Therefore, $g_1W_1$ is connected to $g_1p_0W_0=gW_0$, and so $W_0$ and $gW_0$ are on the same connected components. This shows that dimension $i$ subspaces are in the same connected components. It remains to note that each dimension $n-i$ subspace is connected to a subspace of dimension $i$.

The case for $X_2^{[(a,b)]}(V)$ is analogous.
\end{proof}

\section{Unipotent Representations of \texorpdfstring{$\bbC G$}{CG}}
\label{sec:5}
In this section, we delve into the concept of unipotent representations of the group algebra $\bbC G$, where $G$ is the general linear group $\GL_n(\bbF_q)$. We only consider left modules unless otherwise stated. 

We begin by setting some notations. Let $B=\UT_n(\bbF_q)$ denote the group of upper triangular matrices in $G$ and $e_B$ the idempotent associated to $B$ in $\bbC G$, as defined in Section 2. We also denote by $H$ the Hecke algebra $e_B \bbC G e_B$. 

Following Andrews \cite{Andrews2018}, 
we define irreducible unipotent modules as
\begin{defn}
An irreducible module of $\bbC G$  is called unipotent if it is isomorphic to a submodule of the induced representation $\Ind_{B}^G (1_B)$ of the trivial representation $1_B$ on $B$.
\end{defn}

Note that the module $\Ind_{B}^G (1_B)=\bbC G\otimes_{\bbC B} \bbC$, where $B$ acts trivially on $\bbC$. It is isomorphic to $\bbC Ge_B$ as $\bbC G$-modules by $g\otimes 1\mapsto ge_B$. 

Our questions in Section \ref{sec:4}
concern the right multiplication of certain elements in $H$ on the $\bbC G$-module $\bbC Ge_P$. We will attack this question in Section \ref{sec:6} by flipping the side, and consider the left $H$-action on $e_P \bbC G$. But before that, we need to develop the unipotent representation theory of $G$ the representation theory of $H$.

There is a correspondence between the following three sets, and we will address their pairwise relations in the three subsections in this section:
\begin{enumerate}
    \item The equivalence classes of irreducible unipotent representations of $G$,
    \item The equivalence classes of irreducible representations of $H$,
    \item The equivalence classes of irreducible representations of the symmetric group $S_n$.
\end{enumerate}

If we define unipotent representations of $G$ as direct summands of $(\bbC Ge_B)^r, r>0$, then connections exist among  1-3 without the adjective "irreducible".

\subsection{Representations of Group Rings and Hecke Algebras}
The correspondence of sets 1 and 2 mentioned above can be stated more generally. Let $G$ be a finite group and $e$ an idempotent element of $G$ (i.e., $e^2=e$ and $e\neq 0$). In this subsection, we examine the relationship between the irreducible $\bbC G$ submodules of $\bbC Ge$ and the irreducible modules of the algebra $H=e\bbC Ge$.

The detailed results can be found in Chapter 11D of Curtis and Reiner's book \cite{curtis_reiner1981}, which is a comprehensive source on this topic.

\begin{prop}\label{prop:5.1.1}
Let $G$ be a finite group, and $e\in G$ an idempotent element. The associated Hecke algebra $H=e\bbC Ge$ is semisimple.
\end{prop}

\begin{proof}
This result is a direct consequence of Theorem 5.13 and Theorem 5.18 in \cite{curtis_reiner1981}. Alternatively, one can prove this result by decomposing $\bbC G$ as a direct sum of matrix algebras and diagonalizing each component of $e$. 
\end{proof}

\begin{prop}\label{prop:5.1.2} Let $H=e\bbC G e$ with $G$ and $e$ as defined above. The following statements hold:
    \begin{enumerate}
        \item For a simple $\bbC G$-module $M$, the multiplicity of $M$ in a decomposition of $\bbC Ge$ into simple submodules equals $\dim(eM)$. Moreover, $\dim(eM)=\Tr(e,M)$.

        \item The map $M\mapsto eM$ is a bijection from the isomorphism classes of simple $\bbC G$-submodules of $\bbC Ge$ to the isomorphism classes of simple $H$-modules.
        
        The map $\zeta\mapsto \zeta|_H$ is a bijection from irreducible characters $\zeta$ of $G$ such that $\langle\zeta,\phi_{\bbC Ge}\rangle>0$ to the set of irreducible characters of $H$. Here, $\phi_{\bbC Ge}$ denotes the character of $G$ afforded by $\bbC Ge$.
    \end{enumerate}
\end{prop}
\begin{proof}
These results are classical and can be found in various sources. For example, our theorem follows from Theorem 11.25 and its proof in \cite{curtis_reiner1981}. Alternatively, one could  prove this theorem by again interpreting $\bbC G$ as a direct sum of matrix algebras and components of e as diagonal matrices, transforming this theorem into a statement relevant to matrix algebras.
\end{proof}

Back to the situation at the beginning of this section, we can summarize the results as follows: 

\begin{cor}\label{cor:5.1.4} Let $G=\GL_n(\bbF_q)$ and $B$ is the set of upper triangular matrices in $G$. Let $H=e_B \bbC Ge_B$.
\begin{enumerate}
    \item The map $\zeta\mapsto \zeta|_H$ is a bijection from irreducible representations $\zeta$ of $G$ such that $\langle\zeta,\Ind_B^G (1_B)\rangle>0$ to the set of irreducible representations of $H$.
    \item Let $\zeta$ be an irreducible representation of $G$. Then $(\zeta,\Ind_B^G (1_B))=\zeta(e_B)$.
\end{enumerate}
\end{cor}
\begin{proof}
Note that $\Ind_B^G (1_B)$ is the representation of $\bbC Ge_B$. The first statement then follows from part 2 of Proposition \ref{prop:5.1.2}, and the second statement follows from part 1 of that proposition.
\end{proof}

\subsection{Representations of \texorpdfstring{$S_n$}{Sn} and \texorpdfstring{$\GL_n(\bbF_q)$}{GLn(Fq)}}\label{sec:5.2}

In this subsection, we discuss the relationship between the unipotent representations of the general linear group $G = \GL_n(\mathbb{F}_q)$ and representations of the symmetric group $S_n$. 

Note that $S_n$ is the Weyl group of $G$, and the structure of $S_n$ gives the skeleton of structure of $G$. Roughly speaking, the group $S_n$ can be thought of as the general linear group over ``a field with one element". Therefore, the representations of $G = \GL_n(\mathbb{F}_q)$ are intimately related to the representations of $S_n$.

We will first briefly recall the representations of $S_n$, following \cite[Chapter~II]{sagan2001}, and then relate the representations of $S_n$ with those of $\GL_n(\bbF_q)$, as in \cite[Chapter~67B]{curtis_reiner1986}. 

\subsubsection{Representations of \texorpdfstring{$S_n$}{Sn}}\label{sec:5.2.1}
Irreducible representations of $S_n$ are parametrized by partitions of $n$. We first introduce the concepts of partitions and dominance among partitions.
\begin{defn}
    A partition $\lambda$ of a positive integer $n$ is a sequence of positive integers $(\lambda_1, \lambda_2, \ldots, \lambda_r)$ such that $\lambda_1 \geq \lambda_2 \geq \ldots \geq \lambda_r > 0$ and $\sum_{i=1}^r \lambda_i = n$. The integers $\lambda_i$ are called the parts of the partition. We use $\lambda\vdash n$ to denote that $\lambda$ is a partition of $n$.

    For two partitions $\lambda, \mu$ of $n$, we say that $\lambda$ dominates $\mu$, written as $\lambda \unrhd \mu$, if $\sum_{i=1}^k \lambda_i \geq \sum_{i=1}^k \mu_i$ for all $k\geq 1$.
\end{defn}

As a convention, appending finitely many zeros at the end of a partition does not change the partition. For instance, the partition (3,2,0,0) is the same as the partition (3,2).

Given a partition $\lambda$ of $n$, one can define a Specht module $S^\lambda$ over the group ring of the symmetric group $S_n$, as described in Section 2.3 of Sagan's book \cite{sagan2001}. According to Theorem 2.4.4 in \cite{sagan2001}, each Specht module is simple, and affords a corresponding irreducible character of $S_n$, denoted by $\psi_\lambda$. Moreover, as stated in Theorem 2.4.6 of \cite{sagan2001}, the set $\{\psi_\lambda:\lambda\vdash n \}$ is a complete list of irreducible characters of $S_n$ over $\bbC$.

Next, we consider representations induced from subgroups of $S_n$.

\begin{defn}
    Let $\mu=(\mu_1,\mu_2,...,\mu_r)$ be a partition of $n$. The Young subgroup of $S_n$ corresponding to $\mu$ is
    $S_\mu=S_{\{1,2,...,\mu_1\}}\times S_{\{\mu_1+1,\mu_1+2,...,\mu_1+\mu_2\}}\times...\times S_{\{n-\mu_r+1,n-\mu_r+2,...,n\}}$.
\end{defn}

\begin{defn}\label{def:5.2.3}
Let $\lambda,\mu$ be two partitions of $n$. 
Denote by $\phi_\mu$ the induced character $(1_{S_\mu})^{S_n}=\Ind^{S_n}_{S_\mu}(1_{S_\mu})$. The Kostka number $K_{\lambda,\mu}$ is defined as the multiplicity of $\psi_\lambda$ in $\phi_\mu$, i.e.,
$$\langle \psi_\lambda, \phi_\mu \rangle = K_{\lambda,\mu}.$$
\end{defn}

We record here two properties of the Kostka numbers $K_{\lambda,\mu}$, see Corollary 2.4.7 in \cite{sagan2001}:

\begin{prop}\label{prop:5.2.4} For any partitions $\lambda,\mu$ of $n$,
$K_{\lambda,\mu}\geq 1$ exactly when $\lambda\unrhd\mu$. Moreover, $K_{\mu,\mu}=1$.     
\end{prop}

Therefore, one can decompose the induced character $(1_{S_\mu})^{S_n}=\phi_\mu$ as  $$(1_{S_\mu})^{S_n}=\phi_\mu=\sum_{\lambda\unrhd \mu} K_{\lambda,\mu} \psi_\lambda.$$
We may assemble the Kostka numbers into a matrix $K=(K_{\lambda,\mu})$ indexed by partitions $\lambda$, and totally order all the partitions in a way such that $\lambda\geq \mu$ whenever $\lambda\unrhd \mu$ (say using the lexicographic order, see [Sa,2.2.5]). Then $K$ is a lower triangular matrix with each diagonal entry equal to 1. Then $K$ is an invertible matrix, and the inverse $K^{-1}$ is also lower triangular, and with entries 1 on the diagonal. 

We may arrange all partitions as $\lambda_1<\lambda_2<...<\lambda_{P(n)}$ according to the total order chosen as above, where $P(n)$ is the number of partitions of $n$. Then $$(\phi_{\lambda_1},\phi_{\lambda_2},...,\phi_{\lambda_{P(n)}})=(\psi_{\lambda_1},\psi_{\lambda_2},...,\psi_{\lambda_{P(n)}})K.$$
Therefore, $$(\psi_{\lambda_1},\psi_{\lambda_2},...,\psi_{\lambda_{P(n)}})=(\phi_{\lambda_1},\phi_{\lambda_2},...,\phi_{\lambda_{P(n)}})K^{-1},$$
and each $\psi_\mu=\sum_{\lambda\geq \mu} (K^{-1})_{\lambda,\mu} \cdot \phi_\lambda$. This expresses the irreducible representation as an integer combination of certain induced characters.

\begin{prop}\label{prop:5.2.5}
    Every irreducible character of $S_n$ is an integer combination of induced characters of the form $\phi_\lambda=(1_{S_\lambda})^{S_n}$. 
\end{prop}

\begin{proof}
    For $\mu\vdash n$, the character $\psi_\mu$ ranges over all irreducible characters of $S_n$. Then the statement follows from the equality $\psi_\mu=\sum_{\lambda\geq \mu} (K^{-1})_{\lambda,\mu} \cdot \phi_\lambda$.
\end{proof}

\subsubsection{Coxeter Systems and Weyl Groups}\label{sec:5.2.2}
Before we can establish the link between the representations of $S_n$ and $\GL_n(\bbF_q)$, we need to introduce the concept of a Coxeter system.

A Coxeter system is a pair $(W, S)$ where $W$ is a group generated by a set of reflections $S$, with the presentation $W = \langle S \mid (st)^{m_{st}} = 1 \rangle$ for $s, t \in S$, and $m_{st} \in \{2, 3, \ldots\} \cup \{\infty\}$. Here, $m_{ss}=1$ for all $s \in S$, $m_{st}=m_{ts}$ for all $s, t \in S$, and $m_{st} = \infty$ means that there are no relations between $s$ and $t$.

The Weyl group of a linear algebraic group is an example of a Coxeter group. Now we return to the case where $G = \GL_n(\bbF_q)$. Here the Weyl group is $W=S_n$, and the set of reflections $S$ is $\{s_1=(1,2),s_2=(2,3),...,s_{n-1}=(n-1,n)\}$.

For a subset $J$ of $S$, we denote by $W_J$ the subgroup of $W$ generated by $J$, and call $W_J$ a parabolic subgroup of $W$.  

\begin{prop}\label{prop:5.2.6}
Every Young subgroups of $S_n$ can be identified with some parabolic subgroup of $W=S_n$.

Conversely, every parabolic subgroup of $W=S_n$ is conjugate to some Young subgroup $S_\lambda$ of $S_n$.
\end{prop}

\begin{proof}
Let $\lambda = (\lambda_1, \lambda_2, \ldots, \lambda_r)$ be a partition of $n$. Consider the subset $J(\lambda) = S-\{{s_{\lambda_1}, s_{\lambda_1 + \lambda_2}, \ldots, s_{\lambda_1 + \cdots + \lambda_{r-1}}}\}\subseteq S$. Then the parabolic subgroup $W_{J(\lambda)}$ generated by $J(\lambda)$ is exactly $S_\lambda=S_{\{1,2,...,\lambda_1\}}\times...\times S_{\{n-\lambda_r+1,n-\lambda_r+2,...,n\}}$.

Conversely, let $W_J$ be a parabolic subgroup of $W=S_n$, generated by a subset $J = S-\{s_{\lambda_1}, s_{\lambda_1 + \lambda_2}, \ldots, s_{\lambda_1 + \cdots + \lambda_{r-1}}\} \subseteq S$ for some positive integers $\lambda_1,...,\lambda_{r-1}$ such that $\lambda_1+...+\lambda_{r-1}<n$. We define $\lambda_r=n-(\lambda_1+...+\lambda_{r-1})$. 

Next, we rearrange $\lambda_1,...,\lambda_r$ in weakly descending order to form a partition $\mu=(\lambda_{\sigma(1)},...,\lambda_{\sigma(r)})$ of $n$. Hence, $W_J=S_{\{1,2,...,\lambda_1\}}\times...\times S_{\{n-\lambda_r+1,n-\lambda_r+2,...,n\}}$, and is conjugate to $S_\mu$.
\end{proof}

Recall the Bruhat decomposition $G=\bigsqcup_{w\in W} BwB$, where each $w\in W=S_n$ is identified with the matrix that permutes $\{e_1,...,e_n\}$ via $e_i\mapsto e_{w(i)}$, and $B=\UT_n(\bbF_q)$ is the set of upper triangular matrices.

For a subset $I$ of $S$, we set $P_I=BW_IB$. Then $P_I$ is a subgroup of $G$ (see [CR86, 65.13]), called a parabolic subgroup of $G$. We denote the induced character $(1_{W_I})^{S_n}$ of $S_n$ by $\phi_I$, and denote the character $(1_{P_I})^G$ of $G$ by $\Phi_I$. Then $\phi_\mu=\phi_{J(\mu)}$, with $J(\mu)$ defined as in the proof above. We therefore denote $\Phi_\mu=\Phi_{J(\mu)}$. On the other hand, if $W_I$ is conjugate to $S_\lambda$, the two $\bbC G$-modules $\bbC Ge_{W_I}$ and $\bbC Ge_{S_\lambda}$ are isomorphic as $\bbC G$-modules, and so $\phi_I=\phi_\mu$. The equation $\Phi_I=\Phi_\mu$ follows easily from the injectivity of hat operator in the next part.

\subsubsection{Linking Representations of \texorpdfstring{$S_n$}{Sn} and \texorpdfstring{$\GL_n(\bbF_q)$}{GLn(Fq)}}
We can now establish a correspondence between certain virtual characters of $S_n$ and $G=\GL_n(\bbF_q)$. Recall that a virtual character of a group is an integer combination of some irreducible characters of that group. The discussion of \ref{lemma:innerProdEqDC}-\ref{rem:5.2.10}  follows Chapter 67B in Curtis and Reiner's book \cite{curtis_reiner1986}.

\begin{lemma}\label{lemma:innerProdEqDC}
    Let $G$ be a finite group, and $H,K$ are subgroups of $G$. Then the inner product of the induced characters $\angles{(1_H)^G,(1_K)^G}=|H\backslash G/K|$. 
\end{lemma}

\begin{proof}
    The character $(1_H)^G$ is afforded by the module $\bbC Ge_H=\bbC[G/H]$. Hence, $(1_H)^G$ is the permutation character of $G$ acting on $G/H$ by left multiplication. Therefore, $\angles{(1_H)^G,(1_K)^G}=\frac{1}{|G|}\sum_{g\in G}|(G/H)^g||(G/K)^g|=\frac{1}{|G|}\sum_{g\in G}|(G/H\times G/K)^g|$, where $X^g=\{x\in X: gx=x\}$ for a $G$-set $X$. 
    
    By Burnside's counting theorem, $\frac{1}{|G|}\sum_{g\in G}|(G/H\times G/K)^g|=|(G/H\times G/K)/G|$, the number of orbits of the left $G$-action on $(G/H\times G/K)$. It remains to observe that there is a bijection $(G/H\times G/K)/G\to H\backslash G/K$ by $(xH,yK)\mapsto Hx^{-1}yK$.
\end{proof}

Return to the special setting where $G=\GL_n(\bbF_q)$ and $B=\UT_n(\bbF_q)$. 

We have the following equality:

\begin{prop}\label{prop:5.2.8} Let $(W,S)$ be the Coxeter system associated with $G$. 
Let $I, J$ be subsets of $S$, and let $P_I=BW_IB$ and $P_J=BW_JB$.  Then the inner product of the induced characters $\langle\Phi_I, \Phi_J\rangle = \langle\phi_I, \phi_J\rangle$.
\end{prop}

\begin{proof}
Recall that $\Phi_I=(1_{P_I})^G$ and $\phi_I=(1_{W_I})^W$. By Lemma \ref{lemma:innerProdEqDC}, $\langle\Phi_I, \Phi_J\rangle=|P_I\backslash G/P_J|$, and $\langle\phi_I, \phi_J\rangle=|W_I \backslash W/W_J|$. By Theorem 65.21 in \cite{curtis_reiner1986}, there is a bijection of double cosets $W_I \backslash W/W_J \to P_I\backslash G/P_J$ by $W_IwW_J\mapsto BW_IwW_JB=P_I wP_J$. Therefore, $|P_I\backslash G/P_J|=|W_I \backslash W/W_J|$ and our result follows.
\end{proof}

Now we can establish the correspondence between certain virtual characters of $S_n$ and $G=\GL_n(\bbF_q)$. We define a map from the virtual characters of $S_n$ to the virtual characters of $G=\GL_n(\bbF_q)$ by sending $\phi_I$ to $\Phi_I$ and extending it linearly to virtual characters.

\begin{prop}\label{prop:5.2.9}
The map \[\xi=\sum_{J\subseteq S} n_J \phi_J \mapsto \widehat{\xi}=\sum_{J\subseteq S} n_J \Phi_J, \text{where each $n_J\in \bbZ$} \] defines an injective map from  virtual characters of $W=S_n$ that are integer combinations of $\{\phi_J:J\subseteq S\}$ to virtual characters of $G$. This map preserves inner products.
\end{prop}

\begin{proof}
    Given indexed integers $N=\{n_J\}_{J\subseteq S}\in \bbZ^{\{J:J\subseteq S\}}$, let $\xi_N=\sum_{J\subseteq S} n_J \phi_J$, and $\eta_{N}=\sum_{J\subseteq S} n_J \Phi_J$. For any $N=\{n_J\}_{J\subseteq S},N'=\{n'_J\}_{J\subseteq S}$,
$$\angles{\xi_N,\xi_{N'}}=\sum_{I\subseteq S,J\subseteq S} n_I n'_J\angles{\phi_I,\phi_J}= \sum_{I\subseteq S,J\subseteq S} n_I n'_J\angles{\Phi_I,\Phi_J}=\angles{\eta_{N},\eta_{N'}}.$$ 

    Note that the two maps $N\mapsto \xi_N$ and $N\mapsto \eta_N$ are both linear. Then  $$\xi_{N}=\xi_{N'}\iff \angles{\xi_{N-N'},\xi_{N-N'}}=0\iff \angles{\eta_{N-N'},\eta_{N-N'}}=0\iff \eta_{N}=\eta_N'.$$ Therefore, the map $\xi_{N}\mapsto \eta_{N}$ does not depend on the $N$ chosen, and is injective. This gives the linear association $\xi\mapsto \hat{\xi}$, which preserves inner product.
\end{proof}

\begin{rem}\label{rem:5.2.10}
    If $\xi=\sum_{J\subseteq S} n_J \phi_J$ is an irreducible character of $W$, then $\angles{\xi,\xi}=1$, and $\angles{\hat{\xi},\hat{\xi}}=1$. Therefore, either $\hat{\xi}$ or $-\hat{\xi}$ is an irreducible character of $G$. We will show that the former is the case.
\end{rem}

\begin{prop}\label{prop:5.2.11}
    Suppose $\psi=\sum_{J\subseteq S} n_J \phi_J$ is an irreducible character of $W$. Then $\hat{\psi}$ is an irreducible character of $G$.
\end{prop}

\begin{proof}
Take $I=\varnothing$, and let $\xi=\phi_I=(1_{\{1\}})^W$. Then $\xi$ is the regular character of $W$, and $\angles{\xi, \psi}= \deg(\psi)$ because $\psi$ is an irreducible character of $W$.

Then $\angles{\hat{\xi}, -\hat{\psi}}=-\deg(\psi)<0$. Note that $\hat{\xi}=\Phi_I=(1_B)^G$ is a character of $G$, and the inner product of two characters of $G$ is non-negative. Therefore, $-\hat{\psi}$ is not a character of $G$. By the previous remark, either $\hat{\psi}$ or $-\hat{\psi}$ is irreducible. Then $\hat{\psi}$ is irreducible character of $G$. 
\end{proof}

By Proposition \ref{prop:5.2.5} and \ref{prop:5.2.6}, each irreducible character of $S_n$ is an integer combination of certain characters of the form $\phi_J$. Then the domain of the hat map $\xi\mapsto \hat{\xi}$ is the set of all virtual characters of $W$. Moreover, $\{\hat{\xi}:\xi \in \Irr(S_n)\}$ is the set of all unipotent characters of $G$:
\begin{prop}\label{prop:5.2.12}
    Let $G=\GL_n(\bbF_q)$, and $W=S_n$ is the Weyl group of $G$. The set of all irreducible characters in $(1_B)^G$ is $\{{\Psi_\lambda}: \lambda\vdash n\}$, where $\Psi_\lambda$ is defined as $\widehat{\psi_\lambda}$. Moreover, $(1_B)^G=\sum_{\lambda \vdash n} \deg({\psi_\lambda}){\Psi_\lambda}$.
\end{prop}
\begin{proof}
Again consider $I=\varnothing$ and $\xi=\phi_I$, the regular character of $W=S_n$.
Then $\xi= \sum_{\lambda \vdash n} \deg(\psi_\lambda)\psi_\lambda$ is the decomposition of $\xi$ into irreducible characters.

By the previous remark, one can apply the hat map to every $\psi_\lambda$. Hence, $(1_B)^G=\hat{\xi} =\sum_{\lambda \vdash n} \deg({\psi_\lambda})\widehat{\psi_\lambda}=\sum_{\lambda \vdash n} \deg(\psi_\lambda)\Psi_\lambda$. 
Because the hat map is injective, and maps irreducible characters to irreducible characters, the expression $(1_B)^G=\sum_{\lambda \vdash n} \deg({\psi_\lambda})\Psi_\lambda$ is the decomposition of $(1_B)^G$ to irreducible characters of $G$.
\end{proof}

\begin{rem}
    Proposition \ref{prop:5.2.8}-\ref{prop:5.2.11} applies to general finite group $G$ with $(B,N)$-pairs. To get a result as in Proposition \ref{prop:5.2.12}, one needs to express each character of $W$ as linear combination of characters of the form $\phi_J$. This is not always feasible for other types of $G$. 
\end{rem}

\subsection{Hecke Algebra and Tits' Deformation Theorem}\label{sec:5.3}

In Section \ref{sec:5.2}, we established that every unipotent irreducible representation of $G = \GL_n(\mathbb{F}_q)$ equals $\hat{\psi_{\lambda}}$ for some partition $\lambda$ of $n$. By Proposition \ref{prop:5.1.2}(3), every irreducible representation of $H = e_B \mathbb{C} G e_B$ can be written as $\hat{\psi_{\lambda}}|_H$. Therefore, the map $\psi \mapsto \hat{\psi}|_H$ is a bijection between irreducible representations of $S_n$ and $H$. Moreover, by Corollary \ref{cor:5.1.4}, $\deg(\hat{\psi}|_H)=\angles{\hat{\psi},(1_{B})^G}$, which equals $\deg(\psi)$ by Proposition \ref{prop:5.2.12}.
Since both $\mathbb{C} S_n$ and $H$ are semisimple, it follows that they have the same decomposition into direct sum of matrix algebras. That is, $\mathbb{C} S_n \simeq \bigoplus_{\lambda \vdash n} M_{\deg(\psi_{\lambda})}(\mathbb{C}) \simeq H$.

In this section, we discuss the Hecke algebra $H$ by analyzing the base elements, and state (without proof) the Tits' Deformation Theorem which relates $\mathbb{C} W$ and $\mathbb{C} G$ for a finite group $G$ with a $BN$-pair.

We first assume that $G$ is a finite group, with $B$ a subgroup of $G$. Then we focus on the case where $G$ is a finite group with a $BN$-pair. The reader can refer to \cite[\S VIII]{feit_higman1964} or \cite[\S65]{curtis_reiner1986} for the definition and properties of a group with a $BN$-pair. Alternatively, one may consider our special setting $G = \GL_n(\mathbb{F}_q)$ and $B = \UT_n(\mathbb{F}_q)$ throughout the discussion.

\subsubsection{A Basis of Hecke Algebra}

Let $G$ be a finite group, and and let $B$ be a subgroup of $G$. In this section, we present a basis of Hecke algebra $H = e_B \mathbb{C} G e_B$, and calculate the structure constants of $H$ when $G$ is general and when $G$ is a group with $BN$-pair. We refer the reader to \cite[\S11D]{curtis_reiner1981} for the general calculation, and to \cite[\S67A]{curtis_reiner1986} for the special case.

\begin{prop}
Let $B$ be a subgroup of a finite group $G$, and $H = e_B \mathbb{C} G e_B$. Let $B \backslash G / B = \{D_i\}_{1 \leq i \leq r}$, where $D_i = Bx_iB, 1 \leq i \leq r$.
\begin{enumerate}
\item For each $1 \leq i \leq r$, define $a_i = \frac{1}{|B|} \sum_{x \in D_i} x$ (see also Definition 4.1.3). Then $\{a_i \mid 1 \leq i \leq r\}$ is a $\mathbb{C}$-basis for $H$.
\item For $1 \leq i, j \leq r$, we have $a_i a_j = \sum_{k \in I} \mu_{ijk} a_k$, where the structure constants are given by $\mu_{ijk} = \frac{|D_i \cap x_k D_j^{-1}|}{|B|}$.
\end{enumerate}
\end{prop}

\begin{proof}
\begin{enumerate}
\item An element $c = \sum_{x \in G} c_x x \in \mathbb{C} G$ lies in $H$ exactly when $c_x = c_{bxb'}$ for any $b, b' \in B$. That is, the coefficient $c_x$ is uniform on each $B-B$ double coset of $G$. Then $\{a_i | 1 \leq i \leq r\}$ is a basis of $H$.
\item By definition, $a_i = \frac{1}{|B|} \sum_{x \in D_i} x$, and $a_j = \frac{1}{|B|} \sum_{y \in D_j} y$. Then the coefficient of $x_k$ in $a_i a_j$ equals $\frac{|\{x \in D_i, y \in D_j, x_k = xy\}|}{|B|^2} = \frac{|D_i \cap x_k D_j^{-1}|}{|B|^2}$. On the other hand, the coefficient of $x_k$ in $\sum_{k \in I} \mu_{ijk} a_k$ is $\frac{\mu_{ijk}}{|B|}$. Therefore, $\mu_{i,j,k} = \frac{|D_i \cap x_k D_j^{-1}|}{|B|}$.
\end{enumerate}
\end{proof}

Now consider $G$ is a finite group with a $BN$-pair, and $W$ is the Weyl group of $G$, with a set of reflections $S=\{s_1,...,s_n\} \subseteq W$ such that $(W,S)$ is a Coxeter system. Let $l$ be the length function in $(W,S)$. Since there is a bijection between $W$ and $B \backslash G / B$ given by $w \mapsto B \dot{w} B$, we may denote $a_w = \frac{1}{|B|} \sum_{x \in B \dot{w} B} x$. Then $\{a_w \mid w \in W\}$ is a basis for $W$, and the multiplication in $H$ satisfies:

\begin{prop}\label{prop:5.3.2} For each $s \in S$, let $q_s = |B \dot{s} B / B|$. Then the multiplication in $H$ satisfies $a_s a_w = a_{sw}$ if $l(sw) > l(w)$ and $a_s a_w = q_s a_{sw} + (q_s - 1) a_w$, if $l(sw) < l(w)$.
\end{prop}

\begin{proof}
From the previous proposition, we have that $a_s a_w = \sum_{v \in W} \mu_{s,w,v} a_v$, where $\mu_{s,w,v} = \frac{|B \dot{s} B \cap \dot{v} B \dot{w}^{-1} B|}{|B|} = \frac{|B \dot{s} B \dot{v} \cap B \dot{w} B|}{|B|}$.

By an axiom of $BN$-pairs, we know that $B \dot{s} B \dot{w} \subseteq B \dot{s} \dot{w} B \cup B \dot{w} B$. Therefore, if $B \dot{w} B \cap B \dot{s} B \dot{w}$ is non-empty, then $sw = w$ or $v = w$. This implies that $\mu_{s,w,v} = 0$ for $v$ not equal to $sw$ or $w$.

If $l(sw) > l(w)$, then $B \dot{s} B \dot{w} \subseteq B \dot{s} \dot{w} B$ and does not intersect with $B \dot{w} B$. Therefore, $\mu_{s,w,w} = |B \dot{s} B \dot{w} \cap B \dot{w} B| / |B| = 0$ and $\mu_{s,w,sw} = |B \dot{w} B| / |B| = 1$, so $a_s a_w = a_{sw}$.

Next, consider $a_s^2$. We have $\mu_{s,s,1} = |B \dot{s} B \dot{1} \cap B \dot{s} B| / |B| = q_s$ and $\mu_{s,s,s} = |B \dot{s} B \dot{s} \cap B \dot{s} B| / |B|$. As $B \subseteq B \dot{s} B \dot{s}$ and $B \dot{s} B \dot{s} \subseteq B \dot{s} B \cup B$, we conclude that the intersection $B \dot{s} B \dot{s} \cap B \dot{s} B = B \dot{s} B \dot{s} - B$. Therefore, $\mu_{s,s,s} = |B \dot{s} B \dot{s} - B| / |B| = q_s - 1$, and $a_s^2 = q_s a_1 + (q_s - 1) a_s$.

Finally, if $l(sw) < l(w)$, then $l(ssw) > l(sw)$. Applying the result of the first case, we find $a_s a_w = a_s a_s a_{sw} = (q_s a_1 + (q_s - 1) a_s) a_{sw} = q_s a_1 a_{sw} + (q_s - 1) a_s a_{sw} = q_s a_{sw} + (q_s - 1) a_w$. Note that $a_1 = e_B$ acts as the identity on $H = e_B \mathbb{C} G e_B$.
\end{proof}

One can prove a dual statement to this:
\begin{cor}\label{cor:5.3.3}
Assumptions as in the last proposition, 
$a_w a_s = a_{ws}$ if $l(ws) > l(w)$ and $ a_w a_s= q_s a_{ws} + (q_s - 1) a_w$, if $l(ws) < l(w)$.
\end{cor}

\begin{proof}
Suppose $l(ws)>l(w)$. Let $w=s_1...s_{l(w)}$ be a reduced expression for $w$. Then $ws=s_1...s_{l(w)}s$ is a reduced expression for $ws$. Use the previous proposition repeatedly, one find that  $a_w=a_{s_1}a_{s_2...s_{l(w)}}=...=a_{s_1}...a_{s_{l(w)}}$. Similarly, $a_{ws}=a_{s_1}...a_{s_{l(w)}}a_s$. Therefore, $a_{ws}=a_w a_s$.

Suppose $l(ws)<l(w)$. Then $a_{w}a_s=a_{ws}a_sa_s=a_{ws}(q_s a_1+(q_s-1)a_s)=q_s a_{ws}+(q_s-1)a_w$.
\end{proof}

\begin{cor}\label{cor:5.3.4} Let $w_0$ be the longest element in $(W,S)$, and $w\in W$.
Suppose each $q_s\geq 2$, the coefficient of $a_w$ in $a_{w_0}^2$ is positive, and $G=B\overline{B}B$.
\end{cor}
\begin{proof}
Let $w=s_1...,s_r$ be a reduced expression for $w$, and $s_{r+1}...s_l$ be a reduced expression for $w^{-1}w_0$. Then $w_0=s_1...,s_l$ is a reduced expression for $w_0$. There exists $s'_1,...,s'_r$ such that $w_0=s'_1...s'_r s_l...s_{r+1}$ is also a reduced expression for $w_0$. 

Then $a_{w_0}^2=a_{w_0} a_{s'_1}...a_{s'_r} a_{s_l}...a_{s_{r+1}}$. Since each $q_s\geq 2$, the coefficient of $a_{w_0}$ is positive in $a_{w_0} a_{s'_1}...a_{s'_r}$. On the other hand, the coefficient of $a_{w}$ in $a_{w_0}a_{s_l}...a_{s_{r+1}}=a_wa_{s_{r+1}}...a_{s_l}a_{s_l}...a_{s_{r+1}}$ is also positive. This proves the first statement.

Now, $a_w = \frac{1}{|B|} \sum_{x \in BwB} x$. Then the first statement implies that $BwB\subseteq Bw_0 Bw_0B$ for each $w\in W$. The statement $G=B\overline{B}B$ then follows from the Bruhat Decomposition  $G=\cup_{w\in W} BwB$ and $w_0 Bw_0=\overline{B}$.
\end{proof}

In the case of $G = \GL_n(\mathbb{F}_q)$, each $q_s = q\geq 2$, and $G=B\overline{B}B$.

\subsubsection{Generator-Relation Representation of Hecke Algebra}\label{sec:5.3.2}

The Hecke algebra $H$ can be described in terms of generators and relations. For each $s \in S$, we define a generator $T_s = a_s$. The relations between these generators are given by the braid relations and quadratic relations.

Braid relations: For any $s, t \in S$, we have $(T_s T_t)^{m_{st}} = (T_t T_s)^{m_{st}}$, where $m_{st}$ is the order of $st$ in the Coxeter group $W$.

Quadratic relations: For any $s \in S$, we have $T_s^2 = (q_s - 1) T_s + q_s T_1$.

These relations reflect the Coxeter relations and the reflection properties of the elements of $S$ in the Weyl group $W$. When each $q_s = 1$, the Hecke algebra $H$ reduces to the group algebra $\mathbb{C} W$ of the Weyl group $W$.

\subsubsection{Tits' Deformation Theorem}

The Tits' Deformation Theorem establishes a powerful tool that relates the representations of the Hecke algebra of a finite group with a $BN$-pair and its associated Weyl group.

For a detailed discussion and proof of the following theorem, readers are referred to Theorem 8.1.7 and 9.1.9 in \cite{geck_pfeiffer2000} or Theorem 68.24 in \cite{curtis_reiner1986} for the first part and Theorem 68.21 in \cite{curtis_reiner1986} for the second part. The theorem is stated as follows:

\begin{thm}[Tits' Deformation Theorem]
Let $G$ be a finite group that has a $BN$-pair, and let $(W, S)$ denote the Weyl group associated with $G$. If we let $H=e_B \mathbb{C} G e_B$, then the following holds:
\begin{enumerate}
\item There exists a bijection from $\Irr(\mathbb{C} W)$ to $\Irr(H)$, which, when extended linearly, maps $\phi_J$ to $\Phi_J|_{H}$.
\item There exists a $\mathbb{C}$-algebra isomorphism $H \simeq \mathbb{C} W$.
\end{enumerate}
\end{thm}

This theorem generalizes the discussion at the start of this subsection by providing a concrete relationship between the irreducible representations of the Weyl group and the Hecke algebra for all finite groups $G$ with a $BN$-pair. This insight is crucial for understanding the structure of Hecke algebras and their representations.

\section{Right action of \texorpdfstring{$(e_B w_0 e_B)^2$}{(eBw0eB)2} on \texorpdfstring{$\mathbb{C} G e_P$}{CGeP}}\label{sec:6}

In this section, $G=\GL_n(\mathbb{F}_q)$, $B=\UT_n(\mathbb{F}_q)$, $W=S_n$ and $S=\{s_1=(1,2),\ldots,s_{n-1}=(n-1,n)\}$. The longest element of $W$ is denoted by $w_0$, and $P$ denotes a parabolic subgroup of $G$ containing $B$.

We return to the question raised in Section \ref{sec:4}, the study of the eigenvalues of the right action of the element $(e_Bw_0 e_B)^2$ on the module $\mathbb{C} Ge_P$. This question lies at the heart of finding zeta functions in the general case.

To solve this question, firstly, we change the side and show that it is equivalent to consider the left action of the element $(e_Bw_0 e_B)^2$ on the space $e_P\mathbb{C} G$. Secondly, we employ the decomposition of $\bbC G$-module $\bbC G=\oplus_{M\in \Irr(G)} M^{\dim(M)}$ to get the decomposition $e_B \bbC G= \oplus_{\lambda} (e_B M_\lambda)^{\dim(M_\lambda)}$ and $e_P \bbC G= \oplus_{\lambda} (e_P M_\lambda)^{\dim(M_\lambda)}$.
It turns out $(e_B w_0e_B)^2$ acts as a scalar on each $e_B M_\lambda$ by Springer's theorem on the centrality of certain Hecke algebra elements. Therefore, $(e_B w_0e_B)^2$ acts as a scalar on the subspace $e_P M_\lambda$. It remains to find this scalar, as well as the dimensions of various $ M_\lambda$ and $e_P M_\lambda$.

We will make this outline rigorous in Section \ref{sec:6.1}, and in Section \ref{sec:6.3}, we will give a proof of Springer's Theorem. The dimensions and scalars will be computed in Section \ref{sec:6.2}.

\subsection{Main Theorem on Action of \texorpdfstring{$(e_B w_0 e_B)^2$}{CGeP}}\label{sec:6.1}
We first prove a lemma allowing for changing sides in finding eigenvalues.
\begin{lemma}
Let $R$ be a finite-dimensional split semisimple algebra over $\bbC$. Let $e\in R$ be an idempotent element, and let $a\in eRe$. Then the left multiplication $a_l: eR\to eR$ has the same eigenvalues as the right multiplication $a_r:Re\to Re$.
\end{lemma}

\begin{proof}
It suffices to consider the case where $R=M_n(\bbC)$, for some integer $n$. When $e=1$, then under suitable bases, $a_l:R\to R$ is represented by the block-diagonal matrix $\diag(A,A,...,A)$, with $A$ repeated $n$-times, while $a_r$ is represented by the block-diagonal matrix $\diag(A^T,A^T,...,A^T)$. These two matrices are transposes of each other, and have the same eigenvalues.

Now, $R=eR\oplus (1-e)R$, and $R=Re\oplus R(1-e)$. Note that $\dim(eR)=\Tr(e_l)$ and $\dim(Re)=\Tr(e_r)$. Then since $\Tr(e_l)=\Tr(e_r)$, $\dim(eR)=\dim(Re)$ and $\dim((1-e)R)=\dim(R(1-e))$. Note that $a_l$ acts as 0 on $(1-e)R$, so the eigenvalues of $a_l$ on $R$ consist of the eigenvalues of $a_l$ on $eR$, as well as 0 with multiplicity $\dim((1-e)R)$. A similar statement holds for $a_r$, with $\dim((1-e)R)$ replaced by $\dim(R(1-e))$. Since the eigenvalues of $a_l$ and $a_r$ on $R$ match, and $\dim((1-e)R)=\dim(R(1-e))$, the eigenvalues of $a_l|_{eR}$ and $a_r|_{Re}$ also match.
\end{proof}

Now $\bbC G$ is split semisimple, the eigenvalues of $(e_Be_{\overline{B}}e_B)_r$ on $\bbC Ge_{P}$ are the same as the eigenvalues of $(e_Be_{\overline{B}}e_B)_l$ on $e_P \bbC G$. To find these eigenvalues, we study the decomposition of $e_P\bbC G$. It is beneficial to also consider the larger space $e_B\bbC G$, which is an $H$-module, where $H=e_B \bbC G e_B$. We first decompose $\bbC G$ into simple $\bbC G$-modules.

Let $\lambda\vdash n$. Recall that $\Psi_\lambda=\widehat{\psi_\lambda}$ denotes the irreducible unipotent character of $G$ corresponding to $\lambda$. Let $M_\lambda$ be a representation of $G$ affording this character, and put $d_\lambda=\dim(M_\lambda)$. Let $M_1,...,M_s$ be a complete set of representatives of all the non-unipotent irreducible representations of $G$, with dimensions $d_1,...,d_s$ respectively. Then we have the standard decomposition of $\bbC G$-modules:
\begin{equation}
    \bbC G\simeq \bigoplus_{\lambda\vdash n} M_\lambda^{d_\lambda}\oplus \bigoplus_{i=1}^s M_i^{d_i},
    \label{eq:equation of CG}
\end{equation}

\begin{prop}\label{prop:6.1.2}
    The formula $$e_B \bbC G \simeq\bigoplus_{\lambda\vdash n} (e_B M_\lambda)^{d_\lambda}$$ gives a decomposition of $e_B \bbC G$ into simple $H$-modules, where $H=e_B \bbC G e_B$.
\end{prop}
\begin{proof}
    Since $H$ is a subring of $\bbC G$, the decomposition \eqref{eq:equation of CG} also preserves the inherited $H$-action. Note that $e_B \bbC G$, each $e_B M_\lambda$ and $e_B M_i$ are $H$-modules, we have the decomposition  $$e_B \bbC G \simeq \bigoplus_{\lambda\vdash n} (e_B M_\lambda)^{d_\lambda}\oplus \bigoplus_{i=1}^s (e_BM_i)^{d_i}$$ of $H$-modules. Now it suffices to show that each $e_B M_\lambda$ is simple and each $e_B M_i$ is 0.

    Recall as in the proof of Proposition \ref{prop:3.3.2} that $e M$ is naturally identified with the space  $\Hom(\bbC Ge, M)$ for any idempotent $e\in \bbC G$ and for any $\bbC G$-module $M$. Then $\dim(e_B M)=\dim(\Hom(\bbC Ge_B, M))=\angles{1_B^G,\Psi}$, where $\Psi$ is the character afforded by $M$. But by Proposition \ref{prop:5.2.12}, $(1_B)^G=\sum_{\lambda \vdash n} \deg({\psi_\lambda}){\Psi_\lambda}$. Therefore, $\dim(e_BM_i)=0$ for each $i$ and $\dim(e_BM_\lambda)=\deg({\psi_\lambda})>0$ for each $\lambda$. Now $M_\lambda$ is a simple $\bbC G$-module, and it follows from Proposition \ref{prop:5.1.2} that each $e_B M_\lambda$ is simple.
\end{proof}
Now we turn to the study of $(e_B w_0 e_B)^2_l$ on $e_P \mathbb{C} G$. Firstly, $P$ is a parabolic subgroup of $G$ containing $B$, so by Theorem 65.17 in \cite{curtis_reiner1986}, there exists a subset $I \subseteq S$ such that $P = P_I$. Let $\mu$ be a partition of $n$ so that $S_\mu$ is conjugate to $W_I$ (see Section \ref{sec:5.2.2}). We can now state the main theorem of this section, and prove it modulo the Springer's Theorem \ref{thm:6.3.1}.

\begin{thm}\label{prop:6.4.1}
Let $P$, $I$ and $\mu$ be as above. The left action of $(e_B w_0 e_B)^2$ on $e_P \mathbb{C} G$ is diagonalizable. 
Each partition $\lambda \unrhd \mu$ contributes the eigenvalue $q^{f_\lambda-n(n-1)}$ with multiplicity $d_\lambda K_{\lambda,\mu}$, and this accounts for all the eigenvalues with multiplicities, where $f_\lambda$ is some integer to be defined in Theorem \ref{thm:6.3.1}.
\end{thm}
\begin{proof}
Since $P$ is a parabolic subgroup containing $B$, we note that $e_P=e_Pe_B$, and $e_P \bbC G\subseteq e_B \bbC G$ corresponds to the subspace $\bigoplus_{\lambda\vdash n} (e_P M_\lambda)^{d_\lambda}$ via the isomorphism in Proposition \ref{prop:6.1.2}, in other words, $e_P \bbC G\simeq \bigoplus_{\lambda\vdash n} (e_P M_\lambda)^{d_\lambda}$ as vector spaces. This isomorphism is compatible with action by any $h\in H$ that stabilizes either side.

By Springer's theorem \ref{thm:6.3.1}, the left multiplication by $a_{w_0}^2=q^{n(n-1)}(e_B w_0 e_B)^2\in H$ acts as scalar multiplication by $q^{f_\lambda}$ on the simple $H$-module $e_B M_\lambda$, where $f_\lambda$ is an integer. Then $(e_B w_0 e_B)^2$ acts as the scalar $q^{f_\lambda-n(n-1)}$ on $e_B M_\lambda$ as well as on $e_P M_\lambda\subseteq e_B M_\lambda$. 

As in the proof of Proposition \ref{prop:6.1.2}, $\dim(e_P M_\lambda)=\angles{1_P^G, \Psi_\lambda}$. Since $S_I$ is conjugate to $S_\mu$, the character $1_{W_I}^{W}$ equals $(1_{S_\mu})^W=\phi_\mu$. Then by Section \ref{sec:5.2}, $(1_P)^G=(1_{P_I})^G=\widehat{(1_{W_I})^W}=\widehat{\phi_\mu}=\Phi_\mu$. Then $\dim(e_P M_\lambda)=\angles{\Phi_\mu,\Psi_\lambda}=K_{\lambda,\mu}$, and is nonzero exactly when $\lambda\unrhd \mu$. 

Now, $e_P \bbC G\simeq \bigoplus_{\lambda\unrhd \mu} (e_P M_\lambda)^{d_\lambda}$, and $(e_Bw_0e_B)^2$ acts as $q^{f_\lambda-n(n-1)}$ on each $(e_P M_\lambda)^{d_\lambda}$, and $\dim(e_P M_\lambda)^{d_\lambda}=d_\lambda K_{\lambda,\mu}$. This proves the theorem.
\end{proof}

\subsection{Springer's Theorem on the Centrality of \texorpdfstring{$(e_B w_0 e_B)^2$}{(eBw0eB)2}}\label{sec:6.3}

In our study of the Hecke algebra $H = e_B \mathbb{C} G e_B$ and its representations, we are particularly interested in the action of the element $a_w^2 = (a_{BwB}^2)=[Bw_0B:B]^2 (e_B w_0 e_B)^2=q^{n(n-1)}(e_B w_0 e_B)^2$ on simple $H$-modules, where $w=w_0$ is the longest element in the Weyl group $W=S_n$ associated with $G=\GL_n(\bbF_q)$.

As we shall see, this element $a_w^2$  lies in the center of $H$. Then it acts as a constant on each simple $H$-module $e_B M_\lambda$. This phenomenon is described in the following theorem, which is a special case of Springer's Theorem 9.2.2 in \cite{geck_pfeiffer2000}. We present an ad-hoc proof for the sake of completeness.

\begin{thm}[Springer]\label{thm:6.3.1}
Let $G = \GL_n(\mathbb{F}_q)$ and $H = e_B \mathbb{C} G e_B$ be as before. Let $w$ be the longest element in $S_n$. Then the element $a_w^2$ lies in the center of $H$.

For each partition $\lambda$ of $n$, let $M_\lambda$ is a simple $\bbC G$-module corresponding to $\lambda\vdash n$. Then this element $a_w^2$ acts by scalar multiplication by $q^{f_\lambda}$ on the simple $H$-module $e_B M_\lambda$, where the  exponent $f_\lambda=\frac{n(n-1)}{2} (1 + \frac{\psi_{\lambda}(s)}{\psi_{\lambda}(1)})$ is an integer and $\psi_{\lambda}$ is the character of $S_n$ corresponding to $\lambda$.
\end{thm}
\begin{proof}

By Proposition \ref{prop:5.3.2} or by the generator-relation description in Section \ref{sec:5.3.2}, the algebra $H$ is generated by $\{a_s: s \in S\}$ as a $\mathbb{C}$-algebra. Then for the first statement, it suffices to prove that $a_w^2$ commutes with each $a_s$.

Take any $s \in S$, and let $t = wsw$. Then $l(t) = 1$ (if $s = (j, j+1)$, then $t = (n-j, n-j-1)$). Because $w$ is the longest element, $\ord(w) = 2$, and the two elements $ws = tw$ and $sw = wt$ both have length $l(w) - 1$.

By Proposition \ref{prop:5.3.2} and Corollary \ref{cor:5.3.3}, $a_s a_w = q a_{sw} + (q-1) a_w = q a_{wt} + (q-1) a_w = a_{wt} = a_w a_t$. Similarly, $a_t a_w = q a_{tw} + (q-1) a_w = q a_{ws} + (q-1) a_w = a_w a_s$. Then $a_s a_w^2 = a_w a_t a_w = a_w^2 a_s$. Therefore, $a_w^2$ commutes with each $a_s$, and so $a_w^2$ lies in the center of $H$.

Now, the module $e_B M_\lambda$ is a simple $H$-module by Proposition \ref{prop:6.1.2}. By Schur's lemma, $a_w^2$ acts as a scalar on $e_B M_\lambda$. It remains to find the scalar.

Let $N_\lambda = e_B M_\lambda$. We first figure out the determinant of $a_s$ on $N_\lambda$ for each simple transposition $s\in S$. Note that $a_s^2 = q a_1 + (q-1) a_s$, so that $(a_s - q)(a_s + 1) = 0$. Then the eigenvalues of action of $a_s$ on $N_\lambda$ are $q$ and $-1$. Suppose $q$ appears with multiplicity $x$ and $-1$ appears with multiplicity $y$. Then $x + y = \dim(N_\lambda)=\Tr(a_1,N_\lambda)$, and $qx - y=\Tr(a_s,N_\lambda)$.

For any $a\in H$, and for any $\bbC G$-module $M$, the action of $a$ on $M=eM\oplus (1-e)M$ is zero on $(1-e)M$. Then $\Tr(a,eM)=\Tr(a,M)$. When $M=M_\lambda$, we have $\Tr(a,N_\lambda)=\Tr(a,M_\lambda)=\Psi_\lambda(a)$. In particular, $x+y=\Psi_\lambda(a_1)$ and $qx-y=\Psi_\lambda(a_s)$.

Then by the Frobenius reciprocity, and the formula $e_{P_s}=\frac{1}{q+1}(a_1+a_s)$, we have
$$\langle \Psi_\lambda, (1_{B})^G \rangle = \angles{\Psi_\lambda|_B, 1_B } =\Psi_\lambda(e_B) = \Psi_\lambda(a_1) = x + y,$$ 
$$\langle \Psi_\lambda, (1_{P_s})^G \rangle = \angles{\Psi_\lambda|_{P_s}, 1_{P_s} } = \Psi_\lambda(e_{P_s}) = \frac{\Psi_\lambda(a_1) + \Psi_\lambda(a_s)}{q + 1} = \frac{x + y + qx - y}{q + 1} = x.$$
On the other hand, $$\langle \psi_\lambda, (1_{\{1\}})^S \rangle = \angles{\psi_\lambda|_{\{1\}}, 1_{\{1\}} } = \psi_\lambda(1),$$
$$\langle \psi_\lambda, (1_{\{1,s\}})^S \rangle = \angles{\psi_\lambda|_{\{1,s\}}, 1_{\{1,s\}} } = \frac{\psi_\lambda(1) + \psi_\lambda(s)}{2}.$$ But the hat map $\psi_\lambda \mapsto \Psi_\lambda$ preserves inner product. Then $x + y = \psi_\lambda(1)$, and $x = \frac{\psi_\lambda(1) + \psi_\lambda(s)}{2}$.

Then $\det((a_s)_r, N_\lambda) = (-1)^y q^x $ for any $s \in S$. Since $a_w$ can be written as a product of $l(w) = \frac{n(n-1)}{2}$ many different elements of the form $a_s$, $\det((a_w^2)_r, N_\lambda) = q^{n(n-1)x}$. The scalar of the action of $a_w^2$ on the $(x+y)$-dimensional space $N_\lambda$ is then $\zeta  q^{\frac{n(n-1)x}{x+y}} = \zeta  q^{\frac{n(n-1)}{2}\cdot \frac{\psi_\lambda(1) + \psi_\lambda(s)}{\psi_\lambda(1)} } = \zeta  q^{f_\lambda}$, for some  root of unity $\zeta$. Here, $f_\lambda= \frac{n(n-1)}{2}\cdot \frac{\psi_\lambda(1) + \psi_\lambda(s)}{\psi_\lambda(1)}$.

That $\zeta$ is real follows from the expression of $\Psi$ as an integer combination of $(1_{P_\lambda})^G$. To show that $\zeta$ is $1$, one needs the deformation theory (see Section \ref{sec:8.1}):

Note that the trace $\Psi_{\lambda}(a_w^2)$ of $a_w^2$ on $\epsilon_\lambda \bbC Ge_B$ equals $\zeta\psi_\lambda(1)=\pm \psi_\lambda(1) q^{f_\lambda}$. By Corollary \ref{cor:8.1.2}, there is a polynomial $g(u)$ such that for each $q$, $\Psi_{\lambda}(a_w^2)=g(q)$ in the case of $G=\GL_n(\bbF_q)$, and that $g(1)=\psi_\lambda(w^2)=\psi_\lambda(1)>0$. 

Then the polynomial $g$ agrees with one of the two polynomials $\pm \psi_\lambda(1) u^{f_\lambda}$ infinitely many times, so it must be that polynomial. Because $g(1)=\phi_\lambda(1)$, we conclude that $g(u)=\psi_\lambda(1) u^{f_\lambda}$. Then the trace $\Phi_\lambda(a_w^2)= \psi_\lambda(1)q^{f_\lambda}$. Therefore, $\zeta=1$.

To show that $f_\lambda$ is an integer,  one note that $\frac{n(n-1)}{2}$ is the number of elements in $S_n$ that are conjugate to $s$. By a result in representation theory, $\frac{n(n-1)}{2} \frac{\psi_\lambda(s)}{\psi_\lambda(1)}$ is an algebraic integer for any irreducible character $\psi_\lambda$ of $S_n$ (See Theorem~9.31 in \cite{curtis_reiner1981}).  Then $f_\lambda$ is an algebraic integer, and since $f_\lambda$ is rational, $f_\lambda$ is in fact an integer.
\end{proof}

In the next subsection, we will figure out the number $d_\lambda$ and $f_\lambda$.

\subsection{Formulae for constants}\label{sec:6.2}

The dimensions of the simple $\mathbb{C} G$-modules $M_\lambda$ are afforded by the $q$-hook length formula, and as a preliminary, we need the concept of $q$-binomial coefficients and Young diagrams.

\begin{defn}
The $q$-binomial coefficient, also known as the Gaussian binomial coefficient, is a $q$-analog of the binomial coefficient. It is denoted by $\qbinom{n}{k}$ or $\genfrac{[}{]}{0pt}{}{n}{k}_q$ and is defined as

$$\qbinom{n}{k} = \frac{[n]_q!}{[k]_q! [n-k]_q!},$$
where $[n]_q!$ is the $q$-factorial, defined as $[n]_q! = [n]_q [n-1]_q \cdots [2]_q [1]_q$, and the $q$-number $[n]_q$ is defined by $[n]_q = 1 + q + q^2 + \cdots + q^{n-1}$. 
\end{defn}

\begin{defn}
Let $\lambda = (\lambda_1, \lambda_2, ..., \lambda_m)$ be a partition of $n$. The Young diagram of $\lambda$ is a graphical representation of $\lambda$, consisting of $n$ boxes arranged in left-justified rows, with $\lambda_i$ boxes in the $i$th row, and with first row drawn from top.
\end{defn}

The conjugate partition $\lambda'$ of $\lambda$ is obtained by flipping the Young's diagram against its principal diagonal.

The hook length $h(b)$ of a box $b$ in the Young diagram of $\lambda$ is a crucial parameter in the $q$-hook formula. For a box $b = (i, j)$ located in the $i$th row and $j$th column of the Young diagram, the hook length is given by $h(b) = \lambda_i + \lambda'_j - i - j + 1$, where $\lambda'_j$ is the $j$-th part of the conjugate partition $\lambda'$ of $\lambda$. See Figure \ref{figure: hook length}.

\begin{figure}[ht]
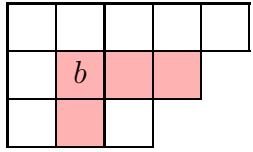

\centering
\begin{ytableau}
\ &  &  &  & \\
 &  *(red!30) b  & *(red!30) & *(red!30) \\
  & *(red!30)  & \\
\end{ytableau}
\caption{The hook length is 4 for the position $b=(2,2)$ }
\label{figure: hook length}
\end{figure}

The $q$-hook formula gives the dimension of the representation $M_\lambda$, see \cite{gnedin_kerov2005} for a reference.
\begin{prop}[$q$-Hook Formula]\label{prop:6.2.1}
The dimension $d_\lambda$ of the representation corresponding to $\Psi_{\lambda}$ is given by:

$$d_\lambda=\dim(\Psi_{\lambda}) = \frac{[n]_q!}{\prod_{b \in \lambda} [h(b)]_q!} q^{n(\lambda)},$$
where the product runs over all boxes $b$ in the Young diagram of $\lambda$, and $n(\lambda) = \sum_{i=1}^{\lambda_1} (i-1) \lambda_i$.
\end{prop}

This formula provides a powerful tool for computing the dimensions of the unipotent representations of $G = \GL_n(\mathbb{F}_q)$.

Now we turn to the scalar $f_\lambda$ and present the Beynon-Lusztig's Theorem, which gives a formula for the scalars $f_\lambda$ for $S_n$. For the proof of this formula, we refer to Beynon and Lusztig's paper \cite[p419]{beynon_lusztig1978} or Theorem 5.4.11 in \cite{geck_pfeiffer2000}.

\begin{thm}[Beynon-Lusztig] \label{thm:6.4.2}
Let $\lambda=(\lambda_1,...,\lambda_r)$ be a partition of $n$, and let $\lambda'=(\lambda'_1,...,\lambda'_s)$ be the conjugate partition. Then
$$\frac{n(n-1)}{2}\frac{\psi_\lambda(s)}{\psi_\lambda(1)}= \sum_i \binom{\lambda_i}{2}-\sum_i \binom{\lambda'_i}{2}.$$
\end{thm}

Therefore $f_\lambda=\frac{n(n-1)}{2}+\frac{n(n-1)}{2}\frac{\psi_\lambda(s)}{\psi_\lambda(1)} =\frac{n(n-1)}{2}+\sum_i \binom{\lambda_i}{2}-\sum_i \binom{\lambda'_i}{2}$.

\begin{rem}
We call the number $\sum_i \binom{\lambda_i}{2}$ the ``row weight" of the partition $\lambda$, denoted by $\wt_r(\lambda)$, because it is computed by assigning the number $\frac{r(r-1)}{2}$ to a row with $r$ elements and summing up, or by filling the number $j-1$ in each column $j$ box in the Young's diagram of $\lambda$ and then summing up. Similarly, the number $\sum_i \binom{\lambda'_i}{2}$ is called the ``column weight" of $\lambda$, denoted by $\wt_c(\lambda)$. 

Note that $n(\lambda)$ in the $q$-Hook Formula (Proposition \ref{prop:6.2.1}), computed by filling the number $i-1$ in each row $i$ box, is also the column weight $\wt_c(\lambda)$. The right hand side of the Beynon-Lusztig formula (Theorem \ref{thm:6.4.2}) equals the difference between the row weight and column weight, which we may call the ``mixed weight" of the partition $\lambda$.

In \cite[\S5.4]{geck_pfeiffer2000}, the row weight $\wt_r(\lambda)$ is called the $a^*$-invariant of the partition $\lambda$, denoted by $a^*(\lambda)$, and the column weight $\wt_c(\lambda)$ is denoted by $a(\lambda)$. Moreover, $a(\lambda)$ equals the $b$-invariant of the corresponding representation $\psi_\lambda$ of $S_n$: $\wt_c(\lambda)=a(\lambda)=b_{\psi_\lambda}$. This highlights the importance of these two weights.

\end{rem}

\section{Zeta Functions in the Generic Case, Kostka Numbers}\label{sec:7}
We are now able to address the question of expressing the zeta function in the generic case. First, we will present the result for the graph $X_0$, then for $X_2$, and finally we will compute the Kostka numbers. In this section and the next, the notations for $G,B,W,S$ are the same as in Section \ref{sec:6}, and $\overline{B}$ denotes the set of lower triangular matrices in $G$.

We begin with a proposition at the core of the generic cases for $X_0$ and $X_2$. 
\begin{prop}\label{prop:7.0.1}
Let $I\subseteq S$, $P=P_I=B W_I B$, and $P'=\overline{P}_I=\overline{B} W_I \overline{B}$. Let $\mu=(\mu_1,...,\mu_s)$ be a partition of $n$ so that $W_I$ is conjugate to $S_\mu$ in $W=S_n$. Let $a_{P,P'}=\frac{1}{|P'|}\sum_{x\in PP'}x$ and $a_{P',P}=\frac{1}{|P|}\sum_{x\in  P'P}x$. Then the eigenvalues of the right multiplication of $a_{P,P'}a_{P',P}$ on $\bbC Ge_{P}$ are parametrized by partitions $\lambda\unrhd \mu$, with each $\lambda$ providing the eigenvalue $q^{f_\lambda-2 \wt_r(\mu)}$ with multiplicity $d_\lambda K_{\lambda,\mu}$. The number $f_\lambda-2\wt_r(\mu)$ is a non-negative integer.
\end{prop}

\begin{proof} Let $L=P\cap P'$. Then $L=\GL_{\mu'_1}(\bbF_q)\times ...\times \GL_{\mu'_s}(\bbF_q)$ consists of block diagonal matrices in $G$ with block sizes prescribed by some permutation $\mu'$ of $\mu$. Observe that $P=LB=BL$, and $P'=L\overline{B}=\overline{B}L$. Then $e_P=e_Le_B=e_Be_L$, and $e_{P'}=e_Le_{\overline{B}}=e_{\overline{B}} e_L$. In particular, $e_L$ commutes with $e_B$ and $e_{\overline{B}}$.

Moreover, $a_{P,P'}=\frac{|PP'|}{|P'|} e_{PP'}=\frac{|P|}{|P'\cap P|}e_{PP'}=q^{\sum_{i<j}\mu_i\mu_j}e_{PP'}$, and $a_{P',P}=q^{\sum_{i<j}\mu_i\mu_j}e_{P'P}$. Then $a_{P,P'} a_{P',P}=q^{2\sum_{i<j}\mu_i\mu_j}e_{PP'}e_{P'P}$.
But $e_{PP'}e_{P'P}=e_{P}e_{P'}e_{P}=e_{B}e_{L}e_{\overline{B}}e_L e_Be_L=e_L e_Be_{\overline{B}} e_B=e_L (e_B w_0 e_B)^2$, and the right action of $e_L$ is trivial on $\bbC Ge_P$ since $e_Pe_L=e_{PL}=e_P$. Then the right action of $a_{P,P'} a_{P',P}$ on $\bbC Ge_P$ is the same as the right action of $q^{2\sum_{i<j}\mu_i\mu_j} (e_Bw_0e_B)^2$ on $\bbC Ge_P$, and has the same eigenvalues as the left action of $q^{2\sum_{i<j}\mu_i\mu_j}(e_Bw_0e_B)^2$ on $e_P \bbC G$.

By Proposition \ref{prop:6.1.2}, the eigenvalues of this action are afforded by partitions $\lambda \unrhd \mu$, with each $\lambda$ providing the eigenvalue $q^{f_\lambda-n(n-1)+2\sum_{i<j}\mu_i\mu_j}$ to the multiplicity $d_\lambda K_{\lambda,\mu}$. Now, $n=\sum \mu_i$, and $-n(n-1)+2\sum_{i<j}\mu_i\mu_j=\sum{\mu_i}-\sum{\mu_i^2}=-2 \wt_r(\mu)$. 

The number $f_\lambda-2\wt_r(\mu)=\frac{n(n-1)}{2}+\wt_r(\lambda)-\wt_c(\lambda)-2\wt_r(\mu)$ is an integer. To show that it is non-negative, we use two properties of weights: 

Let $\lambda,\mu$ be partitions of the same positive integer $n$. (1) If $\lambda\unrhd\mu$, then $\wt_r(\lambda)\geq \wt_r(\mu)$, and $\wt_c(\lambda)\leq \wt_c(\mu)$. (2) $\wt_r(\lambda)+\wt_c(\lambda)\leq \frac{n(n-1)}{2}$.  Both properties are easy combinatorical exercises and we omit their proofs. Then $f_\lambda-2\wt_r(\mu)\geq \frac{n(n-1)}{2}+\wt_r(\lambda)-\wt_c(\lambda)-2\wt_r(\lambda)=\frac{n(n-1)}{2}-\wt_r(\lambda)-\wt_c(\lambda)\geq 0$ as desired.
\end{proof}

\subsection{Zeta Function of \texorpdfstring{$X_0$}{X0} in the Generic Case}\label{sec:7.1}
We find the zeta function $\frac{1}{Z_c(X_0^{[k]},u)}$ in the generic case.
\begin{thm}\label{thm:7.1.1}
For the graph $X_0(\mathbb{F}_q^n)$, if $k\neq \frac{n}{2}$, $\frac{1}{Z_c(X_0^{[k]},u)}$  can be expressed as a product of factors of the form $(1-q^su^2)$, where $s\in \mathbb{Z}_{\geq 0}$.
\end{thm}
Let's consider the setting in Section \ref{sec:4.1}. Let $V=\mathbb{F}_q^n$, and $W_0= \langle e_1,...,e_k \rangle$, with $1\leq k\leq n$. Let $W_1= \langle e_{k+1},...,e_n \rangle$, and $P_0=\text{Stab}(W_0)$, $P_1=\text{Stab}(W_1)$. Let $I=\{s_1,...,s_{k-1}, s_{k+1},...,s_{n-1}\}$. Then $P_0=P_I=BW_IB$, $P_1=\overline{P}_I=\overline{B}W_I \overline{B}$. The partition $\mu$ corresponding to $I$ is $(n-i,i)$ where $i=\min(k,n-k)$. Except for the formula on zeta function, our discussion also applies when $k=\frac{n}{2}$.

By Proposition \ref{prop:3.3.4 Main prop on group zeta} and the first case in Section \ref{sec:3.4}, we have for $k\neq \frac{n}{2}$,
$$Z_c(X_0^{[k]},u) = \prod_{t=1}^M \frac{1}{(1-q_t u^2)},$$
where $q_1,q_2,...,q_M$ are eigenvalues of the map $D(W_0,W_0,2)_r=(a_{P_I,\overline{P}_I}a_{\overline{P}_I P_I})_r:\bbC Ge_{P_I}\to \bbC Ge_{P_I}$. 

The eigenvalues are afforded by proposition \ref{prop:7.0.1}, and parametrized by $\lambda\unrhd \mu$. The partitions $\lambda$ such that $\lambda\unrhd \mu$ along with corresponding $K_{\lambda,\mu}$, $f_\lambda$, and $d_\lambda$ are displayed in the table \ref{Table:2} below. The values $K_{\lambda,\mu}$ and the formula for $d_\lambda$ are derived from Section \ref{sec:7.3}. One may also compute the $d_\lambda$ using the $q$-hook formula in Section \ref{sec:6.2}.

\begin{table}[h]
\centering
\renewcommand{\arraystretch}{1.5}
\begin{tabular}{|c|c|c|c|c|c|c|}
\hline
$ \lambda $ & $ (n) $ & $(n-1,1)$& $ \dots $ & $ (n-j,j) $ & $ \dots $ & $ (n-i,i) $ \\ \hline
$ K_{\lambda,\mu} $ & 1 & 1 & $ \dots $ & 1 & $ \dots $ & 1 \\ \hline
$ f_\lambda $ & $n^2-n$ & $n^2-2n$ & $ \dots $ & $n^2+j^2-nj-n-j$ & $ \dots $ & $n^2+i^2-ni-n-i$ \\  \hline
$ d_\lambda $ & $1$ & $q [n-1]_q$ & $ \dots $ & $\qbinom{n}{j}-\qbinom{n}{j-1}$ & $ \dots $ & $\qbinom{n}{i}-\qbinom{n}{i-1}$ \\  \hline
\end{tabular}
\caption{Partitions $ \lambda \unrhd \mu $ with $ K_{\lambda,\mu},f_\lambda,d_\lambda$ for $\mu=(n-i,i)$}\label{Table:2}
\end{table}

For each $0\leq j\leq i$, let $\lambda_j=(n-j,j)$, with $\lambda_0=(n)$. For the action of $D(W_0,W_0,2)_r$ on $\bbC Ge_{P_0}$, $\lambda_j$ contributes to the eigenvalue $q^{f_{\lambda_j}-2\wt_r(\mu)}$ with multiplicity $d_{\lambda_j}K_{\lambda_j,\mu}=d_{\lambda_j}\geq 0$. The number $f_{\lambda_j}-2\wt_r(\mu)=-j(n-j+1)+2i(n-i)\geq 0$ by the Proposition $\ref{prop:7.0.1}$. We write $d_j$ for $d_{\lambda_j}$, and $f_j$ for $f_{\lambda_j}$. 
We can now prove Theorem \ref{thm:7.1.1}.

\begin{proof}[Proof of Theorem 7.1.1]
Suppose that $1\leq k\leq n-1$ and $k\neq \frac{n}{2}$. Let  $i=\min{(k,n-k)}$. By the preceding formula for the zeta function,
\begin{equation}
    \frac{1}{Z_c(X_0^{[k]},u)}=\prod_{j=0}^{i} {(1-q^{f_j-2\mu} u^2)^{d_j}}=\prod_{j=0}^{i}{(1-q^{-j(n-j+1)+2i(n-i)} u^2)^{\qbinom{n}{j}-\qbinom{n}{j-1}}}.
\end{equation}
 
This shows that the inverse zeta component $\frac{1}{Z_c(X_0^{[k]},u)}$ is a product of factors of the form $(1-q^s u^2)$, where $s\in \bbZ_{\geq 0}$. 
\end{proof}

The values for various $1/Z_c(X_0^{\{n-i,i\}},u)$ for $n\leq 5$ are summarized in Table \ref{table:7.1.2} below, where the cases for $i=\frac{n}{2}$ are computed using the formula in Section \ref{sec:8.5}. The value of $1/Z_c(X_0(\bbF_q^n),u)$ is the product of all terms in the corresponding row.
\begin{table}[h]
\centering
\begin{tabular}{|c||c|c|c|}
\hline
$n$ &  $i=0$ & $i=1$ & $i=2$ \\
\hline
2 & $1 - u^2$ & $(1 + u)^{\qbinom{2}{1}-1} (1 -qu)$ & \\
3 & $1 - u^2$ & $(1 - q u^2)^{\qbinom{3}{1}-1} (1 - q^4 u^2)$ & \\
4 & $1 - u^2$ & $(1 - q^2 u^2)^{\qbinom{4}{1}-1} (1 - q^6 u^2)$ & $(1 - qu)^{\qbinom{4}{2}-\qbinom{4}{1}} (1 + q^2 u)^{\qbinom{4}{1}-1} (1 - q^4 u)$ \\
5 & $1 - u^2$ & $(1 - q^3 u^2)^{\qbinom{5}{1}-1} (1 - q^{8} u^2)$ & $(1 - q^4 u^2)^{\qbinom{5}{2}-\qbinom{5}{1}} (1 - q^7 u^2)^{\qbinom{5}{1}-1} (1 - q^{12} u^2)$ \\
\hline
\end{tabular}
\caption{Expressions of $1/Z_c(X_0^{\{n-i,i\}},u)$ }
\label{table:7.1.2}
\end{table}

\subsection{Zeta Function of \texorpdfstring{$X_2$}{X2} in the Generic Case}\label{sec:7.2}
We find the zeta function $\frac{1}{Z_c(X_2^{[(a,b)]},u)}$ in the generic case. 
\begin{thm}\label{thm:7.2.1}
For the digraph $X_2(\mathbb{F}_q^n)$, if $(a,b)\neq (\frac{n}{3},\frac{2n}{3})$ or $(\frac{2n}{3},\frac{n}{3})$, $\frac{1}{Z_c(X_2^{[(a,b)]},u)}$  can be expressed as a product of factors of the form $(1-q^su^6)$, where $s\in \mathbb{Z}_{\geq 0}$.
\end{thm}
\begin{proof}
We consider the setting in Section \ref{sec:4.2}. Given integers $i,j,k>0$ such that $n=i+j+k$, let the index sets $I,J,K$, vector spaces $V_I,V_J,V_K$, and flags $F_0$ to $F_5$  be defined as there. The parabolic groups $P_0$ to $P_5$ are defined by $P_s=\Stab(F_s)$, $s=0,1,...,5$.

Given $(a,b)$ as stated in the theorem, we choose $(i,j,k)$ as follows:
$$(i,j,k)=\begin{cases}
  (a, b-a, n-b), & \text{ if } a<b,\\
 (a-b, b, n-a), & \text{ if } a>b.
\end{cases}$$
Then $(i,j,k)\neq (\frac{n}{3},\frac{n}{3},\frac{n}{3})$, and the multi-dimensions of $F_0-F_5$ are mutually distinct.

We first find the eigenvalues of the right action of $D(F_0,F_0,6)=a_{P_0,P_3}a_{P_3,P_0}$ on $\bbC Ge_{P_0}$. Note that $P_0$, $P_3$, and $\mu=(i_1,j_1,k_1)$ takes the role of $P$, $P'$ and $\mu$ in Proposition \ref{prop:7.0.1}, where $(i_1,j_1,k_1)$ are the numbers $i,j,k$ in weakly decreasing order. The eigenvalues are then afforded by that proposition: each $\lambda\unrhd\mu$ contributes the eigenvalue $q^{f_\lambda-2\wt_r(\mu)}$ with multiplicity $d_\lambda K_{\lambda,\mu}$. Then one could plug these eigenvalues in the formula in Proposition \ref{prop:3.3.4 Main prop on group zeta} and Case 3 in Section \ref{sec:3.4}, and gets
\begin{equation}
    \frac{1}{Z_c(X_2^{[(a,b)]},u)}=\prod_{\lambda\unrhd \mu} (1-q^{f_\lambda-2\wt_r(\mu)} u^6)^{d_\lambda K_{\lambda,\mu}}.
\end{equation}

Note that the numbers $f_\lambda-2\wt_r(\mu)$ and $d_\lambda K_{\lambda,\mu}$ are non-negative integers, thereby proving the theorem.
\end{proof}

The number $-2\wt_r(\mu)=2ij+2jk+2ik-n(n-1)$. We list the numbers $K_{\lambda,\mu}$, $f_\lambda$, and $d_\lambda$ for partitions $\lambda$ such that $\lambda\unrhd \mu$. We will calculate the formula for $K_{\lambda,\mu}$ in Section 7.3, and deduce the formula for $d_\lambda$ in Section 7.4.

\begin{table}[h]
\centering
\renewcommand{\arraystretch}{1.5}
\begin{tabular}{|c|c|}
\hline
$ \lambda $ &   $ (x,y,z): x\geq i_1, x+y\geq i_1+j_1 $  \\ \hline
$ K_{\lambda,\mu} $ &  $\min{(x-i_1,j_1)}-\max{(0,j_1-y,y-i_1)}+1$  \\ \hline
$ f_\lambda $ &   $\binom{n}{2}+\binom{x}{2}+\binom{y}{2}+\binom{z}{2}-2z-y$  \\  \hline
$ d_\lambda $ & $\qbinom{n}{x,y,z}-\qbinom{n}{x+1,y-1,z}+\qbinom{n}{x+2,y-1,z-1}-\qbinom{n}{x,y+1,z-1}+\qbinom{n}{x+1,y+1,z-2}-\qbinom{n}{x+2,y,z-2}$  \\  \hline
\end{tabular}
\caption{Partitions $ \lambda \unrhd \mu $ with $ K_{\lambda,\mu},f_\lambda,d_\lambda$ for $\mu=(i_1,j_1,k_1)$}\label{table:3}
\end{table}
Here $\qbinom{n}{i,j,k}=\frac{[n]_q!}{[i]_q![j]_q![k]_q!}$, and if some of $i,j,k<0$, the number $\qbinom{n}{i,j,k}$ is understood as 0.

The following table \ref{table:7.2.2} summarizes the expression $1/Z_c(X_2^{[(i,i+j)]}(\bbF_q^{i+j+k}),u)$ for various $(i,j,k)$ with $i+j+k\leq 6$. The first entry for $i=j=k=1$ is computed as in Section \ref{sec:8.4}.
\begin{table}[h]\renewcommand{\arraystretch}{1.2}
\centering
\begin{tabular}{|c||c|}
\hline
$(i,j,k)$ & Expression $1/Z_c(X_2^{[(i,i+j)]},u)$ \\
\hline
(1, 1, 1) & $(1 - u^2)^{q^3}  (1 - q \omega u^2)^{q (1 + q)} (1 - q  \omega^2 u^2)^{q (1 + q)}(1 - q^2 u^2)$ \\
(1, 1, 2) & $(1 - q^2 u^6)^{q^3 (1 + q + q^2)} (1 - q^4 u^6)^{(\qbinom{4}{2}-\qbinom{4}{1})} (1 - q^6 u^6)^{2 (\qbinom{4}{1}-1)} (1 - q^{10} u^6)$ \\
(1, 1, 3) & $(1 - q^4 u^6)^{d_{(3,1,1)}} (1 - q^6 u^6)^{d_{(3,2)}} (1 - q^9 u^6)^{2 d_{(4,1)}} (1 - q^{14} u^6)$ \\
(1, 2, 2) & $(1-q^4 u^6)^{d_{(2,2,1)}} (1-q^6 u^6)^{d_{(3,1,1)}} (1-q^8 u^6)^{2 d_{(3,2)}} (1-q^{11} u^6)^{2 d_{(4,1)}} (1-q^{16} u^6)$ \\
(1, 1, 4) & $(1 - q^6 u^6)^{d_{(4,1,1)}} (1 - q^8 u^6)^{d_{(4,2)}} (1 - q^{12} u^6)^{2 d_{(5,1)}} (1 - q^{18} u^6)$ \\
(1, 2, 3) & $(1 - q^7 u^6)^{d_{(3,2,1)}} (1 - q^{10} u^6)^{d_{(3,3)}+d_{(3,2,1)}} (1 - q^{12} u^6)^{2 d_{(4,2)}} (1 - q^{16} u^6)^{2 d_{(5,1)}} (1 - q^{22} u^6)$ \\
(1, 3, 2) & $(1 - q^7 u^6)^{d_{(3,2,1)}} (1 - q^{10} u^6)^{d_{(3,3)}+d_{(3,2,1)}} (1 - q^{12} u^6)^{2 d_{(4,2)}} (1 - q^{16} u^6)^{2 d_{(5,1)}} (1 - q^{22} u^6)$ \\
(2, 2, 2) & ... \\
\hline
\end{tabular}
\caption{Expressions of $1/Z_c(X_2^{[(i,i+j)]},u)$ for various $(i,j,k)$}
\label{table:7.2.2}
\end{table}

\subsection{Kostka Numbers: Count of Semistandard Tableaux}\label{sec:7.3}

Kostka numbers are integers that capture the combinatorial essence of the decomposition of characters for symmetric groups and general linear groups. They also count the number of semistandard generalized tableaux of a specific shape and type. We refer to \cite[\S 2.9]{sagan2001}.

\subsubsection{Generalized and Semistandard Tableaux}
Let's introduce the concept of a generalized Young tableau. 
\begin{defn}
    For a given partition $\lambda$, a generalized Young tableau of shape $\lambda$ is an array, $T$, created by filling the cells of the Young diagram corresponding to $\lambda$ with positive integers, with repetition allowed.

    Each tableau $T$ has an associated type, represented by an array $\mu = (\mu_1, \mu_2, ..., \mu_m)$. Each $\mu_i$ is the count of the integer $i$ within the tableau $T$.
\end{defn}

We denote the collection of all tableaux of a specific shape $\lambda$ and type $\mu$ as $\calT_{\lambda, \mu}$, defined as:
\[\calT_{\lambda, \mu} = \{ T : T \text{ has shape } \lambda \text{ and type } \mu \}.\]

For example
\[ T = \begin{ytableau}
    3 & 2 & 4 \\ 2
\end{ytableau} \] 
is a generalized Young tableau of shape $(3,1)$ and type $(0,2,1,1)$.

\begin{defn}
A semistandard generalized tableau is a generalized tableau whose rows weakly increase (entries may repeat) and whose columns strictly increase. We denote by $\calT_{\lambda, \mu}^0$ the set of semistandard $\lambda$-tableaux of type $\mu$.
\end{defn}

For example, the tableau
\[
T=\begin{ytableau}
1 & 1 & 1 & 2 & 2\\
2 & 3\\
3
\end{ytableau}
\]
is a semistandard tableau of type $(5,2,1)$ and content $(3,3,2)$. 

\subsubsection{Kostka Numbers}
The Kostka number $K_{\lambda, \mu}$, introduced in Definition \ref{def:5.2.3}, where $\lambda$ and $\mu$ are partitions of the same number $n$, exactly equals the number of semistandard $\lambda$-tableaux of type $\mu$ \cite[\S 2.11]{sagan2001}. That is,
\[
K_{\lambda, \mu} = |\calT_{\lambda, \mu}^0|.
\]

It's worth noting that if $\mathcal{T}_{\lambda, \mu}^0$ is nonempty, then the partition $\lambda$ dominates $\mu$. This is because, if $T\in \mathcal{T}_{\lambda, \mu}^0$, the numbers $1,2,...,a$ should appear in the first $a$ rows of $T$, because the columns of $T$ are strictly increasing. Therefore, $\lambda_1+...+\lambda_a \geq \mu_1+...+\mu_a$, and hence $\lambda \unrhd \mu$. Conversely, $\lambda \unrhd \mu$ implies that $\mathcal{T}_{\lambda, \mu}^0$ is nonempty (See \cite{wildon_mathoverflow} for a combinatorical proof). Note that this proves Proposition \ref{prop:5.2.4}.

For instance, consider $\lambda=(x,n-x)$ and $\mu=(y,n-y)$. If $T\in \mathcal{T}_{\lambda,\mu}^0$, then all the 1's appear in the first row and the remaining entries are filled with 2's. Then $K_{\lambda,\mu}=1$ when $x\geq y\geq \frac{n}{2}$. 

Recall from Definition \ref{def:5.2.3} that Kostka numbers appear as multiplicities of certain characters of $S_n$, and that $\phi_\mu=\sum_{\lambda\unrhd\mu}K_{\lambda,\mu}\psi_\mu$, where $\phi_\mu$ denotes the character $(1_{S_\mu})^{S_n}$. 

Then  $\phi_{(n-i,i)}= \psi_{(n)}+\psi_{(n-1,1)}+...+\psi_{(n-i,i)}$ for $i\leq \frac{n}{2}$. Therefore, for $1\leq i\leq \frac{n}{2}$, $\psi_{(n-i,i)}=\phi_{(n-i,i)}-\phi_{(n-i+1,i-1)}$. Applying the hat map, one gets $\Psi_{(n-i,i)}=\Phi_{(n-i,i)}-\Phi_{(n-i+1,i-1)}$. Note that $\dim(\Phi_{(n-i,i)})=\dim((1_{P_{(n-i,i)}})^{G})=\dim(\bbC Ge_{P_{(n-i,i)}})=\Gr(n,i)=\qbinom{n}{i}$. Hence, the dimension $d_{(n-i,i)}$ of $\Psi_{(n-i,i)}$ equals $\qbinom{n}{i}- \qbinom{n}{i-1}$, as promised in Section \ref{sec:7.1}.

Now, let's consider $\lambda = (x, y, z)$ and $\mu = (i, j, k)$ with $x\geq y\geq z\geq 0$ and $i\geq j\geq k\geq 0$. Assuming $\lambda \unrhd \mu$, we have $x\geq i$ and $x+y\geq i+j$. We have computed the Kostka number $K_{\lambda, \mu}$ by elementary counting of semistandard Young tableaux. The result is  $K_{\lambda,\mu} = K_{(x-z,y-z),(i-z,j-z,k-z)} = \min{(x-i,j-z)} - \max{(0,j-y,y-i)} + 1$. This gives the formula for $K_{\lambda,\mu}$ in Table \ref{table:3}.

For example, the following table \ref{table:5} illustrates the results of $K_{\lambda,\mu}$ with $\mu = (i,4,2)$, $\lambda = (x,y,z)$, and $i\geq6$.

\begin{table}[h]
\centering
\renewcommand{\arraystretch}{1.5}
\begin{tabular}{|c|c|c|c|c|c|c|c|c|}
\hline
$z\backslash y$ & 0 & 1 & 2 & 3 & 4 & 5 & 6 \\
\hline
0 & 1 & 2 & 3 & 3 & 3 & 2 & 1 \\
\hline
1 & & 1 & 2 & 2 & 2 & 1 & \\
\hline
2 & & & 1 & 1 & 1 & & \\
\hline
\end{tabular}
\caption{$K_{\lambda,\mu}$ with $\mu = (i,4,2)$, $\lambda = (x,y,z)$, and $i\geq6$.}\label{table:5}
\end{table}
\subsection{Inverse Kostka Numbers and Jacobi-Trudi Identity}\label{sec:7.4}

In this subsection, we will introduce the Jacobi-Trudi Identity to give an expression of $\psi_\mu$ in terms of $\phi_\lambda$. 

\begin{thm}[Jacobi-Trudi Identity, representation form]\label{thm:7.4.2}
Let $\lambda = (\lambda_1, \lambda_2, \ldots, \lambda_m)$ be a partition. Define the notation $\phi_{k_1} \circ \phi_{k_2} \circ \cdots \circ \phi_{k_s} = \phi_{(k_1, k_2, \ldots, k_s)}$, where $\phi_k = (1_{S_k})^{S_n}$ is the induced character. Then the irreducible character $\psi_\lambda$ can be expressed as a determinant:
$$\psi_\lambda = \det(\phi_{\lambda_i+j-i})_{1 \leq i,j \leq m},$$
where multiplication is understood in the sense of $\circ$ and $\phi_k = 0$ for $k < 0$.
\end{thm}

\begin{rem}
The Jacobi-Trudi identity is a formula that relates two bases $s_\lambda$ (Schur polynomials) and $h_\lambda$ (homogeneous polynomials) of the ring of symmetric functions. The homogeneous polynomials satisfy $h_\lambda=h_{\lambda_1}\cdots h_{\lambda_m}$ for $\lambda=(\lambda_1,\ldots,\lambda_m)$, and $h_a=0$ for $a\in \bbZ_{<0}$. The usual form of Jacobi-Trudi identity states that $$s_\mu=\det(h_{\mu_i+j-i})_{1 \leq i,j \leq m}.$$

The $s_\lambda$ and $h_\lambda$ are also related by inverse Kostka numbers 
$s_\mu=\sum_\lambda K^{-1}_{\lambda,\mu} h_\lambda.$ This formula has the same shape as the formula $\psi_\mu=\sum_{\lambda} K^{-1}_{\lambda,\mu} \phi_\lambda$.
Therefore, the two forms of Jacobi-Trudi identities are easily seen as equivalent.  

See \cite{egecioglu_remmel1990} for more on Schur polynomials and the usual Jacobi-Trudi identity, and see \cite[p57]{neretin2013} for the representation form of this identity.
\end{rem}

For example, we can apply this formula to compute $\psi_\mu$ when $\mu$ has 3 parts:

\begin{prop}\label{prop:7.4.3}
Let $\mu=(i,j,k)$ be a partition with $k\geq 1$. Then
$$\psi_{(i,j,k)} = (\phi_{(i,j,k)} + \phi_{(i+2,j-1,k-1)} + \phi_{(i+1,j+1,k-2)}) - (\phi_{(i,j+1,k-1)} + \phi_{(i+1,j-1,k)} + \phi_{(i+2,j,k-2)}),$$
where $\phi_{(a,b,c)}=(1_{S_{(a,b,c)}})^{S_n}$ is the induced character when $a,b,c\geq 0$. If any one of $a,b,c$ is less than $0$, the expression $\phi_{(a,b,c)}$ is understood as 0.
\end{prop}

\begin{proof}
Apply Theorem \ref{thm:7.4.2} to the partition $\mu=(i,j,k)$:
$$\psi_{(i,j,k)} = \det\begin{pmatrix}
\phi_i & \phi_{i+1} & \phi_{i+2} \\
\phi_{j-1} & \phi_j & \phi_{j+1} \\
\phi_{k-2} & \phi_{k-1} & \phi_k
\end{pmatrix}.$$
Expanding the determinant yields the desired expression for $\psi_{(i,j,k)}$.
\end{proof}

Applying the hat map from characters of $S_n$ to unipotent characters of $G=\GL_n(\bbF_q)$ yields:
$$\Psi_{(i,j,k)} = (\Phi_{(i,j,k)} + \Phi_{(i+2,j-1,k-1)} + \Phi_{(i+1,j+1,k-2)}) - (\Phi_{(i,j+1,k-1)} + \Phi_{(i+1,j-1,k)} + \Phi_{(i+2,j,k-2)}),$$
where $\Phi_{(a,b,c)}=(1_{BS_{(a,b,c)}B})^G$ equals the image of $\phi_{(a,b,c)}$ under the hat map.

The first application of this calculation is the dimension formula $d_\lambda$ in Table \ref{table:3} in Section \ref{sec:7.2}.
Note that the dimension of $\Phi_{(a,b,c)}$ is $\qbinom{a+b+c}{a,b,c}$. Then we get the formula of $d_\lambda=\dim(\Psi_{x,y,z})$ as in Table \ref{table:3}.

This calculation also paves the way for the special case analysis, as we will see in Section \ref{sec:8.2}.
\section{Zeta Functions in the Special Case}\label{sec:8}

In this section, we consider the special case where all $i,j,k$ are equal, and compute the zeta functions $Z_c(X_0^{[\frac{n}{2}]}(\mathbb{F}_q^n), u)$ and $Z_c(X_2^{[(\frac{n}{3},\frac{n}{3})]}(\mathbb{F}_q^n), u)$ explicitly.

\subsection{Deformation Principle}\label{sec:8.1}

In this subsection, we will continue the discussion of Section \ref{sec:5.3}, and relate the representations of the Hecke algebra $H_q=e_B\bbC \GL_n(\bbF_q) e_B$ with the representations of $\bbC S_n$. In particular, we will explain how the process of ``setting $q=1$" relates representations of $H_q$ and $\bbC S_n$.
In this subsection, we let $W=S_n$, $S=\{s_1=(1,2),s_2=(2,3),...,s_{n-1}=(n-1,n)\}$ as in Section \ref{sec:5}.

\paragraph{Hecke Algebras}
In the following, let $q$ be a prime power or 1. Then $H_q=H_{q,S}$ is the $\bbC$-algebra generated by $\{T_s:s\in S\}$ with the Braid and quadratic relations as in Section \ref{sec:5.3.2}:
\begin{itemize}
    \item Braid relations: For any $s,t\in S$, $(T_s T_t)^{m_{st}}=(T_t T_s)^{m_{st}}$, where $m_{st}=3$ if $t=s\pm 1$ and  $m_{st}=2$ otherwise.
    \item Quadratic relations: For any $s\in S$, $T_s^2=(q-1)T_s+q$.
\end{itemize}

For any $w\in W=S_n$, let $w=s_1...s_r$ be a reduced expression. Then we can define $T_w$ as $T_{s_1}...T_{s_n}$. Such a definition gives $T_w$ uniquely, by Braid relation and Mastumoto's theorem. The elements $\{T_w:w\in W\}$ form a basis of $H_q$ over $\bbC$.

For $q>1$ prime power, $H_q$ is identified with $e_B \bbC Ge_B$, where $G=\GL_n(\bbF_q)$, and $B=\UT_n(\bbF_q)$ as before, with $T_w=a_w=\frac{1}{|B|}\sum_{x\in BwB}x$. For $q=1$, $H_1$ is identified with $\bbC S_n$, with $T_w=w$.

The algebras $\{H_q\}$ can be viewed as a family of $\bbC$-algebras parametrized by $q$. Setting different values of $q$ gives information of representations of $\GL_n(\bbF_q)$ and $S_n$.

\paragraph{Induced Modules}
Now we discuss some representations of $H_q$, arising from the induction process.

Let $I\subseteq S$, we may  consider the subalgebra $H_{q,I}$ of $H_q$ generated by $\{T_s:s\in I\}$.  
Then $H_{q,I}$, as a vector space, is spanned by $\{T_w:w\in W_I\}$ over $\bbC$.
For $q>1$, $H_{q,I}$ corresponds to the subalgebra $e_B \bbC P_I e_B$ of $e_B\bbC G e_B$, and for $q=1$, $H_{q,I}$ corresponds to $\bbC S_I$.

Consider the index representation $\bbC$ of $H_{q,I}$, where each $T_w$ acts as $q^{l(w)}$. One can induce a representation $\rho_{q,I}$ of $H_q$ given by the module $M_{q,I}:=H_q\otimes_{H_{q,I}} \bbC$.
Take a set $R$ of distinguished left coset representatives  of $W_I$ in $W$. (This means for each left coset $xW_I$, we choose the unique element of shortest length, see \cite[\S 2.1]{geck_pfeiffer2000}) Then $\{T_r\otimes 1: r\in R\}$ is a $\bbC$-basis of $M_{q,I}$. For $s\in S$, we calculate $T_s(T_r\otimes 1)=(T_s T_r)\otimes 1$. 

If $l(sr)<l(r)$, the element $sr$ must still lie in $R$, and $(T_s T_r)\otimes 1 = (qT_{sr}+(q-1)T_r)\otimes 1 =qT_{sr}\otimes 1+(q-1)T_r\otimes 1$. If $l(sr)>l(r)$, by Deodhar's Lemma [GK00,2.1.2],  either $sr\in R$ or $sr=rv$ for some $v\in I$. Then $(T_s T_r)\otimes 1=T_{sr}\otimes 1$ in the former case, while $(T_s T_r)\otimes 1=T_rT_v\otimes 1=T_s\otimes (T_v.1)=q T_s \otimes 1$ in the latter case. 

Then, with respect to the basis $\calB=\{T_r\otimes 1: r\in R\}$, the matrices of $\{\rho_{q,I}(T_s):s\in S\}$ for different $q$ can be uniformly described by a matrix with entries in $\bbZ[u]$ by the specialization $u\mapsto q$. By this, we mean that there exists a matrix $A_s(u)\in M_{|R|}(\bbZ[u])$, such that for any $q$ which is a prime power or 1, the matrix of $\rho_{q,I}(T_s)$ with respect to the basis $\calB$ is $A_s(q)$.  We can now make clear the setting $q\mapsto 1$ process.

\begin{prop}\label{prop:8.1.1}
Let $\Phi_{q,I}$ be the character of the representation $\rho_{q,I}(T_w)$. Let $w_1,...,w_m$ be elements of $W$. Let $T=T_{w_1}...T_{w_m}$. Then there exists a unique polynomial $f(u)\in \bbZ[u]$, such that for any $q$ which is a prime power or 1, $\Phi_{q,I}(T)=f(q)$.
\end{prop}

\begin{proof}
For each $w\in W$, let $w=s_1...s_k$ be a reduced expression of $w$. Then $T_w=T_{s_1}...T_{s_k}$. Let $A_w(u)=A_{s_1}(u)...A_{s_k}(u)$, where $A_{s}(u)$ are matrices in $ M_{|R|}(\bbZ[u])$ as above. Then for each $q$, the matrix of $\rho_{q,I}(T_w)$ is $A_w(q)$. Let $A(u)=A_{w_1}(u)...A_{w_m}(u)$. Then the matrix of $\rho_{q,I}(T)$ is $A_{w_1}(q)...A_{w_m}(q)=A(q)$.

Therefore, for each $q$, $\Phi_{q,I}(T)=\Tr(\rho_{q,I}(T))=\Tr(A(u))$. Take $f(u)=\Tr(A(u))$. Then $f(u)$ is an integer coefficient polynomial satisfying the requirement. Such an $f$ is clearly unique, because the values of $f(q)$ at infinitely many $q$ determines the polynomial $f$.
\end{proof}

Next, we unravel the representation $\rho_{q,I}$ as representations of $G$ or $S_n$. For $q=1$, the situation is simpler. The induced module $M_{1,I}=H_1\otimes_{H_{1,I}} \bbC =\bbC S_n \otimes_{\bbC S_I} \bbC$, with $S_I$ acting trivially on $\bbC$. Then the character $\Phi_{1,I}$ afforded by $M_{1,I}$ is exactly $\phi_I=(1_{S_I})^{S_n}$. 

In the case $G=\GL_n(\bbF_q)$, consider the character $(1_{P_I})^G$ afforded by the module $\bbC Ge_{P_I}$, which is denoted by $\Phi_{I,G}$.
Identify $H_q$ with the subalgebra $e_B \bbC Ge_B$ of $\bbC G$. Then our claim is: $\Phi_{I,G}|_{H_q} = \Phi_{q,I}$.

In fact, for $a\in H$, $a \bbC Ge_{P_I}\subseteq e_B \bbC Ge_{P_I}$. Then $\Tr(a_l,\bbC Ge_{P_I})=\Tr(a_l, e_B \bbC Ge_{P_I})$. Note that $B\backslash G/P_I$ corresponds bijectively to  $W/W_I$ (Proof of Proposition \ref{prop:5.2.8}). Therefore, $\calB'=\{a_{r}e_{P_I}:r\in R\}$ is a $\bbC$-basis of $e_B \bbC Ge_{P_I}$. 

One can now consider the element $a_s a_r e_{P_I}$. When $l(sr)<l(r)$, $a_s a_r e_{P_I}=q a_{sr}e_{P_I}+(q-1)a_r e_{P_I}$. When $l(sr)>l(r)$, $a_s a_r e_{P_I}=a_{sr}e_{P_I}$ if $sr\in R$, and $a_s a_r e_{P_I}=a_r a_v e_{P_I}=q a_r e_{P_I}$ if $sr=rv$ for some $v\in I$.
Then, the matrices of left multiplication of $\{a_w:w\in W\}$ on $e_B \bbC Ge_{P_I}$ with respect to $\calB'$ agree with those of $\{T_w:w\in W\}$ on $M_{q,I}$ with respect to $\calB$. Therefore, $\Phi_{I,G}(a_w)=\Tr(a_w, e_B \bbC Ge_{P_I})=\Tr(T_w, M_{q,I})=\Phi_{q,I}(T_w)$.

Note the relation $\Phi_{I,G}=\widehat{\phi_I}$ between the two situations.

We can now rephrase Proposition \ref{prop:8.1.1} in a more convenient way:
\begin{cor}\label{cor:8.1.2}
Let $n\geq 1$ be fixed. Let $G_q=\GL_n(\bbF_q)$.
Let $w_1,...,w_m$ be elements of $W$. Then there exists a unique polynomial $f(u)\in \bbZ[u]$, such that for any prime power $q$, $\Phi_{I,G_q}(a_{w_1}a_{w_2}...a_{w_m})=f(q)$ and that $\phi_I(w_1w_2...w_m)=f(1)$.

More generally, for any virtual character $\phi$ of $S_n$, let $\Phi_q=\widehat{\phi}$ be the image $\phi$ to $G_q$. Then there exists a unique polynomial $g\in \bbZ[u]$ such that $g(q)=\Phi_q(a_{w_1}a_{w_2}...a_{w_m})$ for each prime power $q$ and $g(1)=\phi(w_1w_2...w_m)$.
\end{cor}
\begin{proof}
Under the identification $H_1\simeq \bbC S_n$ via $T_w\mapsto w$, the representation $\Phi_{1,I}$ is identified with $\phi_I$. Then $\Phi_{1,I}(T_{w_1}...T_{w_m})=\phi(w_1...w_m)$. Under the isomorphism $H_q\simeq e_B\bbC G_qe_B$ via $T_w\mapsto a_w$, the representation $\Phi_{q,I}$ is identified with $\Phi_{I,G_q}$. Then $\Phi_{q,I}(T_{w_1}...T_{w_m})=\Phi_{I, G_q}(a_{w_1}...a_{w_m})$. The first statement is then a re-phrasal of Proposition \ref{prop:8.1.1}

Note that $\Phi_{I,G}=\widehat{\phi_I}$. Then the second statement is proved when $\phi=\phi_I$. In general, by Proposition \ref{prop:5.2.5}, a virtual character $\phi$ of $S_n$ is a integer combination of $\phi_I$. Then the second statement holds because the hat map is linear.
\end{proof}

This corollary relates the character values of various $\Phi_q$, and is instrumental in finding these values. We have seen its power in the proof of Theorem \ref{thm:6.3.1}, and will use it to determine the eigenvalues of various relative destination elements in the special cases.

\subsection{Zeta Function for \texorpdfstring{$X_2$}{X1}, Special Case}\label{sec:8.2}
Consider the setting in Section \ref{sec:4.2.2}. Suppose $n=3k$, and $w_1=(1,k+1,2k+1)(2,k+2,2k+2)...(k,2k,3k)\in S_n$. Let $P=P_0=P_\mu$, with $\mu=(k,k,k)$. 

The zeta function $Z_c(X_2^{[(\frac{n}{3},\frac{n}{3})]},u)$ is computed by finding the eigenvalues the right multiplication of the element $D(F_0,F_0,2)$ on $\bbC Ge_{P}$, or the left multiplication of $D(F_0,F_0,2)$ on $e_{P}\bbC G\simeq \oplus_{\lambda\unrhd\mu} {(e_{P_0}M_\lambda)}^{d_\lambda}$. By the calculation in Section \ref{sec:7.2}, right multiplication of $D(F_0,F_0,2)^3=D(F_0,F_0,6)$ acts as the constant $q^{f_\lambda-2\wt_r(\mu)}$ on the space $e_P M_\lambda$. Then the eigenvalues of $D(F_0,F_0,2)$ on $e_P M_\lambda$ are among $\{\omega^{t} q^\frac{f_\lambda-2\wt_r(\mu)}{3}: t=0,1,2\}$, where $\omega=e^{\frac{2\pi i}{3}}$. The dimension of $e_P M_\lambda$ is $K_{\lambda,\mu}$. To decide the multiplicity of each eigenvalue, it suffices to find the trace of this action (We know the total multiplicity, and the real and imaginary parts of the trace will give us 2 more equations). 

Note that $\Tr(D(F_0,F_0,2),e_P M_\lambda)=\Tr(D(F_0,F_0,2), M_\lambda)=\Psi_\lambda(D(F_0,F_0,2))$ because $D(F_0,F_0,2)\in e_{P}\bbC Ge_P$ and annihilates $(1-e)M_\lambda$. 

The idea is to let $q\to 1$, and then $D(F_0,F_0,2)\to e_{S_\mu}w_1e_{S_\mu}$. Then $\Psi_\lambda(D(F_0,F_0,2))$ should become $\psi_\lambda(e_{S_\mu}w_1e_{S_\mu})$. Finding the number $\psi_\lambda(e_{S_\mu}w_1e_{S_\mu})$ will then decide the multiplicity of each $\omega^{t} q^\frac{f_\lambda-2\wt_r(\mu)}{3}$ appearing as eigenvalues of the action $D(F_0,F_0,2)_r$. We now make this process precise.

\paragraph{The $q\to 1$ Deformation Process} We now show the validity of the $q\to 1$ process. For this, we need to express the element $D=D(F_0,F_0,2)$ as an expression involving $a_w$ to apply the Corollary \ref{cor:8.1.2}. Since we will do $q\to 1$ eventually, the various $q$ powers will be irrelevant, and the integer exponents will be denoted by $C_1,C_2,...$ when appropriate.

By Section \ref{sec:4.2.2},  $D=D(F_0,F_0,2)=a_{P_0w_1P_0} =q^{C_1}e_{P_0}w_1e_{P_0}.$
Note that since $a_w=a_{BwB}=\frac{1}{|B|}\sum_{x\in BwB}x$, we have
$$e_{P_0}=\frac{1}{|P_0|} \sum_{x\in P_0}x =\frac{|B|}{|P_0|} \sum_{w\in S_\mu}a_{BwB}=\frac{\sum_{w\in S_\mu}a_{w}}{\sum_{w\in S_\mu}{q^{l(w)}}}.$$
We denote by $a_{\mu}$ the sum $\sum_{w\in S_\mu}a_{w}$. Then the formula $D=q^{C_1}e_{P_0}w_1e_{P_0}$ yields 
$$D=q^{C_1}\frac{a_{\mu} w_1 a_{\mu}}{\ind(a_{\mu})^2},$$ where $\ind$ is the index character $a_w\mapsto q^{l(w)}$ as in Section \ref{sec:8.1}.

By Corollary \ref{cor:8.1.2}, there exists a polynomial $g\in \bbZ[u]$, such that for all prime power $q>1$, the numbers $\Psi_\lambda(a_\mu w_1 a_\mu)$ associated with $G=\GL_n(\bbF_q)$ equals $g(q)$, and that $\psi_\lambda(b_\mu w_1 b_\mu)=g(1)$, where $b_\mu=\sum_{x\in S_\mu}x \in \bbC S_\mu$ is the sum of elements in $S_\mu$.

 Therefore, $\Psi_\lambda(D)=g(q) q^{C_1}/\ind(a_\mu)^2=h(q)$ for some rational function $h(u)\in \bbQ(u)$ because  $\ind(a_\mu)=\sum_{w\in S_\mu}{q^{l(w)}}$. Then $h(1)=g(1)/|S_\mu|^2=\psi_\lambda(b_\mu w_1 b_\mu)/|S_\mu|^2=\psi_{\lambda}(e_{S_\mu}w_1e_{S_\mu})$. We will find the number $h(1)$ in Section \ref{sec:8.3}.
 
\paragraph{Deciding $\Psi_\lambda(D)$}
Recall that $\frac{D(F_0,F_0,2)}{q^{{(f_\lambda-2\wt_r(\mu))}/{3}}}=D q^{-C_2/3}$ acts on $e_P M_\lambda$ with eigenvalues among $1,\omega,\omega^2$, and the total multiplicities are $K_{\lambda,\mu}$. Therefore,  $\Psi_\lambda(D q^{-C_2/3})=\Tr(D q^{-C_2/3},e_P M_\lambda)$ falls in the finite set $S_1=\{a+b\omega+c\omega^2:a,b,c\in \bbZ_{\geq 0}, a+b+c=K_{\lambda,\mu}\}$.

Then for any $q$, there exists some $s\in S_1$ such that  $\Psi_\lambda(D)=s q^{C_2/3}$.
Then for some $s_1$, there exist infinitely many $q$'s such that $\Psi_\lambda(D)=s_1 q^{C_2/3}$  when considering  $G=\GL_n(\bbF_q)$.

On the other hand, $\Psi_\lambda(D)=h(q)$ for all prime power $q$. Then  $h(q)=s_1 q^{C_2/3}$ for infinitely many $q$. Treating both sides as rational functions in $q^{1/3}$ for example, one concludes that $h(u)=s_1 u^{C_2/3}$, and thus $C_2/3$ is an integer or $s_1=0$.

Note that $h(1)=s_1$.  The conclusion is that $\Psi_\lambda (D)=h(q)=h(1) q^{C_2/3}$ for all prime power $q$, where $h(1)=\psi_{\lambda}(e_{S_\mu}w_1e_{S_\mu})$. If we find $h(1)$, we know the multiplicities of the eigenvalues $1,\omega,\omega^2$.

\subsection{Character values  of \texorpdfstring{$e_{S_\mu} u e_{S_\mu}$}{eSmu u eSmu}}\label{sec:8.3}

Let $W=S_n$, with $n=3k$ and $\mu=(k,k,k)$. Let $\lambda =(x,y,z)\vdash n$ with $x\geq y\geq z\geq 0$. Suppose $\lambda\unrhd \mu$. We now calculate $\phi_\lambda(e_\mu u e_\mu)$ and $\psi_\lambda(e_\mu u e_\mu)$, where $u=(1,k+1,2k+1)(2,k+2,2k+2)...(k,2k,3k)$. To simplify notation, we write $e_\nu$ for $e_{S_\nu}$, for any partition $\nu\vdash n$.

Recall that $\phi_\lambda=(1_{S_\lambda})^{S_n}$, and is afforded by $\bbC S_n e_\lambda$. Then 
 $$\phi_\lambda(e_\mu u e_\mu)=\Tr(e_\mu u e_\mu)_l, \bbC S_n e_\lambda)=\Tr((e_\mu u e_\mu)_l,e_\mu \bbC S_n e_\lambda).$$
Now, $e_\mu \bbC Ge_\lambda$ has dimension $|S_\mu \backslash S_n/S_\lambda|$. We characterize double cosets in $S_\mu \backslash S_n/S_\lambda$.

Let $X_1=\llbracket1 ,x\rrbracket $, $X_2=\llbracket x+1, x+y\rrbracket $, $X_3=\llbracket x+y+1..,x+y+z\rrbracket$. Let $I_1=\llbracket1 ,k\rrbracket $, $I_2=\llbracket k+1 ,2k\rrbracket$, and $I_3=\llbracket 2k+1 ,3k\rrbracket$. A double coset $S_\mu wS_\lambda$ is characterized by the nine intesection numbers $(|wX_i\cap I_j|)_{1\leq i,j\leq 3}$:

\begin{lemma}
    Let $w_1,w_2\in S_n$. Then $S_\mu w_1S_\lambda=S_\mu w_2S_\lambda$ exactly if $|w_1 X_i\cap I_j|=|w_2 X_i\cap I_j|$ for all $i,j$.
\end{lemma}
\begin{proof}
Suppose that $S_\mu w_1 S_\lambda = S_\mu w_2 S_\lambda$. Then $w_1= s_1 w_2 s_2$ for some $s_1\in S_\mu, s_2\in S_\lambda$. 

Note that for any $s_1$ preserves each $I_j$, and $s_2$ preserves each $X_i$. Then $|w_1X_i\cap I_j|=|s_1 w_2 s_2 X_i\cap I_j|=| w_2 s_2 X_i\cap s_1^{-1} I_j|=|w_2 X_i\cap  I_j|$. Therefore, the intersection numbers $|w_1 X_i \cap I_j|$ and $|w_2 X_i \cap I_j|$ are equal.

Conversely, suppose the intersection numbers are equal. Then there exists some $g\in S_n$ such that $g(w_1 X_i \cap I_j)=w_2 X_i \cap I_j$ for all $i,j$. For each $j$, taking the union of both sides of the equation $gw_1 X_i \cap gI_j=w_2 X_i \cap I_j$ gives $g I_j= I_j$. Therefore, $g\in S_\mu$. Then $w_2 X_i \cap I_j=gw_1 X_i\cap I_j$. For each $i$, taking the union of both sides of $w_2 X_i \cap I_j=gw_1 X_i\cap I_j$ over $j=1,2,3$  yields $gw_1 X_i=w_2X_i$. Then $w_2^{-1}g w_1\in S_\lambda$. Therefore, $S_\mu w_1S_\lambda=S_\mu w_2S_\lambda$.
\end{proof}

Now, we describe the left action of $e_\mu w_1 e_\mu$ on the basis elements $e_\mu w_1e_\lambda$ for $w_1\in S_n$:
\begin{prop}\label{prop:8.3.2}
    $e_\mu u e_\mu e_\mu w_1 e_\lambda= e_\mu u w_1  e_\lambda$. Moreover, $e_\mu u w_1  e_\lambda= e_\mu w_1  e_\lambda $ if and only if $|w_1X_i\cap I_1|=|w_1X_i\cap I_2|=|w_1X_i\cap I_3|$ for $i=1,2,3$.
\end{prop}

\begin{proof}
By definition, $e_\mu= \frac{1}{|S_\mu|}\sum_{x\in S_\mu} x$, and $e_\lambda= \frac{1}{|S_\lambda|}\sum_{y\in S_\lambda} y$. Note that $e_\mu$ is idempotent. Then $$e_\mu u e_\mu e_\mu w_1 e_\lambda=e_\mu u e_\mu w_1 e_\lambda=\frac{1}{|S_\lambda||S_\mu|^2}\sum_{x,y\in S_\mu, z\in S_\lambda} xuyw_1z. \quad (*)$$

Note that $x,y\in S_\lambda$ preserves the three subsets $I_1,I_2,I_3$, while $u(I_1)=I_2,u(I_2)=I_3$, and $u(I_3)=I_1$. Then $xuy$ also rotates the three sets $I_1,I_2,I_3$ by $I_1\mapsto I_2\mapsto I_3\mapsto I_1$.  
Now $|xuyw_1z(X_i\cap I_j)|=|w_1zX_i\cap (xuy)^{-1}(I_j)|=|w_1 X_i\cap u^{-1}I_j|=|uw_1X_i\cap I_j|$.

Then the intersection numbers $|xuyw_1z(X_i\cap I_j)|$ and $|uw_1X_i\cap I_j|$ all agree. Therefore, $xuyw_1z\in S_\mu u S_\lambda$. The sum $(*)$ is taken over elements in $S_\mu u S_\lambda$. Because the sum is in $e_\mu \bbC G e_\lambda$, all the coefficients $c_a$ and $c_b$ are equal when $a,b$ are in the same $S_\mu-S_\lambda$ double coset. Further, one notes that the total coefficient in $(*)$ is 1. Therefore, $e_\mu u e_\mu e_\mu w_1 e_\lambda= e_\mu u w_1  e_\lambda$.

Now, $uw$ and $w_1$ are in the same double coset exactly when $|uw_1 X_i\cap I_j|=|w_1 X_i\cap I_j|$ for all $i,j$. Equivalently, $|w_1 X_i\cap I_j|=|w_1 X_i\cap u^{-1}I_j|$. Since $u^{-1}$ maps $I_1\mapsto I_3\mapsto I_2\mapsto I_1$, this condition says for all $i$, $|w_1 X_i\cap I_1|=|w_1 X_i\cap I_3|= |w_1 X_i\cap I_2|$. Therefore,  $e_\mu u w_1  e_\lambda= e_\mu w_1  e_\lambda $ exactly when for all $i$, $|w_1 X_i\cap I_j|$ is uniform for $j=1,2,3$.    
\end{proof}

As a corollary, one can compute the trace of the left multiplication of $e_\mu u e_\mu$ on the space $\bbC G e_\lambda$, or $\phi_\lambda(e_\mu u e_\mu)$.

\begin{cor}\label{cor:8.3.3} Let $\mu=(k,k,k)$, and $\lambda=(x,y,z)$ be partitions of $n=3k>0$, with $x\geq y\geq z\geq 0$. Then $\phi_\lambda(e_\mu u e_\mu)=1$ exactly when $x,y,z$ are all divisible by 3, and $\phi_\lambda(e_\mu u e_\mu)=0$ otherwise.
\end{cor}

\begin{proof}
    Recall that $\phi_\lambda(e_\mu u e_\mu)=\Tr(e_\mu u e_\mu)_l, \bbC S_n e_\lambda)=\Tr((e_\mu u e_\mu)_l,e_\mu \bbC S_n e_\lambda).$ We now consider the trace of $(e_\mu u e_\mu)_l$ on $e_\mu \bbC S_n e_\lambda$.

    Take a basis $\{ e_\mu w_t e_\lambda:t=1,2,...,|S_\mu \backslash S_n/S_\lambda| \}$ of $e_\mu \bbC G e_\lambda$. The element $e_\mu w_t e_\lambda$ is characterized by the numbers $|X_i\cap w_t I_j|$.
    
    Then $e_\mu u e_\mu e_\mu w_t e_\lambda= e_\mu uw_t e_\lambda$, and equals $e_\mu w_t e_\lambda$ exactly when for each $i$, $|w_t X_i\cap I_1|=|w_t X_i\cap I_2|=|w_t X_i\cap I_3|=\frac{|X_i|}{3}$. There is exactly one such $w_t$ when $x=|X_2|, y=|X_2|,z=|X_3|$ are all divisible by $3$. If any one of $x,y,z$ is not divisible by 3, there is no such $w_t$. Therefore, the trace $\Tr((e_\mu u e_\mu)_l,e_\mu \bbC S_n e_\lambda)$ equals 1 when $x,y,z$ are all multiples of 3, and 0 otherwise.
\end{proof}

\begin{prop}\label{prop:8.3.4}
For $\lambda=(x,y,z)$, $\psi_\lambda(e_\mu u e_\mu)=1$ if $y-z =0 \pmod 3$. $\psi_\lambda(e_\mu u e_\mu)=-1$ if $y-z =1 \pmod 3$. $\psi_\lambda(e_\mu u e_\mu)=0$ if $y-z =2 \pmod 3$. 
\end{prop}
Proof. By Proposition \ref{prop:7.4.3},  $\psi_{(x,y,z)} = (\phi_{(x,y,z)} + \phi_{(x+2,y-1,z-1)} + \phi_{(x+1,y+1,z-2)}) - (\phi_{(x,y+1,z-1)} + \phi_{(x+1,y-1,z)} + \phi_{(x+2,y,z-2)})$.  Again, if one of $a,b,c<0$, the expression $\phi_{(a,b,c)}$ is understood as 0. 

For fixed $x\geq y\geq z\geq 1$, note that at most one of the 6 tuples $(x,y,z),(x+2,y-1,z-1),(x+1,y+1,z-2),(x,y+1,z-1),(x+1,y-1,z),(x+2,y,z-2)$ can have all entries in $3\bbZ_{\geq 0}$.

Note that $x+y+z=n$ is divisible by 3, we can summarize the tuple in $(3\bbZ_{\geq 0},3\bbZ_{\geq 0},3\bbZ_{\geq 0})$ for all the possibilities of $(x,y,z) \pmod{3}$ in the following table \ref{table:6}:

\begin{table}[h]
\renewcommand{\arraystretch}{1}
\centering
\begin{tabular}{|c|c|c|c|}
\hline
  & $z = 0 \pmod 3$ & $z = 1 \pmod 3$ & $z = 2 \pmod 3$ \\
\hline
$y = 0 \pmod 3$ & $(x, y, z)$ & None & $(x+2, y, z-2)$ \\
\hline
$y = 1 \pmod 3$ & $(x+1, y-1, z)$ & $(x+2, y-1, z-1)$ & None \\
\hline
$y = 2 \pmod 3$ & None & $(x, y+1, z-1)$ & $(x+1, y+1, z-2)$ \\
\hline
\end{tabular}
\caption{The tuple in $(3\mathbb{Z},3\mathbb{Z}, 3\mathbb{Z})$ for each $(x=-y-z,y,z) \pmod 3$.}\label{table:6}
\end{table}

Now, given $\lambda=(x,y,z)$. Let $\lambda'$ be one of the 6 tuples above. By Corollary 8.4.3, $\phi_{\lambda'}(e_\mu u e_\mu)=1$ when $\lambda'$ is the suitable tuple in the table, and $\phi_{\lambda'}(e_\mu u e_\mu)=0$ otherwise. Then $\psi_{(x,y,z)}(e_\mu u e_\mu)$ equals $1$ if $y=z \pmod 3$, equals $-1$ if $y=z+1 \pmod 3$, and equals $0$ if $y=z+2 \pmod 3$. 

\subsection{Zeta Function for \texorpdfstring{$X_2$}{X1} in the Special Case (Continued)}\label{sec:8.4}

We now continue the calculation of $Z_c(X_2^{[(k,2k)]},u)$ for $n=3k$. 

Suppose all  the eigenvalues of $D(F_0,F_0,2)_l$ on $e_{P_0}\bbC G$ are $q_1,q_2,...q_M$. Then by Proposition \ref{prop:3.3.4 Main prop on group zeta}, 
$$Z_c(X^{[(k,2k)]}_2,u)=\prod_{i=1}^M \frac{1}{(1-q_iu^2)}.$$

Recall that $\mathbb{C} G e_\mu \simeq \bigoplus_{\lambda \unrhd \mu} (e_P M_\lambda)^{d_\lambda}$ as right $H_\mu$-modules. Now, by Section \ref{sec:8.2}, $\frac{D(F_0,F_0,2)}{q^{{(f_\lambda-2\wt_r(\mu))}/{3}}}$ acts on $e_{P_0}\bbC G$ with eigenvalues among $1,\omega,\omega^2$, and the total multiplicities are $K_{\lambda,\mu}$. Moreover, the sum of the 1, $\omega$, $\omega^2$ according to their multiplicities equals $\psi_\lambda(e_\mu ue_\mu)$ and was computed in Proposition \ref{prop:8.3.4}.

We now compute as in Table \ref{Table:7} the eigenvalues of $D(F_0,F_0,2)_r$ on $e_{P_0}\bbC G$, and on $\bbC Ge_\mu$. Each partition $\lambda=(x,y,z)\unrhd (k,k,k)$ contributes the three eigenvalues with the multiplicities listed to $q_i$.

\begin{table}[ht]
\renewcommand{\arraystretch}{1.2}
\begin{center}
\begin{tabular}{|c|c|c|c|}
\hline
$\lambda=(x,y,z)$ & Eigenvalue & Multiplicity on $e_{P_0}M_\lambda$ & Multiplicity on $\bbC G e_\mu$ \\
\hline
& $q^{(f_\lambda-2\wt_r(\mu))/3}$ & $(K_{\lambda,\mu}+2)/3$ & $d_\lambda(K_{\lambda,\mu}+2)/3$ \\
$y\equiv z \pmod 3$ & $\omega q^{(f_\lambda-2\wt_r(\mu))/3}$ & $(K_{\lambda,\mu}-1)/3$ & $d_\lambda(K_{\lambda,\mu}-1)/3$ \\
& $\omega^2 q^{(f_\lambda-2\wt_r(\mu))/3}$ & $(K_{\lambda,\mu}-1)/3$ & $d_\lambda(K_{\lambda,\mu}-1)/3$ \\
\hline
& $q^{(f_\lambda-2\wt_r(\mu))/3}$ & $(K_{\lambda,\mu}-2)/3$ & $d_\lambda(K_{\lambda,\mu}-2)/3$ \\
$y\equiv z+1 \pmod 3$ & $\omega q^{(f_\lambda-2\wt_r(\mu))/3}$ & $(K_{\lambda,\mu}+1)/3$ & $d_\lambda(K_{\lambda,\mu}+1)/3$ \\
& $\omega^2 q^{(f_\lambda-2\wt_r(\mu))/3}$ & $(K_{\lambda,\mu}+1)/3$ & $d_\lambda(K_{\lambda,\mu}+1)/3$ \\
\hline
& $q^{(f_\lambda-2\wt_r(\mu))/3}$ & $K_{\lambda,\mu}/3$ & $d_\lambda K_{\lambda,\mu}/3$\\
$y\equiv z+2 \pmod 3$ & $\omega q^{(f_\lambda-2\wt_r(\mu))/3}$ & $K_{\lambda,\mu}/3$ & $d_\lambda K_{\lambda,\mu}/3$\\
& $\omega^2 q^{(f_\lambda-2\wt_r(\mu))/3}$ & $K_{\lambda,\mu}/3$ & $d_\lambda K_{\lambda,\mu}/3$\\
\hline
\end{tabular}
\end{center}
\caption{Eigenvalues and multiplicities of $D(F_0,F_0,2)_r$}\label{Table:7}
\end{table}

\begin{rem}
Note that all the multiplicities are integers. Therefore, $K_{\lambda,\mu}\equiv y-z+1 \pmod 3$. This also follows from the explicit formula in Section \ref{sec:7.3}: Note that $\mu=(k,k,k)$. Then $K_{\lambda,\mu}=\min(x-k,k-z)-\max(0,k-y,y-k)+1=\min(2k-y-z,k-z)-|y-k|+1$. For $y\geq k$, $K_{\lambda,\mu}=2k-y-z-(y-k)+1=3k-2y-z+1\equiv y-z+1 \pmod 3$. For $y\leq k$, $K_{\lambda,\mu}=k-z-(k-y)+1=y-z+1$.
\end{rem}

Now, we can express the zeta function as:
\begin{equation*}
\frac{1}{Z_c(X_2^{[(k,2k)]},u)} = \prod_{\lambda\unrhd\mu} \prod_{\omega^j} {(1-\omega^j q^{(f_\lambda -2\wt_r(\mu))/3}u^2)^{d_\lambda(K_{\lambda,\mu}+\epsilon_{\lambda,j} )/3}}
\end{equation*}
where $\epsilon_{\lambda,j}=0,\pm2$ for $\omega^j=1$ and $\epsilon_{\lambda,j}=0,\pm1$ otherwise is a correcting term so that $3| K_{\lambda,\mu}+\epsilon_{\lambda,j}$ as in Table \ref{Table:7}.

Note that $(1-q^s\omega u^2)(1-q^s \omega^2 u^2)=(1+q^s u^2+q^{2s}u^4)$. Therefore, the inverse zeta function  $\frac{1}{Z_c(X_2^{[(k,2k)]},u)}$ is a polynomial in $\bbZ[q^{1/3},u]$. One can further show that all the coefficients lie in $\bbZ$. We give two proofs of this statement in the following remark.

\begin{rem}
    One method is using the result in Section \ref{sec:8.3}: If $3\nmid f_\lambda-2\wt_r(\mu)$, then $s=\psi_\lambda(e_\mu ue_\mu) =0$. Therefore, when $y-z\equiv 0 $ or 1$\pmod{3}$, the number $f_\lambda-2\wt_r(\mu)$ is divisible by 3, and the value $q^{(f_\lambda -2\wt_r(\mu))/3}$ is an integer.
    For $y-z\equiv 2\pmod{3}$, the correcting terms $\epsilon_{\lambda,j}$ vanish for each $j$. Then the term contributed by $\lambda$ is $$\prod_{j=0}^2 {(1-\omega^j q^{(f_\lambda -2\wt_r(\mu))/3}u^2)^{d_\lambda K_{\lambda,\mu} /3}}=(1-q^{(f_\lambda -2\wt_r(\mu))}u^6)^{d_\lambda K_{\lambda,\mu} /3},$$
    and is a integer polynomial.

    One may also use explicit formula at the end of Section \ref{sec:7.2}: $f_\lambda-2\wt_r(\mu)=\binom{3k}{2}+\binom{x}{2}+\binom{y}{2}+\binom{z}{2}-2z-y+2ij+2jk+2ik-n(n-1)=\binom{n}{2}+\binom{x}{2}+\binom{y}{2}+\binom{z}{2}-2z-y+6k^2-3k(3k-1)\equiv \binom{x}{2}+\binom{y}{2}+\binom{z}{2}+z-y \pmod 3$. Make a table like Table 6, one can see that $f_\lambda-2\wt_r(\mu)\equiv 0 \pmod{3}$ for $y-z\equiv 0 $ or 1$\pmod{3}$, and $f_\lambda-2\wt_r(\mu)\equiv 2 \pmod{3}$ for $y-z\equiv 2\pmod{3}$. The remaining argument is the same.
\end{rem}

Hence, we have proved:
\begin{thm}\label{thm:8.4.3}
    When $n=3k$, the inverse zeta function $\frac{1}{Z_c(X_2^{[(k,2k)]},u)}$ can be expressed as a product of factors of the form $(1-q^su^2)$ or $(1-\omega q^su^2)(1-\omega^2 q^su^2)$ or $(1-q^s u^6)$, where $s\in \mathbb{Z}_{\geq 0}$ and $\omega=e^{\frac{2\pi i}{3}}$.
\end{thm}

\subsection{Zeta Function for \texorpdfstring{$X_0$}{X0} in the Special Case}\label{sec:8.5}

Now we consider the special case where $n=2k$ is even, and compute the zeta function $Z_c(X_0^{[k]},u)$ explicitly.

Let $W_0=\langle e_1,...,e_k\rangle$. Then $P_0=\text{Stab}(W_0)$ corresponds to the partition $\mu=(k,k)$ of $n$.  To compute the zeta function $Z_c(X_0^{[k]},u)$, we need to understand the eigenvalues of $D(W_0,W_0,1)_r$ on $\bbC Ge_{P_0}$, or $D(W_0,W_0,1)_l$ on $e_{P_0}\bbC G$. 

 By Section \ref{sec:special_caseDW4.1.3}, $D(W_0,W_0,1)=a_{P_0w_0P_0}=\frac{|P_0w_0P_0|}{|P_0|}e_{P_0}w_0e_{P_0}=q^{k^2}e_{P_0}w_0e_{P_0}$, where $w_0=(1,n)(2,n-1)...(k,k+1)$ is the longest element in $S_n$.

By Section \ref{sec:6.1}, $e_{P_0}\bbC G\simeq \bigoplus_{\lambda\unrhd\mu}(e_{P_0}M_\lambda)^{d_\lambda}$ as right $H_\mu=e_{P_0}\bbC Ge_{P_0}$-modules. We thus need to compute $(e_{P_0}w_0e_{P_0})_l$ on the individual spaces $e_{P_0}M_\lambda$. Note that these individual spaces are 1-dimensional by Table \ref{Table:2}.
By Section \ref{sec:7.1}, the element $D(W_0,W_0,2)=D(W_0,W_0,1)^2$ acts on $e_{P_0}M_\lambda$ by the scalar $q^{f_\lambda-2\wt_r(\mu)}$, so $D(W_0,W_0,1)$ acts by the scalar $\pm q^{(f_\lambda-2\wt_r(\mu))/2}$, and $e_{P_0}w_0e_{P_0}$ acts by $\pm q^{C/2}$ for some integer $C$. It suffices to decide the sign by deformation theory. Again, we first relate the trace with certain character values.

\paragraph{Relating With Character Value $\Psi_\lambda$}
As in Section \ref{sec:8.2},
\begin{align*}
 \Psi_{\lambda}(e_{P_0 w_0 P_0}) &= \Tr((e_{P_0 w_0 P_0})_l, e_{P_0}M_\lambda ) \\
 &= \pm q^{C/2}.
\end{align*}
To determine the sign, we again use the Deformation Theory.

Note that $$e_{P_0}= \frac{1}{|P_0|}\sum_{x\in P_0}x = \frac{|B_0|}{|P_0|}\sum_{w\in S_\mu}a_w =\frac{a_\mu}{\ind(a_\mu)},$$ where $a_\mu=\sum_{w\in S_\mu}a_w$. The element in $\bbC S_n$ corresponding to $e_{P_0}$ is $ \frac{1}{|S_\mu|}(\sum_{w\in S_\mu}w)$, which is $e_{S_\mu}$, also denoted by $e_\mu$. Then the element $e_{P_0}w_0 e_{P_0}$ corresponds to the element $e_\mu w_0 e_\mu$ in $\bbC S_n$. The value of $\Psi_\lambda(e_{P_0} w_0 e_{P_0})$ then specializes to  $\psi_\lambda(e_{\mu} w_0 e_{\mu})$ via $q\mapsto 1$.

\paragraph{Results on $\psi_\lambda$}
We state the result of $\phi_\lambda(e_\mu w_0 e_\mu)$ and $\psi_\lambda(e_\mu w_0 e_\mu)$, as in Section \ref{sec:8.3}.
\begin{prop}
    Let $\mu=(k,k)$, and $\lambda=(x,y)$ be partitions of $n=2k>0$, where $x\geq y\geq 0$. Then $\phi_\lambda(e_\mu w_0 e_\mu)=1$ exactly when $x,y$ are both divisible by 2, and  $\phi_\lambda(e_\mu w_0 e_\mu)=0$ otherwise.
\end{prop}
\begin{proof}
    Follow the argument of Proposition \ref{prop:8.3.2}-\ref{cor:8.3.3}.
\end{proof}

\begin{prop}
    Let $\mu=(k,k)$, and $\lambda=(x,y)$ be as before. Then $\psi_\lambda(e_\mu w_0 e_\mu)=(-1)^x$.
\end{prop}
\begin{proof}
By Section \ref{sec:7.3}, $\psi_{(x,y)}=\phi_{(x,y)}-\phi_{(x+1,y-1)}$. Again we treat $\phi_{(n+1,-1)}$ as 0. 

If $x$ is even, $y=n-x$ is also even. Then 
$$\psi_{(x,y)}(e_\mu w_0 e_\mu)=\phi_{(x,y)}(e_\mu w_0 e_\mu)-\phi_{(x+1,y-1)}(e_\mu w_0 e_\mu)=1-0=1.$$
If $x$ is odd, Then $\psi_{(x,y)}(e_\mu w_0 e_\mu)=\phi_{(x,y)}(e_\mu w_0 e_\mu)-\phi_{(x+1,y-1)}(e_\mu w_0 e_\mu)=0-1=-1$. Therefore, $\psi_\lambda(e_\mu w_0 e_\mu)=(-1)^x$.
\end{proof}

\paragraph{Eigenvalues of $D(W_0,W_0,1)_r$}
By Deformation theory, the sign of $\Psi_{\lambda}(e_{P_0 w_0 P_0})$ is $(-1)^x$ for $\lambda=(x,n-x)$. Hence, the eigenvalues of $D(W_0,W_0,1)_l$ on $e_{P_0}\bbC G$ are $(-1)^x q^{(f_\lambda-2\wt_r(\mu))/2}$  with multiplicity $d_\lambda K_{\lambda,\mu}=d_\lambda$ for each $\lambda=(x,n-x)$, where $k \leq x\leq n$. We may also write $\lambda=(n-j,j)$ as in Section \ref{sec:7.1}. Then $(-1)^x=(-1)^j$, and $d_\lambda, f_\lambda$ can be denoted as $d_j,f_j$, with values summarized in Table \ref{Table:2}.

We can plug in these eigenvalues to Proposition \ref{prop:3.3.4 Main prop on group zeta} and Case 2 of Section \ref{sec:3.4} and compute the zeta function as:
\begin{align*}
Z_c(X_0^{[k]},u)
&= \prod_{j=0}^k \frac{1}{\left(1-(-1)^j q^{(f_j-2\wt_r(\mu))/2}u\right)^{d_j}}= \prod_{j=0}^k \frac{1}{\left(1-(-1)^j q^{k^2-kj+\frac{j^2-j}{2}}u\right)^{d_j}}.
\end{align*}
Note that $(f_j-2\wt_r(\mu))/2=k^2-kj+\frac{j^2-j}{2}$ is a positive integer. This proves the final case regarding the properties of zeta functions of components:

\begin{thm}\label{thm:8.5.3}
When $n=2k$, $\frac{1}{Z_c(X_0^{[k]},u)}$ is the inverse zeta function of a subgraph of $X_0$, and it can be expressed as a product of factors of the form $(1\pm q^su)$, where $s\in \mathbb{Z}_{\geq 0}$.
\end{thm}

\section{Conclusion}
We can now prove the main theorem \ref{thm:main thm on graphs 1.3.11} of this paper.
\begin{thm}[Main theorem on eigenvalues of graphs]
The nonzero eigenvalues of the digraphs $X_0(\bbF_q^n)$ and $X_2(\bbF_q^n)$ are a root of unity times a fractional power of $q$.
\end{thm}
\begin{proof}
Let $V=\mathbb{F}^n_q$. For the graph $X_0(V)$:

By Theorem \ref{thm:7.1.1}, the inverse zeta function $\frac{1}{Z_c(X_0^{[k]}(V),u)}$ is a product of factors of the form $(1-q^su^2)$ with $s$ an integer for $k\neq \frac{n}{2}$, and the nonzero eigenvalues of $X_0^{[k]}(V)$ are then $\pm q^{s/2}$. This is also true for $k=\frac{n}{2}$ by Theorem \ref{thm:8.5.3}.

Given that $X_0(V)$ is the disjoint union of $X_0^{[k]}(V)$ across all equivalence classes $[k]$, the nonzero eigenvalues of $X_0(V)$ combine those of each $X_0^{[k]}(V)$, which are given by $\pm q^{s/2}$. Therefore, they also follow the same form, as desired.

For the digraph $X_2(V)$:

Analogous to the $X_0(V)$ case, Theorem \ref{thm:7.2.1} informs us that the inverse zeta function $\frac{1}{Z_c(X_2^{[(i,j)]}(V),u)}$ is a product of factors in the form $(1-q^su^6)$ with $s$ a positive integer for $[(i,j)]\neq \left\{\left(\frac{n}{3},\frac{2n}{3}\right),\left(\frac{2n}{3},\frac{n}{3}\right)\right\}$. This leads to the nonzero eigenvalues of $X_2^{[(i,j)]}(V)$ being $q^{\frac{s}{6}}$ times a 6th root of unity. The same is also valid for $[(i,j)]=\left\{\left(\frac{n}{3},\frac{2n}{3}\right),\left(\frac{2n}{3},\frac{n}{3}\right)\right\}$ when $n=3k$, as per Theorem \ref{thm:8.4.3}.

Given that $X_2(V)$ is the disjoint union of $X_2^{[(i,j)]}(V)$ across all equivalence classes $[(i,j)]$, the nonzero eigenvalues of $X_2(V)$ combine those of each $X_2^{[(i,j)]}(V)$, and adhere to the form $q^{s/6}$ times a 6th root of unity, as desired.
\end{proof}
\paragraph{Summary of Results}
In this paper, we have computed explicit formulas for the zeta functions of the digraphs $X_0(\bbF_q^n)$ and $X_2(\bbF_q^n)$ by analyzing their connectivity and leveraging representation theory tools.

Our main results, Theorems \ref{thm:7.1.1}, \ref{thm:7.2.1}, \ref{thm:8.4.3} and \ref{thm:8.5.3}, provide formulas for the inverse zeta functions of the component subgraphs of $X_0$ and $X_2$. By taking products over the components, this gives formulas for $Z_c(X_0(V), u)$ and $Z_c(X_2(V), u)$ for any finite-dimensional vector space $V$ over $\bbF_q$.

Moreover, by the relation $Z(\mathcal{B}(V), u) = Z_c(X_2(V), u)$ and the formula in Corollary \ref{cor:1.3.10 on zeta X0 and XV1V2} for $Z(\mathcal{B}(V_1,...,V_r), u)$,
our results yield formulas for the edge zeta functions $Z(\mathcal{B}(V), u)$ and $Z(\mathcal{B}(V_1,...,V_r), u)$ of the buildings $\mathcal{B}(V)$ and $\mathcal{B}(V_1,...,V_r)$.

Our approach demonstrates an intriguing connection between representations of finite groups of Lie type and the combinatorial study of buildings associated to these groups. The relative destination elements we introduced leverage the rich representation theory of $\GL_n(\bbF_q)$ to capture connectivity in the buildings.

\paragraph{Further Directions}
The buildings studied here are of type $\mathbf{A}_n$ and $\mathbf{A}_m \times \mathbf{A}_n$. An interesting direction for future research is to extend the analysis to buildings of other Lie types over finite fields, such as types $\mathbf{B}$, $\mathbf{C}$ or $\mathbf{D}$ or exceptional types. One could start by considering buildings of type $\mathbf{C} \times \mathbf{C}$ and type $\mathbf{C}$ associated with symplectic groups, as the $\mathbf{C}$ type case tends to be simpler than $\mathbf{B}$ or $\mathbf{D}$ type.

One may also explore higher zeta functions for buildings, in contrast to the edge zeta functions studied here. However, this would require suitably defining higher-dimensional analogues of geodesics in buildings.
\printbibliography[heading=bibintoc,title={Bibliography}]
\end{document}